\documentclass{amsart}

\usepackage{amssymb}
\usepackage[all]{xy}
\usepackage{mathrsfs}
\usepackage{upgreek}
\usepackage{stmaryrd}
\usepackage{rotating}
\usepackage{tikz}
\usetikzlibrary{cd}

\usepackage[T1]{fontenc}

\usepackage{color}
\usepackage[pagebackref,hyperindex,linktocpage=true]{hyperref}
\hypersetup{
    colorlinks,
    linkcolor={red!50!black},
    citecolor={blue!50!black},
    urlcolor={blue!50!black},
    filecolor={red!50!black} 
}

\newcommand{\Dbc}{D^{\mathrm{b}}_{\mathrm{c}}}
\newcommand{\Dc}{D_{\mathrm{c}}}

\newcommand{\sH}{\mathsf{H}}
\newcommand{\Mod}{\mathsf{Mod}}
\newcommand{\gMod}{\mathsf{gMod}}
\newcommand{\gproj}{\mathsf{gproj}}
\newcommand{\comod}{\mathsf{comod}}
\newcommand{\modf}{\mathsf{mod}}

\newcommand{\cH}{\mathcal{H}}

\newcommand{\cM}{\mathcal{M}}

\newcommand{\bk}{\Bbbk}
\newcommand{\Z}{\mathbb{Z}}
\newcommand{\R}{\mathbb{R}}
\newcommand{\C}{\mathbb{C}}

\newcommand{\Q}{\mathbb{Q}}
\newcommand{\F}{\mathbb{F}}
\newcommand{\scO}{\mathscr{O}}

\newcommand{\Gm}{\mathbb{G}_{\mathrm{m}}}

\newcommand{\cE}{\mathcal{E}}
\newcommand{\cG}{\mathcal{G}}

\newcommand{\cU}{\mathcal{U}}
\newcommand{\cB}{\mathcal{B}}

\newcommand{\IC}{\mathscr{I} \hspace{-1pt} \mathscr{C}}

\newcommand{\cT}{\mathcal{T}}
\newcommand{\cP}{\mathcal{P}}

\newcommand{\Gr}{\mathrm{Gr}}

\newcommand{\sph}{{\mathrm{sph}}}

\newcommand{\Perv}{\mathsf{Perv}}

\newcommand{\Rep}{\mathsf{Rep}}

\newcommand{\Db}{D^{\mathrm{b}}}

\newcommand{\sHom}{\mathscr{H}\hspace{-2pt} \mathit{om}}

\DeclareMathOperator{\Hom}{Hom}
\DeclareMathOperator{\Ext}{Ext}
\DeclareMathOperator{\End}{End}

\DeclareMathOperator{\pr}{pr}

\DeclareMathOperator{\colim}{colim}
\DeclareMathOperator{\Tor}{Tor}

\newcommand{\id}{\mathrm{id}}
\newcommand{\simto}{\xrightarrow{\sim}}
\newcommand{\la}{\langle}
\newcommand{\ra}{\rangle}

\newcommand{\coH}{\sH}

\newcommand{\bK}{\mathbb{K}}
\newcommand{\bO}{\mathbb{O}}
\newcommand{\bF}{\mathbb{F}}

\newcommand{\scF}{\mathscr{F}}
\newcommand{\scP}{\mathscr{P}}
\newcommand{\scG}{\mathscr{G}}
\newcommand{\scJ}{\mathscr{J}}
\newcommand{\res}{\mathrm{res}}

\newcommand{\pH}{{}^{\mathrm{p}} \hspace{-1pt} \mathscr{H}}

\newcommand{\bbX}{\mathbb{X}}

\makeatletter
\def\lotimes{\@ifnextchar_{\@lotimessub}{\@lotimesnosub}}
\def\@lotimessub_#1{\mathchoice{\mathbin{\mathop{\otimes}^L}_{#1}}%
  {\otimes^L_{#1}}{\otimes^L_{#1}}{\otimes^L_{#1}}}
\def\@lotimesnosub{\mathbin{\mathop{\otimes}^L}}
\makeatother

\makeatletter
\def\lboxtimes{\@ifnextchar_{\@lboxtimessub}{\@lboxtimesnosub}}
\def\@lboxtimessub_#1{\mathchoice{\mathbin{\mathop{\boxtimes}^L}_{#1}}%
  {\boxtimes^L_{#1}}{\boxtimes^L_{#1}}{\boxtimes^L_{#1}}}
\def\@lboxtimesnosub{\mathbin{\mathop{\boxtimes}^L}}
\makeatother

\newcommand{\wttimes}{\mathbin{\widetilde{\times}}}

\newcommand{\pot}[1]{ [\hspace{-0,5mm}[ {#1} ]\hspace{-0,5mm}] }
\newcommand{\rpot}[1]{ (\hspace{-0,7mm}( {#1} )\hspace{-0,7mm}) }
\newcommand{\Spec}{\mathrm{Spec}}
\newcommand{\Diag}{\mathrm{D}}
\mathchardef\mhyphen="2D
\newcommand{\Loop}{\mathrm{L}}
\newcommand{\Lie}{\mathscr{L}\hspace{-1.5pt}\mathit{ie}}

\newcommand{\fa}{\mathfrak{a}}
\newcommand{\cA}{\mathcal{A}}

\newcommand{\ucG}{\underline{\cG}}

\newcommand{\et}{\mathrm{\acute{e}t}}
\newcommand{\fppf}{\mathrm{fppf}}

\newcommand{\Hk}{\mathrm{Hk}}

\newcommand{\Conv}{\mathrm{Conv}}
\newcommand{\HkConv}{\mathrm{HkConv}}
\newcommand{\pD}{{}^{\mathrm{p}} \hspace{-1pt} D}
\newcommand{\sF}{\mathsf{F}}
\newcommand{\sFeq}{\mathsf{F}^{\mathrm{eq}}}

\newcommand{\CT}{\mathsf{CT}}
\newcommand{\abs}{\mathrm{abs}}

\newcommand{\Gammac}{\Gamma_{\mathrm{c}}}
\newcommand{\corr}{\mathrm{corr}}
\newcommand{\Cent}{\mathrm{Z}}
\newcommand{\ad}{\mathrm{ad}}

\newcommand{\rmA}{\mathrm{A}}
\newcommand{\rmB}{\mathrm{B}}
\newcommand{\rmS}{\mathrm{S}}
\newcommand{\rmT}{\mathrm{T}}

\numberwithin{equation}{section}
\newtheorem{thm}{Theorem}[section]
\newtheorem{lem}[thm]{Lemma}
\newtheorem{prop}[thm]{Proposition}
\newtheorem{cor}[thm]{Corollary}

\theoremstyle{definition}

\theoremstyle{remark}
\newtheorem{rmk}[thm]{Remark}

\title{A modular ramified geometric Satake equivalence}

\author{Pramod N. Achar}
\address{Department of Mathematics\\
  Louisiana State University\\
  Baton Rouge, LA 70803\\
  U.S.A}
\email{pramod.achar@math.lsu.edu}

\author{Jo{\~a}o Louren{\c c}o}
\address{Universität Münster, Einsteinstrasse 62, Münster, Germany}
\email{j.lourenco@uni-muenster.de}

\author{Timo Richarz}
\address{Technische Universit\"at Darmstadt, Department of Mathematics, 64289 Darmstadt, Germany}
\email{richarz@mathematik.tu-darmstadt.de}

\author{Simon Riche}
\address{Universit\'e Clermont Auvergne, CNRS, LMBP, F-63000 Clermont-Ferrand, Fran\-ce}
\email{simon.riche@uca.fr}

\thanks{P.A. was supported by NSF Grant Nos.~DMS-1802241 and DMS-2202012. J.L. was supported by the Max-Planck-Institut für Mathematik, the Excellence Cluster of the Universität Münster, and by the ERC Consolidator Grant 770936 via Eva Viehmann. This project has received
funding from the European Research Council (ERC) under the European Union's Horizon 2020
research and innovation programme (S.R. \& T.R., grant agreement No.~101002592).
Also, T.R.~acknowledges support by the Deutsche Forschungsgemeinschaft (DFG, German Research Foundation) TRR 326 \textit{Geometry and Arithmetic of Uniformized Structures}, project number 444845124 and the LOEWE professorship in Algebra.
}

\begin{document}

\begin{abstract}
We extend the ramified geometric Satake equivalence due to Zhu (for tamely ramified groups) and the third named author (in full generality) from rational coefficients to include modular and integral coefficients.
\end{abstract}

\maketitle
\setcounter{tocdepth}{1}
\tableofcontents

\section{Introduction}

\subsection{Ramified geometric Satake equivalence}

The geometric Satake equivalence, first fully established by Mirkovi{\'c}--Vilonen~\cite{mv} after important contributions of Lusztig~\cite{lusztig}, Ginzburg~\cite{ginzburg} and Be{\u\i}linson--Drinfeld~\cite{beilinson-drinfeld}, has now become a cornerstone of geometric approaches to a variety of subjects in representation theory and number theory. It consists of an equivalence of monoidal categories relating the category of perverse sheaves on the affine Grassmannian of a reductive algebraic group to representations of the Langlands dual reductive group.
This construction has many variants, in which one e.g.~changes the coefficients of the sheaf theory, or the field of definition of the geometric objects. See, for instance,~Cass--van den Hove--Scholbach~\cite{cvdhs} and the references cited therein, as well as Table~\ref{tab:summary}.

The present paper is a contribution to another variant of this story, initially developed by Zhu~\cite{zhu}, where one replaces the (split) reductive algebraic group over the base field $\F$ (of characteristic $p$) of which one takes the affine Grassmannian by a \emph{possibly nonsplit} reductive group over $\F\rpot{t}$ and a parahoric model over $\F\pot{t}$ attached to a special facet. 
For $\ell$-adic coefficients with $\ell\neq p$, a version of the Satake equivalence in this setting was obtained in \cite{zhu} assuming the reductive group splits over a tame extension, and then by the third author~\cite{richarz} in full generality. 
Here we extend these constructions to \textit{modular and integral coefficients}, i.e.~to categories of perverse sheaves with coefficients in a finite field of characteristic $\ell$ or the ring of integers of a finite extension of $\Q_\ell$ with $\ell\neq p$.  
(See also Remark~\ref{rmk:analytic} regarding the analytic setting and Remark~\ref{rmk:motivic} for an extension to the motivic setting.)

\begin{table}
\begin{center}
\footnotesize
\begin{tabular}{lllll}
\textit{Name} & \textit{Input group} & \textit{Coefficients} & \textit{Tannakian side} & \textit{Reference} \\
\hline
absolute & $G$ split / $\C$ & comm.~ring $\bk$ & $\Rep(G^\vee_\bk)$ & \cite{mv} \\
motivic & $G$ split / scheme $S$ & motives & $\Rep_{G^\vee}(\mathrm{MTM}(S))$ & \cite{cvdhs} \\
tamely ramified & $\cG$ spec.~parahoric / $\F\pot{t}$ & $\Q_\ell$ & $\Rep((G^\vee_{\Q_\ell})^I)$ &\cite{zhu} \\
ramified & $\cG$ spec.~parahoric / $\F\pot{t}$ & $\Q_\ell$ & $\Rep((G^\vee_{\Q_\ell})^I)$ &\cite{richarz} \\
mod.~ramified & $\cG$ spec.~parahoric / $\F\pot{t}$ & $\Lambda = \Q_\ell, \Z_\ell, \F_\ell$ & $\Rep((G^\vee_\Lambda)^I)$ & this paper \\
\end{tabular}
\end{center}
\bigskip
\caption{Some variants of the geometric Satake equivalence}\label{tab:summary}
\end{table}
\subsection{Ramified affine Grassmannians}

Quite generally, one can associate a positive loop group $\Loop^+ \cH$ and an affine Grassmannian (sometimes called an \emph{affine flag variety}) to any smooth affine group scheme $\cH$ over $\F\pot{t}$. 
(Here $\F$ is a separably closed field; for this part of the discussion it could be an arbitrary field.) Of particular interest is the case when $\cH$ is a parahoric integral model of a connected reductive algebraic group over $\F\rpot{t}$ attached to a facet in the associated Bruhat--Tits building, see in particular~Pappas--Rapoport \cite{pr}. 
This class is however too vast to expect a geometric Satake equivalence (e.g.~since it contains \emph{all} partial affine flag varieties attached to split groups). A geometric study of the cases where some of the important properties of the affine Grassmannians arising in~\cite{mv} (in particular, regarding parity of dimensions of orbits) was undertaken in~\cite{richarz}; it turns out that a nice class of cases is that when the facet involved is \emph{special}. 
This case covers the setting considered in~\cite{mv} (corresponding to hyperspecial facets of split groups), and also those arising in~\cite{zhu}.
More geometric properties of such {\it twisted} affine Grassmannians (in particular, regarding semi-infinite orbits) were later established in Ansch\"utz--Gleason--Lourenço--Richarz~\cite{aglr}.

For twisted affine Grassmannians, and for $\ell$-adic coefficients, the third author established an equivalence generalizing those of~\cite{mv} and~\cite{zhu}, relating the category of equivariant perverse sheaves on the affine Grassmannian, equipped with the convolution monoidal structure, to the category of representations of a certain algebraic group, described as follows. 
The reductive group $G$ over $F=\F\rpot{t}$ involved in the constructions splits over the separable closure $F^{\mathrm{s}}$ of $F$. 
One can then consider the pinned group $G^\vee_{\Q_\ell}$ over $\Q_\ell$ which is Langlands dual to this split group. 
Since the dual group arises from a group over $\F\rpot{t}$, it is equipped with an action of the Galois group $I$ of $F^{\mathrm{s}}$ over $F$ preserving the pinning. The Tannakian group considered above is then identified with the fixed point subgroup $(G^\vee_{\Q_\ell})^I$. 
This group is a possibly disconnected algebraic group over $\Q_\ell$, whose neutral component is reductive.

\begin{rmk}
It is quite remarkable that the Tannakian group above only depends on $G$, and not on the choice of special facet. In fact there are reductive groups that admit essentially different special facets; the associated affine Grassmannians are not isomorphic (and, in fact, have quite different geometric properties), yet the associated categories of perverse sheaves are equivalent. See~\cite[p.~411]{zhu} for more detailed comments on this phenomenon. The same phenomenon occurs for the categories with integral or modular coefficients considered below.
\end{rmk}

\subsection{Fixed points of groups of pinning-preserving automorphisms of reductive group schemes}
\label{ss:intro-fixed-pts}

The equivalence of~\cite{mv} (considered here in the setting of \'etale shea\-ves) has a version where the field of coefficients is replaced by a ring $\Lambda$ which is either a finite field of characteristic $\ell \neq p$ or the ring of integers of a finite extension of $\Q_\ell$. It therefore seemed reasonable to expect a generalization of the equivalences of~\cite{zhu, richarz} for such coefficients, involving the group scheme $(G^\vee_\Lambda)^I$ of fixed points of the action as above on the dual split reductive group scheme $G^\vee_\Lambda$ over the given ring of coefficients $\Lambda$.

As a first step towards this goal, in the companion paper~\cite{alrr} we study the group schemes that arise in this way, exploiting crucially the fact that the action of $I$ stabilizes the pinning of $G^\vee_\Lambda$.
We show in particular that these group schemes are always flat over $\Lambda$ (so, their categories of representations are abelian), but not necessarily smooth; in particular, for an appropriate action of $\Z/2\Z$ on $\mathrm{GL}_{2n+1, \Z_2}$ the group of fixed points turns out to be an example of a non-reductive quasi-reductive group scheme in the sense of Prasad--Yu~\cite{py}.

\subsection{Nearby cycles and the relation with the dual group}
\label{ss:intro-nc}

The reason why the dual group $G^\vee_\Lambda$ of $F^{\mathrm{s}} \otimes_{F} G$ occurs in this story is the following. Consider our integral model $\cG$ over $\F\pot{t}$ and the associated affine Grassmannian $\Gr_\cG$, but also the group scheme over the ring $F^{\mathrm{s}}\pot{z}$ (where $z$ is another formal variable) obtained by base change from the split group $F^{\mathrm{s}} \otimes_{F} G$. Then $\Gr_\cG$ can be described as a degeneration of the affine Grassmannian $\Gr_{F^{\mathrm{s}}\pot{z} \otimes_F G}$ (an ind-scheme over $F^{\mathrm{s}}$): there exists an ind-scheme $\Gr_{\cG, \overline{S}}$ over the spectrum $\overline{S}$ of the integral closure of $\F\pot{t}$ in $F^{\mathrm{s}}$ and a diagram with cartesian squares
\begin{equation}\label{eqn:nc-diagram}
\begin{tikzcd}
 \Gr_{F^{\mathrm{s}}\pot{z} \otimes_F G} \ar[hook, r] \ar[d] &  \Gr_{\cG, \overline{S}} \ar[d] & \Gr_\cG \ar[d] \ar[hook', l] \\
 \Spec(F^{\mathrm{s}}) \ar[hookrightarrow, r] & \overline{S} & \Spec(\F). \ar[hook', l]
\end{tikzcd}
\end{equation}
Associated with this diagram we have a nearby cycles functor sending perverse sheaves on $\Gr_{F^{\mathrm{s}}\pot{z} \otimes_F G}$ (which can be described as representations of $G^\vee_\Lambda$ via the ``usual'' geometric Satake equivalence) to perverse sheaves on $\Gr_\cG$. The main result of the paper is an equivalence of monoidal categories
\begin{equation}
\label{eqn:main-equiv-intro}
\Perv_{\Loop^+ \cG}(\Gr_\cG, \Lambda) \cong \Rep((G^\vee_\Lambda)^I),
\end{equation}
between the categories of $\Loop^+\cG$-equivariant $\Lambda$-perverse sheaves on $\Gr_\cG$ and of representations of $(G^\vee_\Lambda)^I$ on finitely generated $\Lambda$-modules, under which the nearby cycles functor considered above corresponds to the functor of restriction along the natural embedding $(G^\vee_\Lambda)^I \subset G^\vee_\Lambda$. 
This is summarized in the following commutative diagram:
\[
\begin{tikzcd}[column sep=9em]
\Perv_{\Loop^+G}(\Gr_{F^{\mathrm{s}}\pot{z} \otimes_F G}, \Lambda) \ar[r, "\sim"', "\text{``usual'' geometric Satake}"] \ar[d, "\text{nearby cycles}"'] &
\Rep(G^\vee_\Lambda) \ar[d, "\text{restriction}"] \\
\Perv_{\Loop^+ \cG}(\Gr_\cG, \Lambda) \ar[r,"\sim"', "\eqref{eqn:main-equiv-intro}"] & \Rep((G^\vee_\Lambda)^I).
\end{tikzcd}
\]
As for the usual geometric Satake equivalence, for a perverse sheaf $\scF$, the underlying $\Lambda$-module of the representation associated with $\scF$ is the total cohomology $\mathsf{H}^\bullet(\Gr_\cG,\scF)$.

\subsection{Comments on the proof}
\label{ss:intro-proof}

The proof of the equivalence~\eqref{eqn:main-equiv-intro} in the $\ell$-adic case is based on a general result of Bezrukavnikov regarding central functors with domain the category of representations of an algebraic group. 
We do have such a functor for arbitrary coefficients thanks to the constructions explained in~\S\ref{ss:intro-nc},
but Bezrukavnikov's result has no counterpart for integral coefficients, and its direct application for positive-characteristic coefficients in our setting presents some difficulties. We therefore follow a different, and in some sense more explicit, route and treat in parallel the cases of a finite extension $\bK$ of $\Q_\ell$ (for which we provide a proof which is different from those in~\cite{zhu,richarz}), its ring of integers $\bO$, and the residue field $\bk$ of $\bO$, using in a crucial way some ``change of scalars'' arguments to transfer information from one case to another.

Part of our constructions are parallel to those used in~\cite{mv}: we consider some ``weight functors'' constructed using semi-infinite orbits (building on geometric facts established in~\cite{aglr}), and show representability of these functors to construct a flat $\Lambda$-bialgebra $\rmB_\cG(\Lambda)$ and an equivalence of monoidal categories between $\Perv_{\Loop^+ \cG}(\Gr_\cG, \Lambda)$ and the category of $\rmB_\cG(\Lambda)$-comodules which are finitely generated over $\Lambda$. 
Further, we show that we have canonical isomorphisms
\begin{equation}
\label{eqn:isom-change-coeff}
\rmB_\cG(\bK) \cong \bK \otimes_{\bO} \rmB_\cG(\bO), \quad \rmB_\cG(\bk) \cong \bk \otimes_{\bO} \rmB_\cG(\bO).
\end{equation}
The nearby cycles functor of~\S\ref{ss:intro-nc} and the Tannakian formalism provide a morphism of $\Lambda$-bialgebras $\scO(G^\vee_\Lambda) \to \rmB_\cG(\Lambda)$, which is easily seen to factor through a morphism
\begin{equation}
\label{eqn:morph-bialgebras}
\scO((G^\vee_\Lambda)^I) \to \rmB_\cG(\Lambda),
\end{equation}
and what remains to be seen is that this morphism is an isomorphism.

Here we note a first important difference with the constructions in~\cite{mv}: in the present setting, we do not know any geometric construction of the commutativity constraint for convolution in $\Perv_{\Loop^+ \cG}(\Gr_\cG, \Lambda)$ that corresponds under~\eqref{eqn:main-equiv-intro} to the obvious commutativity constraint for the tensor product of representations. (This is due to the absence, in this setting, of the Be{\u\i}linson--Drinfeld deformations of the affine Grassmannian over products of curves.) To bypass this difficulty, we prove surjectivity of the morphism~\eqref{eqn:morph-bialgebras} in case $\Lambda=\bK$; this implies in particular that $\rmB_\cG(\bK)$ is commutative. Using~\eqref{eqn:isom-change-coeff} we deduce that $\rmB_\cG(\bO)$ is commutative, and then that $\rmB_\cG(\bk)$ is commutative too.

The other argument for which the Be{\u\i}linson--Drinfeld deformations are crucially used in~\cite{mv} is in the construction of the monoidal structure on the ``fiber functor'' given by total cohomology (which induces the comultiplication in $\rmB_\cG(\Lambda)$). Here we bypass this difficulty in a different way, constructing this monoidal structure by a method which is new even in the setting of~\cite{mv}, and which corrects a pervasive error in the literature on this topic (see Remark~\ref{monoidality-remark}).

Once these properties are known it is not difficult to check that
\[
\cG^\vee_\Lambda := \Spec(\rmB_\cG(\Lambda))
\]
is a group scheme over $\Lambda$, and that~\eqref{eqn:morph-bialgebras} provides a morphism of group schemes $\cG^\vee_\Lambda \to (G^\vee_\Lambda)^I$. To prove that this morphism is an isomorphism we first treat the case of relative rank $1$, and exploit once again the possibility of transferring information between coefficients $\bK$, $\bO$ and $\bk$. We also use a number of properties of the groups $(G^\vee_\Lambda)^I$ proved in~\cite{alrr}, and specific arguments to treat the case when the quasi-reductive group over $\Z_2$ mentioned in~\S\ref{ss:intro-fixed-pts} appears.

\begin{rmk}
We emphasize that $(G^\vee_\Lambda)^I$ is not a reductive group scheme over $\Lambda$ in general. We find it remarkable that such groups can arise as a Tannakian group for a category of perverse sheaves.
\end{rmk}

\begin{rmk}\label{rmk:analytic}
In the case where $\F = \C$, one can also consider $\Gr_\cG(\C)$ with the \emph{analytic} topology, and work with sheaves of $\Lambda$-modules where $\Lambda$ is any unital, commutative, noetherian ring of finite global dimension as in \cite{mv}.  
In particular, in the analytic setting, one can take $\Lambda = \Z$.  
The main results of this paper should remain valid in the analytic setting, with minor modifications to the proofs, but with one caveat:
in this paper, we work with \'etale sheaves on algebraic stacks, and we are not aware of a reference that treats analytic sheaves on stacks in the appropriate generality.  However, this issue can likely be circumvented by working with ``equivariant derived categories in families'' as in~\cite[Chapter~10]{ar-book}.
\end{rmk}

\begin{rmk}\label{rmk:motivic}
Thibaud van den Hove has announced a motivic refinement of the ramified Satake equivalence with coefficients in $\mathbb Z[\frac{1}{p}]$, $\Q$, or $\F_\ell$ for $\ell\neq p$.
\end{rmk}

\subsection{Motivation}

The main motivation for the constructions in~\cite{zhu, richarz} was the application to some properties of Shimura varieties. At this point it does not seem that our integral and modular versions lead to any specific new application in this direction; in fact our desire to establish the equivalence~\eqref{eqn:main-equiv-intro} rather came from representation theory. Namely, in many cases the group $(G^\vee_\bk)^I$ is still a reductive group. It was conjectured by Brundan~\cite{brundan}, and proved by him in most cases, that the restriction to $(G^\vee_\bk)^I$ of any tilting $G^\vee_\bk$-module remains tilting. (The remaining cases were later treated by van der Kallen~\cite{vdk}). 
This proof is based on case-by-case considerations, which in our opinion does not explain the true meaning of this property. 
We hope that the geometric description of the restriction functor in terms of nearby cycles will lead to an alternative (and more satisfactory) proof of this property.

Let us also point out that in the recent update of Fargues--Scholze \cite[\S VIII.5]{fs} a uniform proof of this result, not relying on case-by-case considerations, is given under the assumption that $I$ acts on $G^\vee_\bk$ through of finite quotient of order prime to the characteristic of $\bk$, see \cite[Theorem VIII.5.15]{fs}. 
It would be interesting to see whether the coprimality assumption can be removed in the present setting. 

\subsection{Contents}

In Section~\ref{sec:aff-Grass} we recall (and sometimes complete) some results on the geometry of affine Grassmannians associated with special parahoric models of reductive groups, following~\cite{zhu,richarz}. 
In Section~\ref{sec:semi-infinite} we recall (and, again, sometimes complete) results from~\cite{aglr} and Haines--Richarz \cite{HainesRicharz_TestFunctions} which form the basis for the construction of the weight (or ``constant term'') functors. 
In Section~\ref{sec:sheaves-Hecke} we introduce the category of equivariant sheaves on our affine Grassmannian, and some of the structures we have on this category. 
In Sections~\ref{sec:monoidality} we construct a monoidal structure on the total cohomology functor, and in Section~\ref{sec:construction} we construct the bialgebra $\rmB_\cG(\Lambda)$ from~\S\ref{ss:intro-proof}, and establish some of its basic properties. Section~\ref{sec:monoidality-unramified} links our constructions to the ``absolute'' case studied in~\cite{mv}.  In Section~\ref{sec:nc} we explain the construction of the nearby cycles functor from~\S\ref{ss:intro-nc}, and how to obtain from it the morphism~\eqref{eqn:morph-bialgebras}.
Finally, in Section~\ref{sec:equivalence} we show that this morphism is an isomorphism, which completes the proof of the equivalence~\eqref{eqn:main-equiv-intro}.

The paper finishes with two appendices: in Appendix~\ref{app:equiv-const} we prove that $\Loop^+ \cG$-equivariance is automatic for perverse sheaves on $\Gr_\cG$ which are constant along $\Loop^+ \cG$-orbits, and in Appendix~\ref{app:sheaves} we collect the results on \'etale sheaves on stacks that are used in the main text.


\subsection{Acknowledgements}

We thank Julien Bichon and Victor Ostrik for answering some questions on Hopf algebras, Tom Haines for providing the argument in Lemma~\ref{lem:dominant-coweights}\eqref{it:dominant-coweights-lifting}, Thibaud van den Hove for helpful comments on a prelimary version of this manuscript, Peter Scholze for discussions surrounding Brundan's conjecture, Jize Yu for discussions surrounding~\cite{yu-integral}, and Weizhe Zheng for answering some questions about~\cite{liu-zheng,liu-zheng-2}.

\section{Affine Grassmannians associated with special facets}
\label{sec:aff-Grass}

\subsection{Loop groups and affine Grassmannians}
\label{ss:loop-gps}

Let $k$ be a field and $x$ be an indeterminate. 
Given a $k$-algebra $R$ we consider the rings $R\pot{x}$ and $R\rpot{x}$ of power and Laurent power series in $x$ with coefficients in $R$ respectively. Given an affine group scheme $\cH$ over $k \pot{x}$, resp.~an affine group scheme $H$ over $k \rpot{x}$, we can define the positive loop group $\Loop^+ \cH$, resp.~the loop group $\Loop H$, as the functor from $k$-algebras to groups defined by
\[
 \Loop^+ \cH(R) = \cH(R \pot{x}), \quad \text{resp.} \quad \Loop H(R)=H(R\rpot{x}).
\]
It is well known (see e.g.~\cite[Lemma~3.17]{richarz-basics}) that
the functor $\Loop \cH$ is represented by an ind-affine group ind-scheme over $k$, and that $\Loop^+ \cH$ is represented by an affine group scheme over $k$. Regarding $\Loop^+ \cH$, more specifically, for any $i \geq 0$ we can consider the functor $\Loop^+_i \cH$ defined by $\Loop^+_i \cH(R) = \cH(R [x]/(x^{i+1}))$. Then we have
\[
\Loop^+ \cH = \lim_{i \geq 0} \Loop^+_i \cH,
\] 
and for any $i \geq 0$ the functor $\Loop^+_i \cH$ is an affine scheme, which is smooth over $k$ if $\cH$ is smooth over $k \pot{x}$. Note also that
we have an obvious morphism of group ind-schemes
\[
 \Loop^+ \cH \to  \Loop (\cH \otimes_{k \pot{x}} k \rpot{x}),
\]
which is representable by a closed immersion if $\cH$ is of finite type.

Let $\cH$ be an affine group scheme over $k \pot{x}$, and set $H:= \cH \otimes_{k \pot{x}} k \rpot{x}$.
The \emph{affine Grassmannian}
associated with $\cH$ is the fppf quotient $[\Loop H / \Loop^+ \cH]_{\fppf}$, i.e.~the fppf sheaf associated with the presheaf
\begin{equation}
\label{eq:Gr-presheaf}
 R \mapsto \Loop H(R) / \Loop^+ \cH(R).
\end{equation}
It is a standard fact (see e.g.~\cite[Proposition~3.18]{richarz-basics}) that if $\cH$ is smooth over $k\pot{x}$ this sheaf identifies with the functor $\Gr_\cH$ from $k$-algebras to sets sending $R$ to the set of isomorphism classes of pairs $(\cE,\alpha)$ where $\cE \to \Spec(R\pot{x})$ is an fppf\footnote{Since $\cH$ is assumed to be smooth here, the notions of fpqc, fppf or \'etale torsors coincide. Moreover, any such torsor is representable by an $\cH$-principal bundle, that is, a (necessarily smooth affine) scheme which is étale locally on the base isomorphic to $\cH$.} $\cH$-torsor and $\alpha$ is a section of $\cE$ over $\Spec(R\rpot{x})$. In particular, this functor is represented by a separated ind-scheme of ind-finite type over $k$; see~\cite[Theorem~3.4]{richarz-basics}. In fact,~\cite[Proposition~3.18]{richarz-basics} also contains the following claim, which will be useful in the present paper.

\begin{lem}
\label{lem:quotient-etale}
If $\cH$ is a smooth affine group scheme over $k\pot{x}$, then we have $\Gr_{\cH} = [\Loop H/\Loop^+ \cH]_{\et}$, i.e.~the functor $\Gr_\cH$ is also the \emph{\'etale} sheafification of the presheaf~\eqref{eq:Gr-presheaf}.
\end{lem}

\begin{rmk}
\phantomsection
\begin{enumerate}
\item
What we call the \emph{affine Grassmannian} of $\cH$ is sometimes called the \emph{partial affine flag variety} of $\cH$, at least when $\cH$ is a parahoric group scheme in the sense of Bruhat--Tits (see e.g.~\cite{pr}). 
This is justified by the fact that partial affine flag varieties attached to split reductive groups over $k$ (as e.g.~in~G\"ortz~\cite[Definition~2.6]{goertz}) are special cases of this construction. 
The cases that we will consider below give rise to ind-schemes whose properties are very close to those of the ``usual'' affine Grassmannians as considered in~\cite{mv, goertz}. To simplify terminology and notation, and following the conventions in~\cite{zhu-notes, richarz-basics}, we will call all of these ind-schemes affine Grassmannians.
\item
It is proved in~Česnavičius~\cite[Theorem~2.5]{cesnavicius} that if $\cH$ is the base change of a reductive group scheme over $k$, then $\Gr_{\cH}$ is the \emph{Zariski} sheafification of the presheaf quotient $\Loop H/\Loop^+ \cH$, and moreover that no sheafification is needed under additional assumptions, see~\cite[Theorem~3.4]{cesnavicius}. These results do not apply in the general setting considered below, and we do not know if these properties are satisfied in our case.
\end{enumerate}
\end{rmk}

The following fact is clear from the definitions as functors.

\begin{lem}
\label{lem:Gr-base-change}
Let $\cH$ be a smooth affine group scheme over $k\pot{x}$.
If $k'$ is an extension of $k$, there exist canonical isomorphisms of $k'$-(ind-)schemes
\begin{gather*}
 \Loop (k'\rpot{x} \otimes_{k\rpot{x}} H) \simto k' \otimes_k (\Loop H), \qquad  
 \Loop^+ (k'\pot{x} \otimes_{k\pot{x}} \cH) \simto k' \otimes_k (\Loop^+ \cH), \\
 \Gr_{k'\pot{x} \otimes_{k\pot{x}}\cH} \simto k' \otimes_k \Gr_\cH.
\end{gather*}
\end{lem}

For any $k$-scheme $Y$ and any $y \in Y(k)$ we have the corresponding tangent space $T_y Y$, see~\cite[\href{https://stacks.math.columbia.edu/tag/0B2C}{Tag 0B2C}]{stacks-project}, which identifies with the subset of $Y(k[\varepsilon]/\varepsilon^2)$ consisting of points whose image under $\epsilon \mapsto 0$ in $Y(k)$ is $y$. 
This definition extends in the obvious way to ind-schemes.
In particular, if $\cH$ is a smooth affine group scheme over $k\pot{x}$, and setting as above $H:=\cH \otimes_{k \pot{x}} k \rpot{x}$, we can consider the Lie algebras $\Lie(\Loop H)$ and $\Lie(\Loop^+ \cH)$ of the group ind-scheme $\Loop H$ and of the group scheme $\Loop^+ \cH$, defined as the tangent spaces at their unit point. We can also consider the base point $e$ of $\Gr_\cH$, and the associated tangent space $T_e \Gr_\cH$.

\begin{lem}
\label{lem:tangent-space-Fl}
Let $\cH$ be a smooth affine group scheme over $k\pot{x}$.
There exists a canonical identification
\[
T_e \Gr_\cH \cong \Lie(\Loop H) / \Lie(\Loop^+ \cH).
\]
\end{lem} 

\begin{proof}
The exact sequence of pointed sheaves $1\to \Loop^+\cH \to \Loop H \to \Gr_\cH \to 1$ induces an exact sequence of $k$-vector spaces
	\begin{equation*}
	0\to \Lie(\Loop^+\cH) \to \Lie(\Loop H)\to T_e\Gr_\cH.
	\end{equation*}
	In fact this sequence is also exact on the right: 
	by Lemma~\ref{lem:Gr-base-change} and compatibility of tangent spaces with extension of the base field, without loss of generality we may and do assume that $k$ is algebraically closed. 
	Then this easily follows from the fact that
	\[
	\Gr_\cH(k[\varepsilon]/\varepsilon^2) = \Loop H(k[\varepsilon]/\varepsilon^2) / \Loop^+ \cH(k[\varepsilon]/\varepsilon^2)
	\]
	by Lemma~\ref{lem:quotient-etale},
	since $k[\varepsilon]/\varepsilon^2$ is a strictly henselian ring.
\end{proof}

In this paper, we will apply the affine Grassmannian construction in two different settings, that we now explain. Let $\F$ be an algebraically closed field, and set
\[
F:=\F \rpot{t}, \quad O_F:=\F\pot{t}
\]
where $t$ is an indeterminate. We denote by $F^s$ a separable closure of $F$, and set 
\[
I := \mathrm{Gal}(F^s / F),
\]
which is the inertia group of $F$ as the residue field $\F$ is algebraically closed. For any $F$-scheme $X$ we will write $X_{F^s}$ for $F^s \otimes_F X$.

Given a smooth affine group scheme $\cH$ over $O_F$, we consider the associated affine Grassmannian $\Gr_{\cH}$ over $\F$ (where the indeterminate $x$ is equal to $t$). On the other hand, if $K$ is either $F$ or $F^s$ and $P$ is a smooth affine group scheme over $K$, and if $z$ is an indeterminate, we also consider the above construction for the smooth group scheme $K\pot{z} \otimes_K P$ over $K\pot{z}$ (and the indeterminate $x=z$) to obtain the $K$-ind-scheme $\Gr_{K\pot{z} \otimes_K P}$. This ind-scheme is the object considered in~\cite{mv} or~\cite[Definition~2.5]{goertz}. Lemma~\ref{lem:Gr-base-change} implies that if $P$ is a smooth affine group scheme over $F$ we have an identification
\begin{equation}
\label{eqn:Gr-base-change}
\Gr_{F^s\pot{z} \otimes_{F^s} (F^s \otimes_F P)} \simto F^s \otimes_F \Gr_{F\pot{z} \otimes_F P}.
\end{equation}

\subsection{Reductive groups and special facets}
\label{ss:special-facets}

From now on we fix a smooth affine group scheme $\cG$ over $O_F$, and set
\[
G:=F \otimes_{O_F} \cG.
\]
It is important to know exactly for which $\cG$ the associated ind-scheme $\Gr_\cG$ is ind-proper (equivalently, ind-projective, as it is always ind-quasi-projective). A full criterion can be found in~\cite[Th\'eor\`eme~5.2]{lourenco-BT}, and relates to (a generalization of) Bruhat--Tits theory. For our purposes, we will assume that $G$ is a connected reductive group over $F$ and that the special fiber $\F \otimes_{O_F} \cG$ is connected; in this case it was shown (earlier) in \cite[Theorem~A]{richarz} that $\Gr_\cG$ is ind-proper if and only if $\cG$ is a parahoric group scheme in the sense of Bruhat--Tits~\cite{BT84}. 

The case that will be relevant in this paper is the following. Consider the (extended) Bruhat--Tits building $\mathscr{B}(G,F)$ associated with the reductive group $G$ over the local field $F$. The parahoric group schemes are attached to the facets in $\mathscr{B}(G,F)$.  We henceforth impose the following two assumptions:
\begin{itemize}
\item $\cG$ is the parahoric group scheme attached to some facet $\fa \subset \mathscr{B}(G,F)$.
\item The facet $\fa$ is \emph{special}.
\end{itemize}

The latter assumption is equivalent to a certain geometric condition on $\Gr_\cG$: see~\cite[Theorem~B]{richarz}.  For an explicit example of this setting with $G$ nonsplit, see~\cite[\S 4]{richarz-schubert}.

\subsection{Tori, (co)weights, and (co)roots}
\label{ss:tori-weights}

By Steinberg's theorem,\footnote{Here ``Steinberg's theorem'' refers to~\cite[Corollary~10.2(a)]{steinberg-reg}. This statement has the assumption that the base field is perfect, but it is well known (and explicitly stated in~\cite[\S 8.6]{bs}) that this assumption can be removed if the group is assumed to be reductive.} $G$ is quasi-split, i.e.~there exist Borel subgroups $B \subset G$ defined over $F$. By~\cite[Corollaire~4.16]{bot}, any such $B$ contains a maximal $F$-split torus $A \subset G$ in its radical, and by~\cite[Th\'eor\`eme~4.15]{bot} the centralizer $T:=Z_G(A)$ of $A$ is then a Levi subgroup of $B$, and thus a maximal torus of $G$.
The facet $\fa$ belongs to the apartment $\mathscr{A}(G,A',F)$ associated with some maximal $F$-split torus $A'$. By conjugacy of maximal $F$-split tori in $G$ (see~\cite[Th\'eor\`eme~4.21]{bot}), we can and will assume that $A=A'$.

Let us denote by $\cA$, resp.~$\cT$, the scheme-theoretic closure of $A$, resp.~$T$, in $\cG$. By~\cite[Appendix~1]{richarz}, $\cT$ is the unique parahoric group scheme of $T$. As explained in~\cite[\S 3.b]{pr}, $\cT$ therefore identifies with the connected N\'eron model of $T$. The group scheme $\cA$ is also the scheme-theoretic closure of $A$ in $\cT$; in view of~\cite[\S 4.4]{BT84}, $\cA$ is therefore the natural split torus extending $A$,
which coincides with its connected N\'eron model. 

Recall that if $H$ is an $F$-torus, the $F^s$-torus $H_{F^s}$ is split, see~\cite[Proposition~1.5]{bot}; we will denote by $\bbX^*(H)$, resp.~$\bbX_*(H)$, its character, resp.~cocharacter, lattice. (We emphasize that, contrary to what the notation might suggest, in general $\bbX^*(H)$ and $\bbX_*(H)$ consist of characters and cocharacters of $H_{F^s}$ and not of $H$. If $H$ is already split however, they can be seen as characters and cocharacters of $H$.) Given a group $C$ and a $\Z$-module $M$ endowed with a linear action of $C$, we will denote by $M_C := \Z \otimes_{\Z[C]} M$ the module of coinvariants. (Here $\Z[C]$ is the group algebra of $C$ over $\Z$, and $\Z$ is endowed with the trivial action of $C$.)

In our present setting, we will in particular consider the torus $T$, its base change $T_{F^s}$, and the lattices $\bbX^*(T)$ and $\bbX_*(T)$.
The group $I$ acts on $\bbX_*(T)$, and this action factors through a finite quotient since $T$ in fact splits over a finite separable extension; we will consider the associated coinvariants $\bbX_*(T)_I$.
Recall the Kottwitz homomorphism $T(F) \to \bbX_\ast(T)_I$, see e.g.~\cite[\S 2.a.2]{pr}; this morphism is functorial on the category of $F$-tori, compatible with products and is given for $T=\Gm$ by the natural map $F^\times \to F^\times/O_F^\times = \bbX_*(\Gm)$, where the identification is induced by $\mu\mapsto \mu(t)$.
As explained in~\cite[\S 5.a]{pr}, this morphism factors through an isomorphism
\begin{equation}
\label{eqn:Gr-torus}
T(F) / \cT(O_F) \simto \bbX_*(T)_I.
\end{equation}

The reductive group
$G_{F^s}$
over $F^s$ is split since it admits the split maximal torus $T_{F^s}$, see~\cite[Exp.~XXII, Proposition~2.2]{sga33}. 
We can therefore consider its roots and coroots with respect to $T_{F^s}$, which will be denoted
\[
\Phi_\abs \subset \bbX^*(T) \quad \text{and} \quad \Phi^\vee_\abs \subset \bbX_*(T)
\]
respectively. (The subscript ``$\abs$'' stands for ``absolute.'') These subsets are stable under the $I$-actions on $\bbX^*(T)$ and $\bbX_*(T)$. The nonzero weights of $T_{F^s}$ in the Lie algebra of the Borel subgroup
$B_{F^s} \subset G_{F^s}$
form a system of positive roots $\Phi_\abs^+ \subset \Phi_\abs$ which is stable under the action of $I$. The subset of dominant coweights of $T_{F^s}$ (with respect to this choice of positive roots) will be denoted $\bbX_*(T)^+$.

We will consider in particular the sublattice $\Z \Phi^\vee_\abs \subset \bbX_*(T)$ generated by the coroots. The quotient $\bbX_*(T)/\Z \Phi^\vee_\abs$ is called the \emph{algebraic fundamental group} of $G$, and is denoted $\pi_1(G)$.

\begin{lem}
\label{lem:coroots-coweights-coinv}
The exact sequence $\Z \Phi^\vee_\abs \hookrightarrow \bbX_*(T) \twoheadrightarrow \pi_1(G)$ induces an exact sequence
\[
(\Z \Phi^\vee_\abs)_I \hookrightarrow \bbX_*(T)_I \twoheadrightarrow \pi_1(G)_I.
\]
\end{lem}

\begin{proof}
By right exactness of the coinvariants functor, it suffices to prove that the morphism $(\Z \Phi^\vee_\abs)_I \to \bbX_*(T)_I$ is injective. However, $\Z \Phi^\vee_\abs$ has a basis consisting of the simple coroots, which is permuted by $I$. As a consequence $(\Z \Phi^\vee_\abs)_I$ is free, with a basis in bijection with $I$-orbits of simple coroots. To prove the claim, it therefore suffices to prove that the induced morphism
\begin{equation}
\label{eqn:coroots-coweights-coinv}
\Q \otimes_\Z \bigl( (\Z \Phi^\vee_\abs)_I \bigr) \to \Q \otimes_\Z \bigl( \bbX_*(T)_I \bigr)
\end{equation}
is injective. Here the first term identifies with $(\Q \otimes_{\Z} \Z \Phi^\vee_\abs)_I$, and the second one with $(\Q \otimes_\Z \bbX_*(T))_I$. Now the subspace $\Q \otimes_{\Z} \Z \Phi^\vee_\abs \subset \Q \otimes_\Z \bbX_*(T)$ has an $I$-stable complement, consisting of the elements orthogonal to all roots, which implies the injectivity of~\eqref{eqn:coroots-coweights-coinv}.
\end{proof}

The modules considered in Lemma~\ref{lem:coroots-coweights-coinv} appear in the description of the set of connected components $\pi_0(\Gr_{\cG})$ of $\Gr_{\cG}$; namely,
by~\cite[Theorem~0.1]{pr} the Kottwitz morphism induces a bijection
\begin{equation}
\label{eqn:conn-comp-Gr}
\pi_0(\Gr_{\cG}) \simto \pi_1(G)_I.
\end{equation}

Now we consider the lattices $\bbX^*(A)$ and $\bbX_*(A)$ 
associated with the split torus $A$; the system of (relative) roots of $(G,A)$ will be denoted
\[
\Phi \subset \bbX^*(A).
\]
Restriction along $A\subset T$ induces a canonical morphism $\bbX^*(T) \to \bbX^*(A)$, which is $I$-equivariant with respect to the trivial action on $\bbX^*(A)$. 
Since $G$ is quasi-split, this morphism sends $\Phi_\abs$ onto $\Phi$.
The image $\Phi^+$ of $\Phi^+_\abs$ is a system of positive roots for $\Phi$; moreover, if we denote by $\Phi_\abs^{\mathrm{s}}$, resp.~$\Phi^{\mathrm{s}}$, the associated basis of $\Phi_\abs$, resp.~$\Phi$, then $\Phi_\abs^{\mathrm{s}}$ is stable under the action of $I$, and our morphism $\bbX^*(T) \to \bbX^*(A)$ restricts to a surjection
\begin{equation}
\label{eqn:surjection-Phis}
\Phi_\abs^{\mathrm{s}} \twoheadrightarrow \Phi^{\mathrm{s}}
\end{equation}
whose fibers are exactly the $I$-orbits in $\Phi_\abs^{\mathrm{s}}$. (For all of this, see~\cite[\S 4.1.2]{BT84}.)

\subsection{Iwahori--Weyl group and Schubert varieties}
\label{ss:Iwahori-Weyl-Schubert}

The $\Loop^+\cG$-orbit subsche\-mes inside $\Gr_{\cG}$ are smooth and locally closed, giving rise to a topological stratification. We will now describe these orbits together with their closures, following~\cite[\S\S 1-2]{richarz-schubert} and~\cite[\S 2.1 and~\S 3]{richarz}.

The \emph{Iwahori--Weyl group} of $(G,A)$ is the quotient
\[
W := \bigl( N_G(A) \bigr) (F) / \cT(O_F),
\]
where $N_G(A)$ is the normalizer of $A$ in $G$. This group contains the quotient
\[
T(F) / \cT(O_F) \cong \bbX_*(T)_I
\]
(see~\eqref{eqn:Gr-torus})
as a normal subgroup, and the quotient is the finite Weyl group
\[
W_0 := \bigl( N_G(A) \bigr) (F) / T(F).
\]
Setting
\[
W_\fa := \bigl( N_G(A) (F) \cap \cG(O_F) \bigr) / \cT(O_F),
\]
the composition
\[
W_\fa \hookrightarrow W \to W_0
\]
is an isomorphism (see~\cite[Appendix, Proposition~13]{pr} or~\cite[Remark~1.4]{richarz-schubert}). 
So we obtain an identification
\[
W \cong W_\fa \ltimes \bbX_*(T)_I.
\]

Consider the natural
composition $\bbX_*(A) \to \bbX_*(T) \to \bbX_*(T)_I$; the induced morphism
\begin{equation}
\label{eqn:identification-cocharacters}
a \colon  \Q \otimes_\Z \bbX_*(A) \to
\Q \otimes_\Z (\bbX_*(T)_I)
\end{equation}
is an isomorphism. We set
\[
\bbX_*(T)_I^+ = \{ \lambda \in \bbX_*(T)_I \mid \forall \alpha \in \Phi^+, \, \langle a^{-1}(\overline{\lambda}), \alpha \rangle \geq 0 \},
\]
where $\overline{\lambda}$ denotes the image of $\lambda$ in $\Q \otimes_\Z (\bbX_*(T)_I)$.

\begin{lem}
\phantomsection
\label{lem:dominant-coweights}
\begin{enumerate}
\item
\label{it:dominant-coweights-orbits}
The composition
\[
\bbX_*(T)_I^+ \hookrightarrow \bbX_*(T)_I \twoheadrightarrow \bbX_*(T)_I / W_0
\]
 is a bijection; in other words, $\bbX_*(T)_I^+$ is a system of representatives for the $W_0$-orbits in $\bbX_*(T)_I$.
 \item
 \label{it:dominant-coweights-lifting}
 The composition
 \[
 \bbX_*(T)^+ \hookrightarrow \bbX_*(T) \twoheadrightarrow \bbX_*(T)_I
 \]
 factors through a map $\bbX_*(T)^+ \to \bbX_*(T)_I^+$. 
 This map is surjective if one of the following equivalent conditions holds, where $Z$ is the (scheme-theoretic) center of $G$:
  \begin{enumerate}
  	\item
	\label{it:dominant-coweights-lifting-cond-1}
	$Z$ is a torus; 
	\item
	\label{it:dominant-coweights-lifting-cond-2}
	the abelian group $\bbX^*(Z)$ of characters of $F^s \otimes_F Z$ is torsion free;
	\item
	\label{it:dominant-coweights-lifting-cond-3}
	the natural map $\bbX_*(T)\to \bbX_*(T/Z)$ is surjective.
  \end{enumerate}
\end{enumerate}
\end{lem}

\begin{proof}
\eqref{it:dominant-coweights-orbits}
Let $Z$ be the center of $G$, and consider the adjoint group $G_{\mathrm{ad}}:=G/Z$. Let also $T_{\mathrm{ad}}$, resp.~$A_{\mathrm{ad}}$, be the image of $T$, resp.~$A$, in $G_{\mathrm{ad}}$; then $T_{\mathrm{ad}}$, resp.~$A_{\mathrm{ad}}$, is a maximal, resp.~maximal split, torus of $G_{\mathrm{ad}}$, see~\cite[Th\'eor\`eme~2.20]{bot2}.
If we denote by $\Phi_{\mathrm{ad}}$, resp.~$W_{{\rm ad},0}$, the (relative) root system, resp.~Weyl group, of $(G_{\mathrm{ad}},A_{\mathrm{ad}})$, then
the quotient map $G\to G_{\mathrm{ad}}$ induces bijections $\Phi \cong \Phi_{\mathrm{ad}}$ and $W_0 \cong W_{{\rm ad},0}$, compatible with the action of the latter on the former.
So the map $T\to T_{\mathrm{ad}}$ induces a commutative diagram of pointed sets
\[
\begin{tikzcd}
 \bbX_*(T)^+_I \ar[r] \ar[d] &  Q^+ \ar[d] \ar[hookrightarrow, r] & \bbX_*(T_{\mathrm{ad}})_I^+ \ar[d]\\
 \bbX_*(T)_I/W_0 \ar[r] & Q/W_0  \ar[hookrightarrow, r] & \bbX_*(T_{\mathrm{ad}})_I/W_0,
\end{tikzcd}
\]
where $Q$ is the image of $\bbX_*(T)_I\to \bbX_*(T_{\mathrm{ad}})_I$
and $Q^+$ is its intersection with $\bbX_*(T_{\mathrm{ad}})_I^+$.
The left horizontal maps are surjective. Moreover,
since the action of $W_0$ and the pairing with the relative roots is trivial on $\ker(\bbX_*(T)_I\to \bbX_*(T_{\rm ad})_I)$, their fibers are canonically bijective to this kernel,
which is a subgroup of the monoid $\bbX_*(T)^+_I$.
Hence, to prove our claim it suffices to prove bijectivity of the map $Q^+\to Q/W_0$.

For this it is convenient to make the connection with the ``\'echelonnage root system'' as follows.
Fixing a point in $\fa$ defines an identification of
\[
V:=\bbX_*(A_{\rm ad}) \otimes_\Z \R = \bbX_*(T_{\rm ad})_I \otimes_\Z \R
\]
(see~\eqref{eqn:identification-cocharacters})
with the apartment $\mathscr{A}(G_{\rm ad},A_{\rm ad},F)$ on which the Iwahori--Weyl group $W_{\mathrm{ad}}=W(G_{\rm ad},A_{\rm ad})$ acts by affine transformations.
There exists a reduced root system $\Sigma \subset V^*$, called \emph{\'echelonnage root system} by Bruhat--Tits, such that the associated affine Weyl group $W_{\rm af}(\Sigma)$ is isomorphic to the Iwahori--Weyl group $W_{\rm sc}=W(G_{\rm sc}, A_{\rm sc})$ of the simply connected cover $G_{\rm sc}\to G_{\rm ad}$ (see~\cite[Proposition~2.24]{bot2}) with $A_{\rm sc}$ being the preimage of $A_{\rm ad}$ in $G_{\rm sc}$, and such that the identification $W_{\rm af}(\Sigma)=W_{\rm sc}$ is compatible with the actions on $V$; see \cite[Section 6.1]{Haines:Dualities}.
In particular, $W_0(\Sigma)=W_0$ for the finite Weyl groups. 
One necessarily has $\mathbb Q\cdot \Phi=\mathbb Q\cdot \Sigma$ (in fact, those are the same up to multiples of $2^{\pm 1}$), so the positive relative roots $\Phi^+$ determine a system of positive roots $\Sigma^+ \subset \Sigma$.
In $V$ we have (in general strict) inclusions of $\Z$-sublattices
\[
Q^\vee(\Sigma)={\rm image}(\bbX_*(T_{\rm sc})_I\to \bbX_*(T_{\rm ad})_I)\subset Q \subset \bbX_*(T_{\rm ad})_I\subset P^\vee(\Sigma),
\]
where $Q^\vee(\Sigma)\subset P^\vee(\Sigma)$ is the coroot (resp.~coweight) lattice of $\Sigma$, and $T_{\mathrm{sc}}$ is the centralizer of $A_{\mathrm{sc}}$ in $G_{\mathrm{sc}}$ (which is the same as the preimage of $T_{\mathrm{ad}}$ in $G_{\mathrm{sc}}$).
Here we note that $\bbX_*(T_{\mathrm{ad}})_I$ is torsion free because $\bbX_*(T_{\mathrm{ad}})$ admits a basis permuted by $I$ (namely, the fundamental coweights).
Hence, the bijectivity of $Q^+\to Q/W_0$ follows from standard facts on reduced root systems.

\eqref{it:dominant-coweights-lifting}
Since the positive relative roots $\Phi^+$ are restrictions of the positive absolute roots $\Phi_{\rm abs}^+$, the restriction of the quotient map $\bbX_*(T)\to \bbX_*(T)_I$ to $\bbX_*(T)^+$ factors through a map $\bbX_*(T)^+\to \bbX_*(T)^+_I$.

To see that the conditions 
\eqref{it:dominant-coweights-lifting-cond-1}, \eqref{it:dominant-coweights-lifting-cond-2} and \eqref{it:dominant-coweights-lifting-cond-3}
are equivalent we note that the center $Z$ of $G$ is a $F$-group scheme of multiplicative type. 
So $Z$ is a torus if and only if $\bbX^*(Z)$ is torsion free, which shows the equivalence of~\eqref{it:dominant-coweights-lifting-cond-1} and~\eqref{it:dominant-coweights-lifting-cond-2}. 
Next, the exact sequence $1\to Z\to T\to T_{\rm ad}\to 1$ (where we use the notation of the proof of~\eqref{it:dominant-coweights-orbits}) induces a short exact sequence $0\to \bbX^*(T_{\rm ad})\to \bbX^*(T)\to \bbX^*(Z)\to 0$ of $\Z$-modules where $\bbX^*(T_{\rm ad})$ and $\bbX^*(T)$ are torsion free.
Applying the functor $\Hom_{\Z{\rm \text{-}Mod}}({\rm \text{-}},\Z)$ we obtain an exact sequence
\[
0\to 
\Hom_{\Z{\rm \text{-}Mod}}(\bbX^*(Z),\Z) \to \bbX_*(T)\to \bbX_*(T_{\rm ad}) \to {\rm Ext}^1_{\Z{\rm \text{-}Mod}}(\bbX^*(Z),\Z)\to 0.
\]
Hence, $\bbX^*(Z)$ being torsion free is equivalent to the surjectivity of $\bbX_*(T)\to \bbX_*(T_{\rm ad})$, which shows the equivalence of~\eqref{it:dominant-coweights-lifting-cond-2} and~\eqref{it:dominant-coweights-lifting-cond-3}.

Next, we show surjectivity of the map $\bbX_*(T)^+ \to \bbX_*(T)_I^+$ in case $G=G_{\mathrm{ad}}$ in the notation used above.
In this case, the fundamental coweights $(\omega^\vee_\alpha : \alpha\in \Phi_\abs^{\mathrm{s}})$ form a $\Z$-basis of $\bbX_*(T)$, and a $\Z_{\geq 0}$-basis of $\bbX_*(T)^+$.
For any $\alpha\in \Phi_\abs^{\mathrm{s}}$, the image $\bar\omega^\vee_\alpha$ of $\omega^\vee_\alpha$ in $\bbX_*(T)_I$ only depends on the orbit $I \cdot \alpha$, and the family
$(\bar\omega^\vee_\alpha : \alpha\in \Phi_\abs^{\mathrm{s}}/I)$ forms a $\Z$-basis of $\bbX_*(T)_I$. We claim that this family is also a $\Z_{\geq 0}$-basis of $\bbX_*(T)_I^+$, which will imply the desired claim.

In fact, the map $a$ in~\eqref{eqn:identification-cocharacters} is a composition of the natural maps
\[
 \Q \otimes_{\Z} \bbX_*(A) \xrightarrow{a_1} \Q \otimes_{\Z} \bbX_*(T) \xrightarrow{a_2} \Q \otimes_{\Z} \bbX_*(T)_I.
\]
For $\alpha,\beta \in \Phi_\abs^{\mathrm{s}}$ one has 
\[
\langle \bar\omega^\vee_\alpha,\beta_{|A} \rangle = \langle a_1(\bar\omega^\vee_\alpha),\beta \rangle.
\]
If we fix a finite quotient $\bar I$ of $I$ through which the action on $\bbX_*(T)$ factors, then the element $a_1(\bar\omega^\vee_\alpha)$ is $\bar I$-invariant, so 
\[
\textstyle{
\langle a_1(\bar\omega^\vee_\alpha),\beta \rangle = \frac{1}{|\bar I|} \langle \sum_{i\in \bar I} i \cdot a_1(\bar\omega^\vee_\alpha),\beta \rangle = \frac{1}{|\bar I|} \langle a_1(\bar\omega^\vee_\alpha), \sum_{i\in \bar I} i \cdot\beta \rangle.
}
\]
Now $\sum_{i\in \bar I} i \cdot\beta$ is $I$-invariant, and the elements $a_1(\bar\omega^\vee_\alpha)$ and $\omega^\vee_\alpha$ have the same image in $\Q \otimes_{\Z} \bbX_*(T)_I$, hence we have
\[
 \textstyle{
\frac{1}{|\bar I|} \langle a_1(\bar\omega^\vee_\alpha), \sum_{i\in \bar I} i \cdot\beta \rangle = \frac{1}{|\bar I|} \langle \omega^\vee_\alpha, \sum_{i\in \bar I} i \cdot\beta \rangle
}
= \begin{cases}
\frac{1}{| I \cdot \alpha|} & \text{if $\beta\in I\alpha$;} \\
0 & \text{otherwise.}
\end{cases}
\]
From this computation we deduce that a linear combination of the $\bar\omega^\vee_\alpha$ has nonnegative pairing with any simple (relative) root if and only if the coefficient of each $\bar\omega^\vee_\alpha$ is nonegative, which proves the desired claim.

Finally, we will prove that the map $\bbX_*(T)^+ \to \bbX_*(T)_I^+$ is surjective in case the morphism $\bbX_*(T)\to \bbX_*(T_{\rm ad})$ is surjective.
Setting  $K:=\ker(\bbX_*(T)\to \bbX_*(T_{\rm ad}))$, we consider the diagram of natural maps
\[
\begin{tikzcd}
0\ar[r] &K \ar[r] \ar[d, two heads] & \bbX_*(T) \ar[r] \ar[d, two heads] &  \bbX_*(T_{\rm ad}) \ar[d, two heads] \ar[r] & 0\\
 & K_I \ar[r] & \bbX_*(T)_I \ar[r] &  \bbX_*(T_{\rm ad})_I\ar[r] & 0.
\end{tikzcd}
\]
If $\bar\lambda\in \bbX_*(T)^+_I$, then its image in $\bbX_*(T_{\rm ad})_I$ belongs to $\bbX_*(T_{\rm ad})_I^+$, hence is the image of some $\mu \in \bbX_*(T_{\rm ad})^+$ by the case treated above. If $\nu$ is any preimage of $\mu$ in $\bbX_*(T)$, then its image $\bar \nu$ in $\bbX_*(T)_I$ has the same image as $\bar \lambda$ in $\bbX_*(T_{\rm ad})_I$. Hence there exists $\eta \in K$ whose image in $K_I$ has image $\bar\lambda - \bar\nu$ in $\bbX_*(T)_I$. Then $\lambda := \nu + \eta$ has image $\bar \lambda$ in $\bbX_*(T)_I$, and its image in $\bbX_*(T_{\mathrm{ad}})$ is $\mu$. Since the pairing of any coweight of $T$ with an absolute root only depends on its image in $\bbX_*(T_{\rm ad})$ we have $\lambda \in \bbX_*(T)^+$, which finishes the proof.
\end{proof}

\begin{rmk}
\label{rmk:counterexample-surjectivity}
Contrary to what is asserted in~\cite[Remark~3.8]{richarz} and \cite[page 3227, lines 30--31]{HainesRicharz:Smoothness},
the surjectivity of $\bbX_*(T)^+\to \bbX_*(T)_I^+$ fails in general.
Both references use the false claim implicitly through \cite[Proof of Corollary 2.8]{zhu} where the same false claim appears.
The proof can be fixed using Corollary \ref{cor:passing-to-adjoint} below by passing to adjoint groups and referring to Lemma \ref{lem:dominant-coweights}\eqref{it:dominant-coweights-lifting} in this case.

For an explicit example, let $G={\rm SU}_3$ be the special unitary group on a $3$-dimensional hermitian vector space defined by some separable quadratic extension $F'/F$. 
Then, $\bbX_*(T)$ can be identified with the group $\Z^3_{\Sigma =0}$ of elements $(a,b,c)\in \Z^3$ with $a+b+c=0$, with
the non-trivial Galois involution in ${\rm Gal}(F'/F)$ acting by $(a,b,c)\mapsto (-c,-b,-a)$.
The subgroup $(1-I)\Z^3_{\Sigma =0}$ identifies with the subset of vectors $(-b,2b,-b)$ with $b\in \Z$, and the map $(a,b,c)\mapsto a-c$ induces an isomorphism $\Z^3_{\Sigma =0}/(1-I)\Z^3_{\Sigma =0}\simeq\Z$.
Hence, $\bbX_*(T)\to \bbX_*(T)_I$ identifies with $\Z^3_{\Sigma =0}\to \Z$, $(a,b,c)\mapsto a-c$.
The dominant elements $(\Z^3_{\Sigma =0})^+$ with respect to the upper triangular Borel subgroup are given by the condition $a\geq b\geq c$ in $\Z^3_{\Sigma =0}$ and by $a\geq 0$ in $\Z$. 
An easy calculation shows that the image of $(\Z^3_{\Sigma =0})^+\to \Z^+=\Z_{\geq 0}$ is $2\Z_{\geq 0}$, so the map is not surjective.
\end{rmk}

The Iwahori--Weyl group $W$ acts on the apartment $\mathscr{A}(G,A,F)$ by affine transformations. Any choice of a point in $\fa$ defines an identification of
\[
\bbX_*(A) \otimes_\Z \R = \bbX_*(T)_I \otimes_\Z \R
\]
(see~\eqref{eqn:identification-cocharacters})
with $\mathscr{A}(G,A,F)$ such that an element $\mu$ of the subgroup $T(F)/\cT(O_F)\cong \bbX_*(T)_I$ (see~\eqref{eqn:Gr-torus}) of $W$ acts by translation by $-\mu$. 
Let $\fa_0$ be the alcove in $\mathscr{A}(G,A,F)$ containing $\fa$ in its closure and which belongs to the chamber corresponding under this identification to the chamber in $X_*(A)\otimes_\Z\R$ determined by $B$. 
This choice determines a structure of a quasi-Coxeter group on $W$; see~\cite[Appendix, Lemma~14]{pr} and~\cite[\S 1]{richarz-schubert} for details. In terms of this structure, one can consider a subset
\[
{}_\fa W^\fa \subset W
\]
characterized in~\cite[(1.5)]{richarz-schubert}, which is a system of representatives for the double quotient
\[
W_\fa \backslash W / W_\fa \cong \bbX_*(T)_I / W_0,
\]
and is thus in canonical bijection with the double quotient
\[
\cG(O_F)\backslash G(F)/\cG(O_F) = \Loop^+ \cG(\F) \backslash \Loop G(\F) / \Loop^+ \cG(\F);
\]
see~\cite[Lemma~1.3]{richarz-schubert}.
Since $\fa$ is special, the subset ${}_\fa W^\fa$ has an explicit description, which follows from~\cite[Corollary~1.8]{richarz-schubert} or the more general claim~\cite[Lemma~3.9]{richarz}: it coincides with the subset
\[
\bbX_*(T)^+_I \subset \bbX_*(T)_I \subset W.
\]

Any $\mu \in \bbX_*(T)_I$ determines a point in $T(F)/\cT(O_F)$ (see~\eqref{eqn:Gr-torus}), and hence an $\F$-point of $\Gr_{\cG}$, which we denote $t^\mu$. We denote by $\Gr_{\cG}^\mu$ the $\Loop^+ \cG$-orbit of $t^\mu$, and by $\Gr_{\cG}^{\leq \mu}$ the scheme-theoretic closure of $\Gr_{\cG}^\mu$. 
We also denote by
\begin{equation}
\label{eqn:jmu}
j^\mu \colon \Gr_{\cG}^\mu \to \Gr_{\cG}, \qquad j^{\leq \mu} \colon \Gr_{\cG}^{\leq \mu} \to \Gr_{\cG}
\end{equation}
the natural immersions.
Then $\Gr_{\cG}^{\leq \mu}$ is a projective variety over $\F$, and $\Gr_{\cG}^\mu$ is a smooth open dense subscheme of $\Gr_{\cG}^{\leq \mu}$, which admits a paving by affine spaces;\footnote{The paving by affine spaces in this general setting can be deduced from \cite[Proposition~8.7]{pr} using Demazure resolutions.} see~\cite[\S 2.1]{richarz} for details. 
This notation is justified by the following fact. Recall from Lemma~\ref{lem:coroots-coweights-coinv} and its proof that the submodule $(\Z \Phi_\abs^\vee)_I \subset \bbX_*(T)_I$ admits a basis in natural bijection with the $I$-orbits of simple roots. We can therefore define an order $\leq$ on $\bbX_*(T)_I$ by declaring that $\lambda \leq \mu$ iff $\mu-\lambda$ belongs to the submonoid generated by this basis.
Then on the underlying sets we have
\begin{equation}
\label{eqn:stratif-Schubert}
| \Gr_{\cG}^{\leq \mu} | = \bigsqcup_{ \substack{\lambda \in \bbX_*(T)_I^+ \\ \lambda \leq \mu }} |\Gr_{\cG}^\lambda| ;
\end{equation}
see~\cite[Corollary~1.8 and Proposition~2.8]{richarz-schubert}.

Let $\rho \in \Q \otimes_\Z \bbX^*(T)$ be one-half the sum of the roots in $\Phi^+_\abs$. Since this element is $I$-invariant, the pairing
\[
\langle -, 2\rho \rangle \colon \bbX_*(T) \to \Z
\]
factors through a map $\bbX_*(T)_I \to \Z$, which we denote similarly. With this notation, for any $\lambda \in \bbX_*(T)_I^+$ we have
\begin{equation}
\label{eqn:dim-orbits}
\dim(\Gr_{\cG}^\lambda) = \langle \lambda, 2\rho \rangle,
\end{equation}
see again~\cite[Corollary~1.8 and Proposition~2.8]{richarz-schubert}.
 
\begin{rmk}
It is known that $\Gr_{\cG}^{\leq \mu}$ is always geometrically unibranch (i.e.~the normalization morphism is a universal homeomorphism, see~\cite[Corollary~32]{kollar}), see~\cite[Proposition~3.1]{HainesRicharz:CM}.
Under some additional mild assumptions, $\Gr_{\cG}^{\leq \mu}$ turns out to also be normal and Cohen--Macaulay, due to results of Faltings \cite{faltings-loops}, Pappas--Rapoport \cite{pr}, Fakhruddin--Haines--Louren{\c c}o--Richarz~\cite{FHLR} and the second named author \cite{lourenco-normality} (see also \cite{HLR, br-affgrass} for failures of this property). 
We do not need these results below, since universal homeomorphisms induce equivalences on small étale topoi.
\end{rmk}

\subsection{Convolution schemes}
\label{ss:convolution-schemes}

Given a smooth affine group scheme $\cH$ over $O_F$, we can
consider the ``convolution functor'' $\Conv_{\cH}$ over $\F$ whose $R$-points consist of isomorphism classes of triples $(\cE_1, \cE_2, \alpha,\beta)$ where $\cE_1,\cE_2$ are $\cH$-torsors over $\Spec(R\pot{t})$, $\alpha$ is a trivialization of $\cE_1$ over $\Spec(R\rpot{t})$, and $\beta$ is an isomorphism between the restrictions of $\cE_1$ and $\cE_2$ to $\Spec(R\rpot{t})$. We have a ``multiplication'' morphism
\[
m \colon \Conv_{\cH} \to \Gr_{\cH}, \; (\cE_1, \cE_2, \alpha,\beta)\mapsto (\cE_2, \beta \circ \alpha),
\]
where we use the moduli description of $\Gr_{\cH}$ recalled in~\S\ref{ss:loop-gps}, as well as the projection map 
\[
\pr_1 \colon \Conv_{\cH} \to \Gr_{\cH}, \; (\cE_1, \cE_2, \alpha,\beta)\mapsto (\cE_1, \alpha).
\]
It is clear that the product map
\begin{equation}
\label{eqn:Conv-product}
(\pr_1,m) \colon \Conv_{\cH} \to \Gr_{\cH} \times \Gr_{\cH}
\end{equation}
is an isomorphism; in particular, this shows that $\Conv_{\cH}$ is representable by an ind-scheme, which is ind-projective if $\cH$ is parahoric.
In this case, $m$ is ind-projective as well.

This ind-scheme can also be constructed as a twisted product
\begin{equation}\label{eqn:Conv-twisted}
\Conv_\cH \cong \Loop H \times^{\Loop^+\cH} \Loop H/\Loop^+\cH
\end{equation}
(where, as usual, $H=F \otimes_{O_F} \cH$, and we consider \'etale quotients).
From this perspective, given any locally closed subschemes $X, Y \subset \Gr_\cH$ where $Y$ is $\Loop^+\cH$-stable, we can define a locally closed subscheme
\[
X \wttimes Y := (X \times_{\Gr_\cH} \Loop H) \times^{\Loop^+\cH} Y \qquad\subset \Conv_\cH.
\]
For instance, if $\cG$ is as in~\S\ref{ss:special-facets}, given $\lambda,\mu \in \bbX_*(T)_I^+$, we set
\[
\Conv_\cG^{(\lambda,\mu)} := \Gr_{\cG}^\lambda \wttimes \Gr_{\cG}^\mu.
\]
The closures of these schemes are denoted by
\[
\Conv_{\cG}^{\leq (\lambda,\mu)} := \overline{\Conv_\cG^{(\lambda,\mu)}} = \Gr_{\cG}^{\leq \lambda} \wttimes \Gr_{\cG}^{\leq \mu}.
\]
In view of~\eqref{eqn:stratif-Schubert} we have
\begin{equation}
 \label{eqn:stratif-Schubert-Conv}
 | \Conv_{\cG}^{\leq (\lambda,\mu)} | = \bigsqcup_{\substack{\lambda',\mu' \in \bbX_*(T)_I^+ \\ \lambda' \leq \lambda, \, \mu' \leq \mu}} | \Conv_\cG^{(\lambda',\mu')} |.
\end{equation}

Note that in the special case where $\cH$ is abelian, 
the left multiplication action of $\Loop^+\cH$ on $\Gr_\cH$ is trivial. 
We deduce an identification 
\begin{equation}\label{eqn:Conv-ab}
\Conv_\cH \cong \Loop \cH \times^{\Loop^+\cH} \Gr_\cH \cong \Gr_\cH \times \Gr_\cH \qquad\text{if $\cH$ is abelian}.
\end{equation}
We emphasize that this isomorphism is \emph{different} from the one in~\eqref{eqn:Conv-product}.

\subsection{Some quotient stacks}
\label{ss:quotient-stacks}

In this subsection we introduce some stacks that will be used later in our study of sheaves on $\Gr_{\cH}$.

Let $\Hk_{\cH}$ be the Hecke stack (over $\F$) associated with $\Gr_{\cH}$, such that $\Hk_{\cH}(R)$ is the category of triples $(\cE_1, \cE_2, \alpha)$ where $\cE_1, \cE_2$ are $\cH$-bundles on $\Spec(R\pot{t})$ and $\alpha$ is an isomorphism between their restrictions to $\Spec(R\rpot{t})$. We have a natural morphism
\[
h \colon \Gr_{\cH} \to \Hk_{\cH}
\]
sending $(\cE,\alpha)$ to $(\cE_0,\cE,\alpha)$ where $\cE_0$ is the trivial $\cH$-bundle. This morphism
is an $\Loop^+ \cH$-bundle; as in~\cite[Proposition~6.1.7]{fs} it factors through an isomorphism
\begin{equation}\label{eqn:Hk-quotient}
[\Loop^+ \cH \backslash \Gr_{\cH}]_{\et} \simto \Hk_{\cH}
\end{equation}
where the left-hand side is the \'etale quotient stack of $\Gr_{\cH}$ by the action of $\Loop^+ \cH$. (In particular, the left-hand side identifies with the fppf quotient stack of $\Gr_{\cH}$ by the action of $\Loop^+ \cH$.)

By the proof of its representability (see e.g.~\cite{richarz-basics}, or \cite[Lemma~A.5]{rs-intersection} for a more general result), the ind-scheme $\Gr_{\cH}$ admits a presentation $\Gr_{\cH} = \colim_{i \geq 0} X_i$ where each $X_i$ is an $\F$-scheme of finite type such that the action of $\Loop^+ \cH$ on $\Gr_{\cH}$ factors through an action on $X_i$, and such that the latter action factors though an action of $\Loop^+_{n_i} \cH$ for some $n_i \geq 0$. One can also assume that if $i \leq j$ we have $n_i \leq n_j$, so that the composition $\Loop^+ \cH \to \Loop^+_{n_j} \cH$ factors through the quotient morphism $\Loop^+_{n_i} \cH \to \Loop^+_{n_j} \cH$.
Then we have
\begin{equation}
\label{eqn:Hk-colim-quotients}
\Hk_{\cH} = \colim_i [\Loop^+ \cH \backslash X_i]_{\et}, \quad [\Loop^+ \cH \backslash X_i]_{\et} = \lim_{n \geq n_i} \, [\Loop^+_{n} \cH \backslash X_i]_{\et}.
\end{equation}
(Here again, each \'etale quotient stack identifies with the corresponding fppf quotient stack. Note that each $[\Loop^+_{n} \cH \backslash X_i]_{\et}$ is an algebraic stack over $\F$, see~\cite[\href{https://stacks.math.columbia.edu/tag/06FI}{Tag 06FI}]{stacks-project}, which is moreover of finite type.)

We will also consider the Hecke convolution stack $\HkConv_{\cH}$ over $\F$, defined in such a way that $\HkConv_{\cH}(R)$ is the category of tuples $(\cE_1,\cE_2,\cE_3,\alpha,\beta)$ where $\cE_1,\cE_2,\cE_3$ are $\cH$-bundles over $\Spec(R\pot{t})$, $\alpha$ is an isomorphism between the restrictions of $\cE_1$ and $\cE_2$ to $\Spec(R\rpot{t})$, and $\beta$ is an isomorphism between the restrictions of $\cE_2$ and $\cE_3$ to $\Spec(R\rpot{t})$. 
We have a canonical morphism
\[
\tilde h \colon \Conv_{\cH} \to \HkConv_{\cH}
\]
sending $(\cE_1,\cE_2,\alpha,\beta)$ to $(\cE_0,\cE_1,\cE_2,\alpha,\beta)$ where $\cE_0$ is the trivial $\cH$-bundle. 
This morphism is an $\Loop^+ \cH$-bundle, and factors through an isomorphism
\[
[\Loop^+ \cH \backslash \Conv_{\cH}]_{\et} \simto \HkConv_{\cH}.
\]
We also have descriptions of this stack similar to those in~\eqref{eqn:Hk-colim-quotients}.

The morphisms $\pr_1 \colon \Conv_{\cH} \to \Gr_{\cH}$ and $m \colon \Conv_\cH \to \Gr_\cH$ have analogues at the level of quotient stacks, which will also be denoted
\[
\pr_1 \colon \HkConv_{\cH} \to \Hk_{\cH}, \qquad
m \colon \HkConv_{\cH} \to \Hk_{\cH},
\]
and which send $(\cE_1,\cE_2,\cE_3,\alpha,\beta)$ to $(\cE_1, \cE_2, \alpha)$ and $(\cE_1,\cE_3,\beta \circ \alpha)$, respectively. We also have morphisms
\[
p \colon \HkConv_{\cH} \to \Hk_{\cH} \times \Hk_{\cH},
\qquad
p \colon \Conv_\cH \to \Gr_\cH \times \Hk_\cH,
\]
where the former sends $(\cE_1,\cE_2,\cE_3,\alpha,\beta)$ to $\bigl( (\cE_1,\cE_2,\alpha), (\cE_2,\cE_3,\beta) \bigr)$, and the latter is defined similarly.  These maps fit into cartesian squares
\begin{equation}
\label{eqn:ppr-cartesian}
\begin{tikzcd}[column sep=small]
\HkConv_\cH \ar[r, "p"] \ar[d, "\pr_1"' ] & \Hk_\cH \times \Hk_\cH \ar[d] \\
\Hk_\cH \ar[r] & \Hk_\cH \times [\Loop^+\cH\backslash \mathrm{pt}]_\et,
\end{tikzcd}
\quad
\begin{tikzcd}[column sep=small] 
\Conv_\cH \ar[r, "p"] \ar[d, "\pr_1"' ] & \Gr_\cH \times \Hk_\cH \ar[d] \\
\Gr_\cH \ar[r] & \Gr_\cH \times [\Loop^+\cH\backslash \mathrm{pt}]_\et,
\end{tikzcd}
\end{equation}
where the bottom horizontal arrows are induced by the obvious morphism $\mathrm{pt} \to [\Loop^+\cH\backslash \mathrm{pt}]_\et$.

In the case where $\cH$ is abelian, we have a counterpart to~\eqref{eqn:Conv-ab}: an isomorphism
\begin{equation}\label{eqn:HkConv-ab}
\HkConv_\cH \cong \Hk_\cH \times \Gr_\cH \qquad\text{if $\cH$ is abelian}.
\end{equation}
In this case, the following diagram commutes:
\begin{equation}\label{eqn:HkConv-ab-commute}
\begin{tikzcd}[row sep=tiny]
& \Gr_\cH \times \Gr_\cH \ar[r, "\sim"', "\eqref{eqn:Conv-ab}"] \ar[dd, "h \times \id"] \ar[dl, "h \times h"'] 
  & \Conv_\cH \ar[dd, "\tilde h"] \ar[r, "m"] 
  & \Gr_\cH \ar[dd, "h"] \\
\Hk_\cH \times \Hk_\cH \\
& \Hk_\cH \times \Gr_\cH \ar[r, "\sim"', "\eqref{eqn:HkConv-ab}"] \ar[ul, "\id \times h"'] 
  & \HkConv_\cH \ar[r, "m"] \ar[ull, bend left=35, "p"]
  & \Hk_\cH.
\end{tikzcd}
\end{equation}

In later sections we will consider the morphisms $m$, $p$, $h$ and $\tilde h$ for the group $\cG$, but also for some subgroups. 
To avoid confusions, when necessary we will add a subscript to these notations to indicate which group is considered.

\section{Semi-infinite orbits}
\label{sec:semi-infinite}

\subsection{Attractors and fixed points in Bruhat--Tits theory}
\label{ss:group-theory}

Recall that for any base scheme $S$, given an $S$-scheme $X$ endowed with an action of the multiplicative group $\mathbb{G}_{\mathrm{m}, S}$ we have associated \emph{attractor} and \emph{fixed points} functors from $S$-schemes to sets, denoted $X^+$ and $X^0$; see e.g.~\cite{richarz-Gm}. These functors are not always representable, but they are under reasonable assumptions that will be satisfied in all the examples we consider below. We also have natural morphisms of functors $X^0 \to X^+$ and $X^+ \to X$.  The former admits a left inverse $X^+ \to X^0$ (the ``limit'' morphism). If $X$ is separated over $S$ then the morphism $X^+ \to X$ is a monomorphism, see~\cite[Remark~1.19(ii)]{richarz-Gm}.

Let us now consider the setting of Section~\ref{sec:aff-Grass}, fix a cocharacter
\[
\lambda\colon \mathbb{G}_{{\mathrm{m}},F} \to A \subset G,
\]
and then make $\mathbb{G}_{{\mathrm{m}},F}$ act on $G$ by conjugation via this cocharacter.  Denote the attractor and fixed-points functors by
\[
P_\lambda := G^+ \qquad\text{and}\qquad
M_\lambda := G^0,
\]
respectively.  These functors are represented by smooth subgroup schemes of $G$.  Indeed, by~\cite[Proposition~2.2.9]{CGP10}, $P_\lambda$ is a parabolic subgroup of $G$ containing $A$, and $M_\lambda$ is a Levi factor of $P_\lambda$.  Moreover, every parabolic subgroup of $G$ can be described in this way as an attractor.

\begin{rmk}
\phantomsection
\label{rmk:parabolics-cocharacters}
\begin{enumerate}
\item
\label{it:parabolics-cocharacters}
To amplify the preceding comment: if $P \subset G$ is a parabolic subgroup containing $A$, then $P = P_\lambda$ for some $\lambda \in \bbX_*(A)$.  (This follows from the conjugacy of maximal $F$-split tori in $P$: see~\cite[Lemme~4.6 and Proposition~11.6]{bot}.)
In this case, $M_\lambda$ is the unique Levi factor of $P$ containing $A$. The parabolic subgroup $P_\lambda$ is standard with respect to $B$ (i.e., contains $B$) if and only if $\lambda$ is dominant (with respect to $\Phi^+_\abs \subset \bbX^*(T)$).
\item
\label{it:opposite-parabolic}
If one replaces $\lambda$ by the opposite cocharacter $-\lambda$, then we obtain the same Levi subgroup: $M_{-\lambda}=M_\lambda$, and the parabolic subgroup $P_{-\lambda}$ is the opposite parabolic subgroup such that $P_\lambda \cap P_{-\lambda} = M_\lambda$.
\end{enumerate}
 \end{rmk}

Next, we observe that $\lambda$ extends uniquely to a cocharacter
\[
\lambda_{O_F} \colon \mathbb{G}_{\mathrm{m},O_F} \to \cA \subset \cG
\]
of $O_F$-groups by the universal property of the N\'eron model $\cA$ (or, in a more elementary way, by the equality
$\Hom_{O_F\mhyphen\mathrm{gps}}(\mathbb{G}_{\mathrm{m}, O_F},\cA)=\Hom_{F\mhyphen\mathrm{gps}}(\mathbb{G}_{\mathrm{m}, F},A)$, 
which follows from the facts that $\cA$ is a split torus and $\Spec(O_F)$ is connected). 
We can therefore consider, as above, the action of $\mathbb{G}_{\mathrm{m}, O_F}$ on $\cG$ induced by $\lambda_{O_F}$, as well as the corresponding attractors and fixed points
\[
\cP_\lambda := \cG^+
\qquad\text{and}\qquad \cM_\lambda := \cG^0
\]
over $O_F$.  According to~\cite[Lemma 4.5]{HainesRicharz_TestFunctions} (and its proof), $\cP_\lambda$ and $\cM_\lambda$ are both smooth affine group schemes over $O_F$, and they coincide with the scheme-theoretic closures of their counterparts in $G$:
\begin{equation}
\label{eqn:parabolics-cocharacters-2}
\cP_\lambda = \overline{P_\lambda}
\qquad\text{and}\qquad
\cM_\lambda = \overline{M_\lambda}.
\end{equation}
In particular, the group schemes $\cP_\lambda$ and $\cM_\lambda$ depend only on $P_\lambda$ and $M_\lambda$, and not on the choice of the cocharacter $\lambda$.  The natural limit morphism $\cP_\lambda \to \cM_\lambda$ is split by the embedding $\cM_\lambda \to \cP_\lambda$, and its kernel $\cU_\lambda$ is a smooth affine $O_F$-group scheme with connected unipotent geometric fibers.

Since $M_\lambda$ contains $A$, we can consider the apartment $\mathscr{A}(M_\lambda, A, F)$ in the building $\mathscr{B}(M_\lambda, F)$ of $M_\lambda$. 
This identifies with $\mathscr{A}(G, A, F)$, hence we can consider the unique facet $\fa_{M_\lambda}$ in $\mathscr{A}(M_\lambda, A, F)$ containing $\fa$. In view of~\cite[Appendix~1]{richarz}, we know that:
\begin{itemize}
\item $\cM_\lambda$ is the parahoric group scheme attached to the facet $\fa_{M_\lambda} \subset \mathscr{B}(M_\lambda,F)$;
\item the facet $\fa_{M_\lambda}$ is special.
\end{itemize}
In other words, $\cM_\lambda$ matches the set-up of~\S\ref{ss:special-facets}. (As usual, these data depend only on $M_\lambda$, and not on the choice of $\lambda$.)

\subsection{Attractors and fixed points on the affine Grassmannian}
\label{ss:attractor-fixed-pts}

We continue with the setting of~\S\ref{ss:group-theory}.
The composition 
\begin{equation}\label{eqn:gm-loop-g}
\mathbb{G}_{{\mathrm{m}},\F}\subset \Loop^+\mathbb{G}_{\mathrm{m}, O_F} \xrightarrow{\Loop^+\lambda_{O_F}} \Loop^+\cG
\end{equation}
and the $\Loop^+\cG$-action on $\Gr_{\cG}$
provide a $\mathbb{G}_{\mathrm{m},\F}$-action on the ind-scheme $\Gr_{\cG}$. 
It is clear by construction that this action preserves Schubert varieties. 
Furthermore, by~\cite[Lemma~5.3]{HainesRicharz_TestFunctions} it is Zariski locally linearizable in the sense of~\cite{richarz-Gm}, and thus the fixed points $\Gr_{\cG}^0$ and the attractor $\Gr_{\cG}^+$ (defined in the obvious way, generalizing the definition for schemes) are representable by ind-schemes (see~\cite[Theorem~2.1]{HainesRicharz_TestFunctions}). Moreover the natural morphism $\Gr_{\cG}^0 \to \Gr_{\cG}$ is representable by a closed immersion; the natural morphism $\Gr_{\cG}^+ \to \Gr_{\cG}$ is bijective (but not a homeomorphism); and its restriction to each connected component of $\Gr_{\cG}^+$ is representable by a locally closed immersion.
Below we identify precisely this sub-ind-scheme, confirming the expectation in~\cite[Remark 4.8]{HainesRicharz_TestFunctions}.

The morphisms of group schemes
\[
\cM_\lambda \to \cP_\lambda \to \cG
\]
induce $\mathbb{G}_{\mathrm{m},\F}$-equivariant morphisms of the corresponding affine Grassmannians
\[
\Gr_{\cM_\lambda} \to \Gr_{\cG} \quad \text{and} \quad \Gr_{\cP_\lambda} \to \Gr_{\cG}.
\]
Since $\mathbb{G}_{\mathrm{m},\F}$ acts trivially on $\Gr_{\cM_\lambda}$, the former factors through a morphism
\begin{equation}
\label{eqn:fixed-pts-map}
\Gr_{\cM_\lambda} \to \Gr_{\cG}^0.
\end{equation}
Similarly, since the action of $\mathbb{G}_{\mathrm{m},O_F}$ on $\cP_\lambda$ extends to an action of the monoid $\mathbb{A}^1_{O_F}$, the action of $\mathbb{G}_{\mathrm{m},\F}$ on $\Gr_{\cP_\lambda}$ extends to an action of $\mathbb{A}^1_\F$, which implies that the monomorphism $(\Gr_{\cP_\lambda})^+ \to \Gr_{\cP_\lambda}$ is an isomorphism. We deduce a morphism
\begin{equation}
\label{eqn:attractor-map}
\Gr_{\cP_\lambda} = (\Gr_{\cP_\lambda})^+ \to \Gr_{\cG}^+.
\end{equation}

\begin{prop}
\label{prop:attractors-fixed-pts}
	The morphisms~\eqref{eqn:fixed-pts-map} and~\eqref{eqn:attractor-map} are isomorphisms.
\end{prop}

\begin{rmk}
In particular, Proposition~\ref{prop:attractors-fixed-pts} and~\eqref{eqn:parabolics-cocharacters-2} show that $\Gr_{\cG}^+$, resp.~$\Gr_{\cG}^0$, only depends on $\cG$ and $P_\lambda$, resp.~$M_\lambda$, and not on the actual choice of $\lambda$.
\end{rmk}

For the proof of this proposition we will need a preliminary result on tangent spaces (see~\S\ref{ss:loop-gps}). 

\begin{lem}
\label{lem:tangent-spaces}
Let $e \in \Gr_{\cP_\lambda}(\F)$ be the base point, and denote similarly its image under~\eqref{eqn:attractor-map}. Then~\eqref{eqn:attractor-map} induces an isomorphism
\[
T_e \Gr_{\cP_\lambda} \simto T_e \Gr_{\cG}^+.
\]
Similarly, if we denote by $e' \in \Gr_{\cM_\lambda}(\F)$ the base point and its image under~\eqref{eqn:fixed-pts-map}, then~\eqref{eqn:fixed-pts-map} induces an isomorphism
\[
T_{e'} \Gr_{\cM_\lambda} \simto T_{e'} \Gr_{\cG}^0.
\]
\end{lem}

\begin{proof}
We give the proof of the first claim, and leave that of the second one to the reader.
	As the formation of attractors commutes with that of tangent spaces at fixed points (see~\cite[Proposition~1.4.11(vi)]{Dr13}), we have
	\[
	T_e \Gr_{\cG}^+ = (T_e \Gr_{\cG})^+;
	\]
	moreover the right-hand side identifies with the subspace in $T_e \Gr_{\cG}$ spanned by weight vectors of nonnegative weight for the action of $\mathbb{G}_{\mathrm{m},\F}$ induced by the action on $\Gr_{\cG}$. To prove the desired claim, we therefore have to show that~\eqref{eqn:attractor-map} induces an isomorphism
	\begin{equation}\label{equation.tangent.spaces}
	(T_e\Gr_{\cG})^+=T_e\Gr_{\cP_\lambda}.
	\end{equation}

By Lemma~\ref{lem:tangent-space-Fl}, we have isomorphisms
	\[
	T_e\Gr_{\cG} = \Lie(\Loop G) / \Lie(\Loop^+\cG), \qquad
	T_e \Gr_{\cP_\lambda} = \Lie(\Loop P_\lambda) / \Lie(\Loop^+ \cP_\lambda).
	\]
	The first isomorphism implies that
	\[
	(T_e\Gr_{\cG})^+ = \bigl( \Lie(\Loop G) \bigr)^+ / \bigl( \Lie(\Loop^+\cG) \bigr)^+
	\]
	where $(\Lie(\Loop G))^+$ is the subspace of $\Lie(\Loop G)$ spanned by weight vectors of nonnegative weights, and similarly for $(\Lie(\Loop^+\cG))^+$. Hence proving~\eqref{equation.tangent.spaces} amounts to proving an identification
	\begin{equation}\label{equation.tangent.spaces2}
	\bigl( \Lie(\Loop G) \bigr)^+ / \bigl( \Lie(\Loop^+\cG) \bigr)^+ = \Lie(\Loop P_\lambda) / \Lie(\Loop^+ \cP_\lambda).
	\end{equation}
	
	By \cite[\S\S 3.8.1--3.8.2, \S 4.6.2]{BT84}, the product map $(u^-,t,u^+)\mapsto u^-\cdot t\cdot u^+$ gives an open immersion of $O_F$-schemes 
	\begin{equation}
	\label{eqn:big-cell-cG}
	\prod_{\alpha\in \Phi^{\mathrm{nd},-}}\cU_{\alpha}\times \cT\times \prod_{\alpha\in \Phi^{\mathrm{nd}, +}}\cU_{\alpha}\to \cG,
	\end{equation}
	where $\Phi^{\mathrm{nd}} \subset \Phi$ is the system of non-divisible, relative roots for $(G,A)$, and $\Phi^{\mathrm{nd}, +} = \Phi^+ \cap \Phi^{\mathrm{nd}}$, $\Phi^{\mathrm{nd}, -} = -\Phi^{\mathrm{nd}, +}$.
	(Here, for any $\alpha$, $\cU_\alpha$ is the ``root group'' associated with $\alpha$.)
	After applying the loop functor $\Loop$, resp.~positive loop functor $\Loop^+$, the product map is still formally \'etale (and therefore identifies the tangent spaces), so we obtain isomorphisms
	\begin{align}
	\label{eqn:Lie-LG-+}
	\Lie(\Loop G) &= \Lie(\Loop\cT)\oplus\bigoplus_{\alpha\in \Phi^\mathrm{nd}}\Lie(\Loop\cU_\alpha), \\
	\label{eqn:Lie-LG-+2}
	\Lie(\Loop^+\cG) &= \Lie(\Loop^+\cT)\oplus\bigoplus_{\alpha\in \Phi^\mathrm{nd}}\Lie(\Loop^+\cU_\alpha).
	\end{align}
	Since $\mathbb{G}_{\mathrm{m},\F}$ acts on $\Lie(\Loop\cU_\alpha)$ with weights in $\{ \langle \alpha,\lambda\rangle, 2\langle \alpha,\lambda\rangle \}$,
	passing to attractors yields 
		\begin{align*}
	\Lie(\Loop G)^+ &= \Lie(\Loop\cT)\oplus\bigoplus_{\alpha\in \Phi_{\lambda}^\mathrm{nd}}\Lie(\Loop\cU_\alpha), \\
	\Lie(\Loop^+\cG)^+ &= \Lie(\Loop^+\cT)\oplus\bigoplus_{\alpha\in \Phi_{\lambda}^\mathrm{nd}}\Lie(\Loop^+\cU_\alpha),
	\end{align*}
	where $\Phi_{\lambda}^\mathrm{nd} = \{\alpha \in \Phi^\mathrm{nd} \mid \langle \lambda,\alpha \rangle \geq 0\}$. 
	
On the other hand, passing to attractors in the open immersion~\eqref{eqn:big-cell-cG}, by~\cite[Corollary~2.3]{HainesRicharz_TestFunctions} we obtain an open immersion
\[
	\prod_{\alpha\in \Phi^{\mathrm{nd},-}_{\lambda}}\cU_{\alpha}\times \cT\times \prod_{\alpha\in \Phi^{\mathrm{nd}, +}_{\lambda}}\cU_{\alpha}\to \cP_\lambda,
\]
where $\Phi^{\mathrm{nd},\pm}_{\lambda} = \Phi^{\mathrm{nd},\pm} \cap \Phi_{\lambda}^\mathrm{nd}$. As above we deduce identifications
		\begin{align}
		\label{eqn:Lie-LP}
	\Lie(\Loop P_\lambda) &= \Lie(\Loop\cT)\oplus\bigoplus_{\alpha\in \Phi_{\lambda}^\mathrm{nd}}\Lie(\Loop\cU_\alpha), \\
	\label{eqn:Lie-LP2}
	\Lie(\Loop^+\cP_\lambda) &= \Lie(\Loop^+\cT)\oplus\bigoplus_{\alpha\in \Phi_{\lambda}^\mathrm{nd}}\Lie(\Loop^+\cU_\alpha).
	\end{align}
Comparing~\eqref{eqn:Lie-LG-+}--\eqref{eqn:Lie-LG-+2} with~\eqref{eqn:Lie-LP}--\eqref{eqn:Lie-LP2} we deduce the identification~\eqref{equation.tangent.spaces2}, which finishes the proof.
\end{proof}

\begin{proof}[Proof of Proposition~\ref{prop:attractors-fixed-pts}]
In this proof we will use the following property of tangent spaces. Let $X$ be an $\F$-scheme of finite type. For any $\F$-algebra $R$ and any $y \in X(R) = (R \otimes_{\F} X)(R)$ we have the associated conormal sheaf $\omega_y$, see~\cite[II, \S 4, 3.1]{dg}, which is an $R$-module and has the property that if $J \subset R$ is an ideal such that $J^2=0$ there is a canonical identification between $\Hom_{R}(\omega_y, J)$ and the set of points $y' \in X(R)$ whose image in $X(R/J)$ is the image of $y$; see~\cite[II, \S 4, 3.2]{dg}. In particular, if $x \in X(\F)$, $\omega_x$ identifies with the dual of the finite-dimensional $\F$-vector space $T_x X$, and if we denote by $x_R$ the image of $x$ in $X(R)$ we have $\omega_{x_R} = R \otimes_{\F} \omega_x$, see~\cite[I, \S 4, 1.4]{dg}. For any $\F$-algebra $R$ and any ideal $J \subset R$ such that $J^2=0$, we deduce an identification between $J \otimes_\F T_x X$ and the set of points $y \in X(R)$ whose image in $X(R/J)$ is (the image of) $x$.
This property of course extends to ind-schemes of ind-finite type over $\F$ by passing to the colimit of tangent spaces in any presentation.

Below we will consider various local artinian $\F$-algebras. If $R$ is such an algebra, we will denote by $\mathrm{rad}(R)$ its Jacobson radical, i.e.~its unique maximal ideal. We will also denote by $\mathrm{n}(R)$ the minimal positive integer $n$ such that $\mathrm{rad}(R)^n = 0$. Note that if the field $R/\mathrm{rad}(R)$ is algebraically closed, then $R$ is a strictly henselian local ring. 
Using Lemma~\ref{lem:quotient-etale} we deduce that in this case we have
\begin{equation}
\label{eqn:points-Fl}
\Gr_{\cP_\lambda}(R) = \Loop P_\lambda (R)/ \Loop^+ \cP_\lambda(R).
\end{equation}

	By~\cite[Proposition~4.7(i)]{HainesRicharz_TestFunctions}, 
	our maps are closed immersions.
	To finish the proof, in view of~\cite[Lemma 8.6 and Remark~8.7]{HLR} it therefore suffices to check that for any local artinian $\F$-algebra $R$ such that $R/\mathrm{rad}(R)$ is algebraically closed, every $R$-valued point $x_R$ of $\Gr_{\cG}^0$, resp.~$\Gr_{\cG}^+$, already lies in $\Gr_{\cM_\lambda}(R)$, resp.~$\Gr_{\cP_\lambda}(R)$. 
	We shall treat the case of the attractor and leave the case of fixed points to the reader. 
	
We proceed by induction on $\mathrm{n}(R)$. If $\mathrm{n}(R) = 1$, then $R$ is an algebraically closed field. In this case, using base change we can assume that $R=\F$; this case is treated in~\cite[Proposition~4.7(iii)]{HainesRicharz_TestFunctions}.
	
	Now assume that $\mathrm{n}(R) >1$, and let $J \subset R$ be an ideal such that $J^2=0$ and $\mathrm{n}(R/J) < \mathrm{n}(R)$. (For instance, one can take $J=\mathrm{rad}(R)^{\mathrm{n}(R)-1}$.) Consider the $R/J$-valued point $x_{R/J}$ of $\Gr_{\cG}^+$ associated with $x_R$. 
	By induction, the point $x_{R/J}$ lies in $\Gr_{\cP_\lambda}(R/J)$. 
	By~\eqref{eqn:points-Fl}, we can lift $x_{R/J}$ to an $R/J$-valued point $y_{R/J}$ of $\Loop P_\lambda$, and then further to an $R$-valued point $y_R$ of $\Loop P_\lambda$ by formal smoothness. 
	In other words, replacing $x_R$ by $y_R^{-1}x_R$ we may assume that $x_{R/J}$ is the base point $e$ of $\Gr_{\cG}^+$. By the reminder at the beginning of the proof, $x_R$ corresponds to an element in $T_e(\Gr_{\cG}^+) \otimes_\F J$. By Lemma~\ref{lem:tangent-spaces} and the same considerations, there exists an $R$-point of $\Gr_{\cP_\lambda}$ whose image in $\Gr_{\cP_\lambda}(R/J)$ is $e$ and whose image under~\eqref{eqn:attractor-map} is $x_R$, which finishes the proof.
\end{proof}

We conclude this subsection with a remark on the convolution schemes $\Conv_\cG$, $\Conv_{\cM_\lambda}$ and $\Conv_{\cP_\lambda}$ (see~\S\ref{ss:convolution-schemes}).  Since $\Loop^+\cG$ acts on $\Conv_\cG$ on the left, the homomorphism~\eqref{eqn:gm-loop-g} yields a $\mathbb{G}_{\mathrm{m},\F}$-action on $\Conv_\cG$, and one can consider the attractor and fixed-point ind-schemes for this action.  The same reasoning that led to~\eqref{eqn:fixed-pts-map} and~\eqref{eqn:attractor-map} lets us construct maps
\begin{gather}
\Conv_{\cM_\lambda} \to \Conv_\cG^0, \label{eqn:fixed-pts-conv} \\
\Conv_{\cP_\lambda} = (\Conv_{\cP_\lambda})^+ \to \Conv_\cG^+. \label{eqn:attractor-conv}
\end{gather}
In view of the isomorphism~\eqref{eqn:Conv-product} (which is $\Loop^+\cG$-equivariant, and hence $\mathbb{G}_{\mathrm{m},\F}$-equivariant), we obtain the following immediate consequence of Proposition~\ref{prop:attractors-fixed-pts}.

\begin{prop}
\label{prop:attractors-fixed-pts-conv}
	The morphisms~\eqref{eqn:fixed-pts-conv} and~\eqref{eqn:attractor-conv} are isomorphisms.
\end{prop}

\subsection{Parabolic subgroups and cartesian diagrams}
\label{ss:cartesian-diagram}

Let us now consider two parabolic subgroups $P,P' \subset G$ containing $A$ and such that $P \subset P'$. If we denote by $M$ and $M'$ their Levi factors containing $A$, then we have $M \subset M'$. Moreover, $P \cap M'$ is a parabolic subgroup of the reductive group $M'$ containing $A$, and $M$ is its Levi factor containing $A$. Let $\cP$, $\cP'$, $\cM$, and $\cM'$ be their scheme-theoretic closures inside $\cG$ (cf.~\eqref{eqn:parabolics-cocharacters-2}).

Point~\eqref{it:cart-diag-3} of the following lemma involves fiber products of ind-schemes. For this notion, see~\cite[Lemma~1.10]{richarz-basics}.

\begin{lem}
\phantomsection\label{lem:cart-diag}
\begin{enumerate}
\item
\label{it:cart-diag-1}
The intersection $\cP \cap \cM'$ is the scheme-theoretic closure of the group scheme $P \cap M'$ inside $\cM'$, and is smooth over $O_F$.
\item
\label{it:cart-diag-2}
The following commutative square of group schemes is cartesian:
\[
\begin{tikzcd}
\cP \ar[r] \ar[d] & \cP \cap \cM' \ar[d] \\ \cP' \ar[r] & \cM'.
\end{tikzcd}
\]
\item
\label{it:cart-diag-3}
The following commutative squares of ind-schemes are cartesian:
\[
\begin{tikzcd}
\Gr_{\cP} \ar[r] \ar[d] & \Gr_{\cP \cap \cM'} \ar[d] \\ \Gr_{\cP'} \ar[r] & \Gr_{\cM'},
\end{tikzcd}
\qquad
\begin{tikzcd}
\Conv_{\cP} \ar[r] \ar[d] & \Conv_{\cP \cap \cM'} \ar[d] \\ \Conv_{\cP'} \ar[r] & \Conv_{\cM'}.
\end{tikzcd}
\]
\end{enumerate}
\end{lem}

\begin{proof}
\eqref{it:cart-diag-1} 
Let $\lambda \in \bbX_*(A)$ be a cocharacter such that  $P=P_\lambda$ (see Remark~\ref{rmk:parabolics-cocharacters}). Then from the definitions one sees that $\cP \cap \cM'$ is the attractor associated with the $\mathbb{G}_{\mathrm{m},O_F}$-action on $\cM'$ defined by $\lambda$, so our claim is a special case of~\eqref{eqn:parabolics-cocharacters-2}.

\eqref{it:cart-diag-2} The morphism $\cP \to \cP'$ is a closed immersion, and hence so is the induced morphism  $\cP \to \cP' \times_{\cM'} (\cP \cap \cM')$; in other words, the associated morphism
\[
\scO(\cP' \times_{\cM'} (\cP \cap \cM')) \to \scO(\cP)
\]
is surjective.
On the other hand this morphism becomes an isomorphism after tensor product with $F$, and the fiber product $\cP' \times_{\cM'} (\cP \cap \cM')$ is flat over $\Spec(O_F)$. (Indeed the projection $\cP' \to \cM'$ is smooth, so this fiber product is smooth over $\cP \cap \cM'$, which is itself smooth over $\Spec(O_F)$ by~\eqref{it:cart-diag-1}.) This morphism is therefore an isomorphism, which finishes the proof.

\eqref{it:cart-diag-3}
Let us first consider the left-hand diagram.
If $R$ is an $\F$-algebra, $\Gr_{\cP}(R)$ classifies pairs consisting of an $\cP$-torsor on $\Spec(R\pot{t})$ and a trivialization on $\Spec(R\rpot{t})$. On the other hand, by~\eqref{it:cart-diag-1},
\[
(\Gr_{\cP'} \times_{\Gr_{\cM'}} \Gr_{\cP \cap \cM'})(R) = \Gr_{\cP'}(R) \times_{\Gr_{\cM'}(R)} \Gr_{\cP \cap \cM'}(R)
\]
parametrizes an $\cP'$-torsor on $\Spec(R\pot{t})$ with a trivialization on $\Spec(R\rpot{t})$, an $\cP \cap \cM'$-torsor on $\Spec(R\pot{t})$ with a trivialization on $\Spec(R\rpot{t})$, and an isomorphism between the induced $\cM'$-bundles and their trivializations. Now from~\eqref{it:cart-diag-2} we deduce
that the datum of an $\cP$-torsor is equivalent to that of an $\cP'$-torsor, an $\cP \cap \cM'$-torsor, and an isomorphism between the induced $\cM'$-bundles. The desired claim follows.

The fact that the right-hand diagram is cartesian is immediate from the similar property of the left-hand diagram and~\eqref{eqn:Conv-product}.
\end{proof}

\subsection{Semi-infinite orbits}
\label{ss:semi-inf-orbits}

Let us now study further the case where the cocharacter in $\bbX_*(A)$ is such that the attractor and fixed point sets in $G$ are given by
\[
G^+ = B \qquad\text{and}\qquad G^0 = T.
\]
(This is indeed possible thanks to Remark~\ref{rmk:parabolics-cocharacters}.)  Let $\cB$ be the scheme-theoretic closure of $B$ in $\cG$.  Then, by~\eqref{eqn:parabolics-cocharacters-2}, $\cB$ is smooth and we have
\[
\cG^+ = \cB \qquad\text{and}\qquad \cG^0 = \cT.
\]
It is well known (see e.g.~\cite[\S 3.b]{pr}) that the underlying topological space of $\Gr_\cT$ is discrete, with
\begin{equation}
\label{eqn:GrT-X}
|\Gr_\cT| = \bbX_*(T)_I
\end{equation}
(see~\eqref{eqn:Gr-torus}).
Since the morphism $\Gr_{\cG}^+ \to \Gr_{\cG}^0$ induces a bijection on the sets of connected components (see~\cite[Proposition~1.17]{richarz-Gm} and~\cite[Theorem~2.1]{HainesRicharz_TestFunctions}), we deduce a bijection between the set of connected components of $\Gr_{\cG}^+$, i.e.~of $\Gr_{\cB}$, and $\bbX_*(T)_I$.

For $\lambda \in \bbX_*(T)_I$, we will denote by $\rmS_\lambda$ the associated connected component of $\Gr_{\cG}^+$. Then the natural map $\rmS_\lambda \to \Gr_{\cG}$ is representable by a locally closed immersion, and the natural map
\[
\bigsqcup_{\lambda \in \bbX_*(T)_I} \rmS_\lambda \to \Gr_{\cG}
\]
is bijective. If we denote by $\rmS_{\leq \lambda}$ the ind-schematic closure of $\rmS_{\lambda}$ inside $\Gr_{\cG}$, then we have
\begin{equation}
\label{eqn:closure-S}
| \rmS_{\leq \lambda} | = \bigsqcup_{\mu \leq \lambda} | \rmS_\mu |.
\end{equation}
(This property is proved in~\cite[Proposition 5.4]{aglr} for Witt vector affine Grassmannians; the same proof goes through in our present setting.)
Choosing a presentation $\Gr_{\cG} = \mathrm{colim}_i X_i$ by $\Loop^+ \cG$-stable projective $k$-schemes (see~\S\ref{ss:quotient-stacks}), for any $i$ we have
\begin{equation}
\label{eqn:attractor-X_i}
| \Gr_{\cG}^+ \times_{\Gr_{\cG}} X_i | = \bigsqcup_{\lambda \in \bbX_*(T)_I^+} | \rmS_\lambda \times_{\Gr_{\cG}} X_i |,
\end{equation}
and $\Gr_{\cG}^+ \times_{\Gr_{\cG}} X_i$ identifies canonically with the attractor for the induced $\mathbb{G}_{\mathrm{m},\F}$-action on $X_i$. The intersection $\rmS_\lambda \times_{\Gr_{\cG}} X_i$ is a locally closed subscheme of $X_i$ (in particular, a scheme of finite type by~\cite[Example~3.45]{goertz-wedhorn}) and it is empty unless $t^\lambda \in X_i$, which happens only for a finite number of $\lambda$'s.

If $\mu \in \bbX_*(T)_I$ and $\lambda \in \bbX_*(T)^+_I$ then we can consider the intersection
\[
\rmS_\mu \cap \Gr_{\cG}^{\leq \lambda},
\]
a locally closed subscheme of the projective scheme $\Gr_{\cG}^{\leq \lambda}$, which is in particular an $\F$-scheme of finite type.
For the statement of the next lemma, recall that $\bbX_*(T)^+_I$ is a system of representatives for $\bbX_*(T)_I/W_0$, see Lemma~\ref{lem:dominant-coweights}\eqref{it:dominant-coweights-orbits}.

\begin{lem}
\label{lem:dim-estimate}
If $\mu \in \bbX_*(T)_I$ and $\lambda \in \bbX_*(T)^+_I$, the scheme
\[
\rmS_{\mu} \cap \Gr_{\cG}^{\leq \lambda}
\]
is nonempty if and only if the only element $\mu' \in (W_0 \cdot \mu) \cap \bbX_*(T)_I^+$ satisfies $\mu' \leq \lambda$. In this case, this scheme is affine and equidimensional of dimension $\la \mu +\lambda, \rho \ra$.
\end{lem}

Here we write $\la \mu +\lambda, \rho \ra$ for $\frac{1}{2}\la \mu +\lambda, 2\rho \ra$ (which is an integer in the case considered in the statement).
	Again, this result is proved in~\cite[Lemma 5.5]{aglr} for ramified groups over $p$-adic fields, but the arguments also apply in our current setting.

Of course one can play the same game with the Borel subgroup $B^-$ opposite to $B$ (with respect to $T$), see Remark~\ref{rmk:parabolics-cocharacters}\eqref{it:opposite-parabolic}. The connected components of the associated affine Grassmannian will be denoted $(\rmT_\lambda : \lambda \in \bbX_*(T)_I)$.

\subsection{Semi-infinite orbits on convolution schemes}
\label{ss:semi-inf-orbits-Conv}

We finish this section by explaining how to adapt the discussion of~\S\ref{ss:semi-inf-orbits} to the setting of convolution schemes.  In view of~\eqref{eqn:Conv-product}, the ind-scheme $\Conv_\cT$ is discrete, and the connected components of both $\Conv_\cT$ and $\Conv_\cB$ are in bijection with $\bbX_*(T)_I \times \bbX_*(T)_I$.  However, we will use a labeling of these components that is \emph{not} compatible with~\eqref{eqn:Conv-product}: given $\lambda,\mu \in \bbX_*(T)_I$, we let
\[
\rmS_\lambda \wttimes \rmS_\mu \subset \Conv_\cB
\]
be the connected component that identifies under~\eqref{eqn:Conv-product} with $\rmS_\lambda \times \rmS_{\mu + \lambda} \subset \Gr_\cB \times \Gr_\cB$.  This labeling has the advantage that it makes the statement of the following lemma (which follows from the same considerations as for Lemma~\ref{lem:dim-estimate}) cleaner.

\begin{lem}
\label{lem:dim-estimate-Conv}
If $\mu, \mu' \in \bbX_*(T)_I$ and $\lambda, \lambda' \in \bbX_*(T)^+_I$,
	the scheme
	\[
	(\rmS_{\mu} \wttimes \rmS_{\mu'}) \cap \Conv_{\cG}^{\leq (\lambda,\lambda')}
	\]
	is nonempty if and only if the only elements $\nu \in (W_0 \cdot \mu) \cap \bbX_*(T)_I^+$ and $\nu' \in (W_0 \cdot \mu') \cap \bbX_*(T)_I^+$ satisfy $\nu \leq \lambda$ and $\nu' \leq \lambda'$. In this case, this scheme is affine and equidimensional of dimension $\la \mu + \mu' +\lambda + \lambda', \rho \ra$.
\end{lem}

\section{Sheaves on the Hecke stack}
\label{sec:sheaves-Hecke}

\subsection{Constructible sheaves on the affine Grassmannian and the Hecke stack}
\label{ss:Perv-Gr-Hk}

From now on we fix a prime number $\ell$ invertible in $\F$, and a finite extension $\bK$ of $\Q_\ell$. The ring of integers of $\bK$ will be denoted $\bO$, and the residue field of $\bO$ will be denoted $\bk$. We will use the notation $\Lambda$ to denote one of the rings $\bK$, $\bO$, $\bk$, when the choice does not matter. Below we will use the bounded constructible categories of \'etale $\Lambda$-sheaves on Artin stacks of finite type; see~\S\ref{ss:const} for a brief review of the relevant ingredients, and for references on the construction.

Recall the notation of~\S\ref{ss:quotient-stacks}. For any $i \geq 0$ we have the algebraic stack of finite type $[\Loop^+_{n_i} \cG \backslash X_i]_{\et}$ over $\F$, and we can consider the bounded constructible derived category
\[
 \Dbc([\Loop^+_{n_i} \cG \backslash X_i ]_{\et}, \Lambda).
\]
For any $n \geq n_i$, as in~\cite[Proposition~VI.1.10]{fs}, the quotient of the surjective morphism $\Loop^+_n \cG \to \Loop^+_{n_i} \cG$ is an extension of copies of the additive group $\mathbb{G}_{\mathrm{a},\F}$, so that by standard arguments the pullback functor
\begin{equation}
\label{eqn:pullback-Xi}
 \Dbc([\Loop^+_{n_i} \cG \backslash X_i]_{\et}, \Lambda) \to \Dbc([\Loop^+_{n} \cG \backslash X_i]_{\et}, \Lambda)
\end{equation}
is an equivalence of categories. We can therefore set
\[
 \Dbc([\Loop^+ \cG \backslash X_i]_{\et}, \Lambda) := \lim_{n \geq n_i} \Dbc([\Loop^+_{n_i} \cG \backslash X_i]_{\et}, \Lambda),
\]
this limit being stationary. If $i \leq j$ we have an obvious pushforward functor $\Dbc([\Loop^+ \cG \backslash X_i]_{\et}, \Lambda) \to \Dbc([\Loop^+ \cG \backslash X_j]_{\et}, \Lambda)$, and we can therefore set
\[
 \Dbc(\Hk_{\cG}, \Lambda) := \colim_i \Dbc([\Loop^+ \cG \backslash X_i]_{\et}, \Lambda).
\]
Standard arguments show that this category does not depend (up to equivalence) on the choice of presentation $\Gr_{\cG} = \colim_{i \geq 0} X_i$ as in~\S\ref{ss:quotient-stacks}.

Of course, in this construction one can forget the $\Loop^+ \cG$-action, and consider the category
\[
 \Dbc(\Gr_{\cG}, \Lambda) := \colim_i \Dbc(X_i, \Lambda).
\]
The quotient morphism $h \colon \Gr_\cG \to \Hk_{\cG}$ induces a pullback (or ``forgetful'') functor
\begin{equation}
\label{eqn:pullback-Hk-Gr}
 h^* \colon \Dbc(\Hk_{\cG}, \Lambda) \to \Dbc(\Gr_{\cG}, \Lambda).
\end{equation}

For any $i \geq 0$ one can consider the perverse t-structure on $\Dbc(X_i, \Lambda)$. These t-structures ``glue'' to define a (bounded) t-structure on $\Dbc(\Gr_{\cG}, \Lambda)$, which once again does not depend on the choice of presentation $\Gr_{\cG} = \colim_{i \geq 0} X_i$, and which will be called the perverse t-structure. Its heart will be denoted $\Perv(\Gr_{\cG}, \Lambda)$.

If $i \geq 0$ and $n \geq n_i$, one can also consider the perverse t-structure on the category $\Dbc([\Loop^+_{n} \cG \backslash X_i]_{\et}, \Lambda)$.
This t-structure \emph{does} depend on the choice of $n$; to remedy this we introduce a shift in this definition, so that the pullback functor
\[
 \Dbc([\Loop^+_{n} \cG \backslash X_i]_{\et}, \Lambda) \to \Dbc(X_i, \Lambda)
\]
becomes t-exact. With this normalization, for any $n \geq n_i$ the equivalence~\eqref{eqn:pullback-Xi} is t-exact, so that we obtain an induced t-structure on $\Dbc([\Loop^+ \cG \backslash X_i]_{\et}, \Lambda)$. If $j \geq i$ the pushforward functor $\Dbc([\Loop^+ \cG \backslash X_i]_{\et}, \Lambda) \to \Dbc([\Loop^+ \cG \backslash X_j ]_{\et}, \Lambda)$ is t-exact, so that once again these t-structures ``glue'' to define a (bounded) t-structure on $\Dbc(\Hk_{\cG}, \Lambda)$, which is called the perverse t-structure, and whose heart will be denoted $\Perv(\Hk_\cG,\Lambda)$. By construction of the perverse t-structure for stacks 
and our choice of normalization the functor~\eqref{eqn:pullback-Hk-Gr} is t-exact; in fact it ``detects perversity'' in the sense that for $\scF \in \Dbc(\Hk_{\cG}, \Lambda)$ the complex $\scF$ is concentrated in nonpositive perverse degrees, resp.~concentrated in nonnegative perverse degrees, resp.~perverse, if and only if so is $h^*(\scF)$.

The constructions in~\S\ref{ss:const-scalars} provide canonical ``extension of scalars'' functors
\begin{multline}
\label{eqn:functors-change-scalars-1}
\bk \lotimes_{\bO} (-) \colon \Dbc(\Hk_\cG, \bO) \to \Dbc(\Hk_\cG, \bk), \\
\bK \otimes_{\bO} (-) \colon \Dbc(\Hk_\cG, \bO) \to \Dbc(\Hk_\cG, \bK),
\end{multline}
and a canonical ``restriction of scalars'' functor
\begin{equation}
\label{eqn:functors-change-scalars-2}
 \Dbc(\Hk_\cG, \bk) \to \Dbc(\Hk_\cG, \bO)
\end{equation}
which is right adjoint to $\bk \lotimes_{\bO} (-)$. (This functor will usually be omitted from notation.)
The right-hand functor in~\eqref{eqn:functors-change-scalars-1} and the functor in~\eqref{eqn:functors-change-scalars-2} 
are t-exact for the perverse t-structures, hence induce exact functors
\[
\bK \otimes_{\bO} (-) \colon \Perv(\Hk_\cG,\bO) \to \Perv(\Hk_\cG,\bK)
\]
and
\[
 \Perv(\Hk_\cG,\bk) \to \Perv(\Hk_\cG,\bO).
\]
On the other hand the left-hand functor in~\eqref{eqn:functors-change-scalars-1} is right t-exact; we therefore obtain a right exact functor
 \[
 \pH^0(\bk \lotimes_{\bO} (-)) \colon \Perv(\Hk_\cG,\bO) \to \Perv(\Hk_\cG,\bk).
\]
Similar comments apply for categories of sheaves on $\Gr_\cG$, or for locally closed subschemes, or for the various stacks we have already considered, and we will use similar notation in these cases.

For any $\lambda \in \bbX_*(T)_I^+$, we can consider the perverse sheaves
\[
\scJ_!(\lambda,\Lambda) := \pH^0(j^\lambda_! \underline{\Lambda}_{\Gr_\cG^\lambda}[\langle \lambda, 2\rho \rangle]), \quad 
\scJ_*(\lambda,\Lambda) := \pH^0(j^\lambda_* \underline{\Lambda}_{\Gr_\cG^\lambda}[\langle \lambda, 2\rho \rangle])
\]
in $\Perv(\Hk_\cG,\Lambda)$, where we use the notation of~\eqref{eqn:jmu}. By adjunction there exists a canonical morphism
\[
 \scJ_!(\lambda,\Lambda) \to \scJ_*(\lambda,\Lambda)
\]
whose image is denoted $\scJ_{!*}(\lambda, \Lambda)$. In case $\Lambda \in \{\bk,\bK\}$ each $\scJ_{!*}(\lambda, \Lambda)$ is simple, and the assignment $\lambda \mapsto \scJ_{!*}(\lambda, \Lambda)$ induces a bijection between $\bbX_*(T)_I^+$ and the set of isomorphism classes of simple objects in $\Perv(\Hk_\cG,\Lambda)$; see e.g.~\cite[\S 8]{laszlo-olsson-perv}.

\subsection{Convolution}
\label{ss:convolution}

Recall now the ind-scheme $\Conv_{\cG}$ defined in~\S\ref{ss:convolution-schemes}, and the stack $\HkConv_{\cG}$ introduced in~\S\ref{ss:quotient-stacks}. Considerations similar to those of~\S\ref{ss:Perv-Gr-Hk} allow to define the triangulated categories
\[
 \Dbc(\Conv_{\cG}, \Lambda) \quad \text{and} \quad \Dbc(\HkConv_{\cG}, \Lambda),
\]
and the pullback functor
\begin{equation}
\label{eqn:pullback-HkConv}
  \tilde{h}^* \colon \Dbc(\HkConv_{\cG}, \Lambda) \to \Dbc(\Conv_{\cG}, \Lambda).
\end{equation}
One can also endow these categories with perverse t-structures, in such a way that the functor~\eqref{eqn:pullback-HkConv} is t-exact.

Recall the maps
\[
\Hk_\cG \times \Hk_\cG \xleftarrow{p} \HkConv_\cG \xrightarrow{m} \Hk_\cG
\]
from~\S\ref{ss:quotient-stacks}.
We can now define the convolution bifunctor
\[
\star \colon \Dbc(\Hk_{\cG},\Lambda) \times \Dbc(\Hk_{\cG},\Lambda)\to \Dbc(\Hk_{\cG},\Lambda)
\]
by setting
\[
\scF \star \scG := m_* p^*(\scF \lboxtimes_\Lambda \scG).
\]
Our goal in this subsection is to show that this bifunctor is right t-exact (on both sides) for the perverse t-structure. 

The morphism $m$ fits into a cartesian square
\[
\begin{tikzcd}
\Conv_\cG \ar[r, "\tilde{h}"] \ar[d, "m"'] & \HkConv_\cG \ar[d, "m"] \\
\Gr_\cG \ar[r, "h"] & \Hk_\cG.
\end{tikzcd}
\]
For any $\lambda,\mu,\nu \in \bbX_*(T)_I^+$ such that $\lambda + \mu \leq \nu$, the left-hand morphism $m$ restricts to a morphism
\[
m^\nu_{\lambda,\mu} \colon \Conv_{\cG}^{\leq (\lambda,\mu)} \to \Gr_{\cG}^{\leq \nu}.
\]
Recall also the notions of a stratified locally trivial morphism and a stratified semi\-small morphism from~\S\ref{ss:semismall}.

\begin{prop}
\label{prop:semismall}
For any $\lambda,\mu,\nu \in \bbX_*(T)_I^+$ such that $\lambda + \mu \leq \nu$, the morphism $m_{\lambda,\mu}^\nu$ is stratified locally trivial and
stratified semismall with respect to the stratifications
\[
| \Conv_{\cG}^{\le(\lambda,\mu)} | = \bigsqcup_{\substack{\lambda' \leq \lambda \\ \mu' \leq \mu}} | \Conv_{\cG}^{(\lambda',\mu')} |, \qquad | \Gr_{\cG}^{\leq \nu} | = \bigsqcup_{\nu' \leq \nu} | \Gr_{\cG}^{\nu'} |,
\]
see~\eqref{eqn:stratif-Schubert} and~\eqref{eqn:stratif-Schubert-Conv}.
\end{prop}

\begin{proof}
Our morphism is proper, and stratified locally trivial by $\Loop^+ \cG$-equivariance. It remains to check the condition on dimensions of fibers, which we will obtain from Lemma~\ref{lem:pushforward-semismall}. In the present setting, the first assumption of that lemma follows from $\Loop^+ \cG$-equivariance. For the second one
we observe that, by definition of convolution, for any $\lambda',\mu' \in \bbX_*(T)_I^+$ such that $\lambda' \leq \lambda$ and $\mu' \leq \mu$ we have 
\[
(m_{\lambda,\mu}^{\nu})_! \IC(\Conv_{\cG}^{(\lambda',\mu')}, \underline{\bK}) \cong \scJ_{!*}(\lambda', \bK) \star \scJ_{!*}(\mu', \bK).
\]
Hence the fact that this complex is perverse follows from~\cite[Theorem~5.11(i)]{richarz} (see also~\cite[Corollary~2.8]{zhu} in the tamely ramified case).
\end{proof}

\begin{rmk}
\begin{enumerate}
\item
The proofs of~\cite[Theorem~5.11(i)]{richarz} and~\cite[Corollary~2.8]{zhu} are incorrect as written, because of the problem mentioned in Remark~\ref{rmk:counterexample-surjectivity}. However they can easily be fixed by adding a step of reduction to adjoint groups.
\item
The proof of Proposition~\ref{prop:semismall} given here is suggested (with some precautions) in~\cite[Remark~2.9(i)]{zhu} (see also Remark~\ref{rmk:semismall} for comments on this method).
A different proof of Proposition~\ref{prop:semismall} can be obtained by copying the arguments of~\cite[Lemma~4.4]{mv} (see also~\cite[\S 1.6.3]{br}) and using the results of Section~\ref{sec:semi-infinite}.
\end{enumerate}
\end{rmk}

We can finally reach the goal of this subsection.

\begin{cor}
\label{cor:convolution-exact}
The bifunctor $\star$ is right t-exact, in the sense that if $\scF$ and $\scG$ belong to $\pD^{\leq 0}(\Hk_{\cG},\Lambda)$, then so does $\scF \star \scG$.
\end{cor}

\begin{proof}
The bifunctor $(\scF,\scG) \mapsto \scF \lboxtimes_\Lambda \scG$ is clearly right t-exact. On the other hand, by standard arguments involving t-exactness of shifted pullback along a smooth morphism the functor $p^*$ is t-exact, and the functor $m_*$ is t-exact by Proposition~\ref{prop:semismall} (and the fact that $h^*$ ``detects perversity,'' see~\S\ref{ss:Perv-Gr-Hk}). The desired claim follows.
\end{proof}

It is a standard fact that the bifunctor $\star$ defines a monoidal structure on the category $\Dbc(\Hk_{\cG},\Lambda)$. As a consequence of Corollary~\ref{cor:convolution-exact}, we obtain that the bifunctor
\[
\star^0 \colon \Perv(\Hk_{\cG},\Lambda) \times \Perv(\Hk_{\cG},\Lambda) \to \Perv(\Hk_{\cG},\Lambda)
\]
defined by
\[
\scF \star^0 \scG = \pH^0(\scF \star \scG)
\]
defines a monoidal structure on $\Perv(\Hk_{\cG},\Lambda)$. It is clear that the ``change of scalars'' functors considered in~\S\ref{ss:Perv-Gr-Hk} on the categories $\Dbc(\Hk_{\cG},\Lambda)$ commute with the bifunctor $\star$ in the obvious way; it follows that their counterparts on the categories $\Perv(\Hk_{\cG},\Lambda)$ admit canonical monoidal structures.

\subsection{Constant term functors}
\label{ss:CT-functors}

Consider a parabolic subgroup $P \subset G$ containing $A$, and let $M \subset P$ be its Levi factor containing $A$. Let $\cP$ and $\cM$ be their respective scheme-theoretic closures in $\cG$, see~\S\ref{ss:group-theory}. In~\S\ref{ss:attractor-fixed-pts} we have considered a diagram
\begin{equation}
\label{eqn:ct-diagram-0}
\Gr_{\cG} \xleftarrow{i_\cP} \Gr_{\cP} \xrightarrow{q_\cP} \Gr_{\cM}.
\end{equation}
These maps are equivariant for the action of $\Loop^+\cM$ on the left, so we obtain an analogous diagram of quotient stacks. (The morphisms will be denoted by the same symbols.) Appending the obvious quotient map $[\Loop^+\cM\backslash \Gr_\cG]_{\et} \to [\Loop^+\cG\backslash \Gr_\cG]_{\et}$ on the left, we obtain the following diagram:
\begin{equation}
\label{eqn:ct-diagram}
\Hk_\cG \xleftarrow{h_{\cM,\cG}} [\Loop^+\cM\backslash \Gr_\cG]_{\et} \xleftarrow{i_\cP} [\Loop^+\cM\backslash \Gr_\cP]_{\et} \xrightarrow{q_\cP} \Hk_\cM.
\end{equation}
We will use this diagram to define the \emph{constant term functor}
\[
\CT_{\cP,\cG} \colon \Dbc(\Hk_\cG,\Lambda) \to \Dbc(\Hk_\cM,\Lambda)
\]
as follows.

Recall from~\eqref{eqn:conn-comp-Gr} that the set $\pi_0(\Gr_\cM)$ of connected components of $\Gr_{\cM}$ is in canonical bijection with $(\bbX_*(T)/Q^\vee_M)_I$, where $Q^\vee_M \subset \bbX_*(T)$ is the coroot lattice of $(M_{F^s}, T_{F^s})$. Recall also that we denote by $\rho$ the half-sum of the positive roots of $(G_{F^s}, T_{F^s})$ with respect to $B_{F^s}$. We will similarly denote by $\rho_M$ the half-sum of the positive roots of $(M_{F^s}, T_{F^s})$ with respect to $(M \cap B)_{F^s}$; as for $\rho$, this element is fixed by the action of $I$.  The $\Z$-linear map $\langle {-}, 2\rho -2\rho_M\rangle \colon \bbX_*(T) \to \Z$ vanishes on $Q_M^\vee$ and is $I$-invariant, so it induces a $\Z$-linear map
\[
\langle {-}, 2\rho -2\rho_M\rangle \colon (\bbX_*(T)/Q^\vee_M)_I \to \Z.
\]
If $X \subset \Gr_\cM$ is the connected component corresponding to $\lambda \in (\bbX_*(T)/Q^\vee_M)_I$, we set $\corr_{M,G}(X) = \langle \lambda, 2\rho -2\rho_M \rangle$.  This defines a function
\[
\corr_{M,G} \colon \pi_0(\Gr_\cM) \to \Z.
\]
For $\scF$ in $\Dbc(\Gr_\cM,\Lambda)$ or $\Dbc(\Hk_\cM,\Lambda)$, one can make sense of the expression
\[
\scF[\corr_{M,G}]
\]
as meaning $\scF[\corr_{M,G}(X)]$ if $\scF$ is supported on the connected component $X$, and then extending to arbitrary $\scF$ by additivity.

We now define the constant term functor by the formula
\[
\CT_{\cP,\cG}(\scF) = q_{\cP!}i_\cP^* h_{\cM,\cG}^*(\scF)[\corr_{M,G}].
\]
The rest of this section is devoted to the proof of some basic properties of these functors. All the proofs are similar to those of their classical counterparts for split groups, see~\cite{mv, fs, br}; still, we will generally give (sketches of) the arguments for completeness.

First of all, it is clear that the functor $\CT_{\cP,\cG}$ commutes with the ``change of scalars'' functors considered in~\S\ref{ss:Perv-Gr-Hk} (for derived categories) in the obvious sense.

Next, we explain an alternative formula for $\CT_{\cP,\cG}$ involving the opposite parabolic $P^-$ to $P$. Let $\cP^-$ be the scheme theoretic closure of $P^-$ in $\cG$.  Then we can consider the following counterpart of~\eqref{eqn:ct-diagram}:
\begin{equation}\label{eqn:ct-diagram-2}
\Hk_\cG \xleftarrow{h_{\cM,\cG}} [\Loop^+\cM\backslash \Gr_\cG]_{\et} \xleftarrow{i_{\cP^-}} [\Loop^+\cM\backslash \Gr_{\cP^-}]_{\et} \xrightarrow{q_{\cP^-}} \Hk_\cM.
\end{equation}

\begin{prop}
\label{prop:Braden}
For any $\scF$ in $\Dbc(\Hk_{\cG}, \Lambda)$ there is a natural isomorphism
\[
(q_{\cP^-})_* i_{\cP^-}^! h_{\cM,\cG}^*(\scF)[\corr_{M,G}] \simto \CT_{\cP,\cG}(\scF).
\]
\end{prop}

\begin{proof}
This is an application of Braden's theory of hyperbolic localization~\cite{braden}, in the version explained in~\cite{richarz-Gm}.  Choose a cocharacter $\lambda \in \bbX_*(A)$ such that $P=P_\lambda$ (see Remark~\ref{rmk:parabolics-cocharacters}\eqref{it:parabolics-cocharacters}), and then let $\mathbb{G}_{\mathrm{m},\F}$ act on $\Gr_\cG$ via this cocharacter.  There is a natural transformation $(q_{\cP^-})_* \circ i_{\cP^-}^! \to q_{\cP!} \circ i_\cP^*$ given by~\cite[Construction~2.2]{richarz-Gm}.  
Since the pullback functor $\Dbc(\Hk_{\cM},\Lambda) \to \Dbc(\Gr_{\cM}, \Lambda)$ does not kill any nonzero object, one can check whether this morphism is an isomorphism on any given object after forgetting the $\Loop^+\cM$-equivariance, i.e., working with the diagram~\eqref{eqn:ct-diagram-0} rather than~\eqref{eqn:ct-diagram}.  According to~\cite[Theorem~2.6]{richarz-Gm}, our map is an isomorphism on any object in $\Dbc(\Gr_\cG,\Lambda)$ that is $\mathbb{G}_{\mathrm{m}}$-monodromic in the sense of~\cite[Definition~2.3]{richarz-Gm}.  But all objects in the image of $\Dbc(\Hk_\cG,\Lambda) \to \Dbc(\Gr_\cG,\Lambda)$ (or even of $\Dbc([\Loop^+\cM\backslash \Gr_\cG]_{\et}, \Lambda) \to \Dbc(\Gr_\cG,\Lambda)$) are $\mathbb{G}_{\mathrm{m}}$-monodromic, so we are done.
\end{proof}

If $P \subset P'$ are parabolic subgroups of $G$ containing $A$ then as in~\S\ref{ss:cartesian-diagram} the intersection $P \cap M'$ is a parabolic subgroup of the reductive group $M'$. We can therefore consider the construction above for $M'$, its special facet $\fa_{M'}$ (see~\S\ref{ss:group-theory}), and the parabolic subgroup $P \cap M'$, and obtain the functor $\CT_{\cP \cap \cM', \cM'}$.

\begin{lem}
\label{lem:transitivity-CT}
Let $P \subset P'$ be parabolic subgroups of $G$ containing $A$.
There exists a canonical isomorphism of functors
\[
\CT_{\cP \cap \cM', \cM'} \circ \CT_{\cP',\cG} \simto \CT_{\cP,\cG} \colon \Dbc(\Hk_\cG,\Lambda) \to \Dbc(\Hk_\cM,\Lambda).
\]
\end{lem}

\begin{proof}
This follows from the base change theorem, in view of Lemma~\ref{lem:cart-diag}\eqref{it:cart-diag-3}.
\end{proof}

\begin{prop}
\label{prop:exactness-CT}
For any parabolic subgroup $P \subset G$ containing $A$,
the functor
\[
\CT_{\cP,\cG} \colon \Dbc(\Hk_{\cG}, \Lambda) \to \Dbc(\Hk_{\cM},\Lambda)
\]
is t-exact and conservative.
\end{prop}

\begin{proof}
Lemma~\ref{lem:transitivity-CT} reduces the proof of the proposition to the case where $P$ is conjugate to $B$, which clearly reduces to the case $P=B$.
Since $\CT_{\cB,\cG}$ is a triangulated functor, to prove conservativity it suffices to prove that this functor does not kill any nonzero object. Let $\scF \in \Dbc(\Hk_{\cG}, \Lambda)$ be nonzero, and let $\lambda \in \bbX_*(T)_I^+$ be such that $\Gr_{\cG}^\lambda$ is open in the support of $\scF$. Then the restriction of $\scF$ to $\Gr_{\cG}^\lambda$ is nonzero, and if $\mu$ is the unique $W_\fa$-conjugate of $\lambda$ which belongs to $-\bbX_*(T)^+_I$ we have $| \rmS_\mu \cap \Gr_{\cG}^\lambda | = \{t^\mu\}$ by the proof of~\cite[Lemma~5.3]{aglr}. It follows that $\CT_{\cB,\cG}(\scF)$ is nonzero on the component of $\Gr_{\cT}$ corresponding to $\mu$, which finishes the proof of conservativity.

For t-exactness 
let us first consider the case $\Lambda \in \{\bk,\bK\}$. In this case, the subcategory $\pD^{\leq 0}(\Hk_{\cG}, \Lambda)$ is generated under extensions by the objects of the form $j^\lambda_! \underline{\Lambda}[\langle \lambda, 2\rho \rangle]$. Now, by the base change theorem, for any $\mu \in \bbX_*(T)^+_I$ we have
\[
\bigl( \CT_{\cB,\cG}(j^\lambda_! \underline{\Lambda}[\langle \lambda,2\rho \rangle]) \bigr)_{t^\mu} \cong R\Gammac(\Gr_{\cG}^\lambda \cap \rmS_\mu; \Lambda)[\langle \lambda+\mu,2\rho \rangle].
\]
By Lemma~\ref{lem:dim-estimate}, $\Gr_{\cG}^\lambda \cap \rmS_\mu$ has dimension $\langle \lambda+\mu,\rho \rangle$ when it is nonempty; hence $R\Gammac(\Gr_{\cG}^\lambda \cap \rmS_\mu; \Lambda)$ is concentrated in degrees at most $\langle \lambda+\mu,2\rho \rangle$, which implies that $\CT_{\cB,\cG}(j^\lambda_! \underline{\Lambda}[\langle \lambda,2\rho \rangle])$ is indeed concentrated in nonpositive degrees. This proves that $\CT_{\cB,\cG}$ is right t-exact. Left t-exactness follows using Verdier duality, Proposition~\ref{prop:Braden} and Remark~\ref{rmk:parabolics-cocharacters}\eqref{it:opposite-parabolic}.

Finally we consider the case $\Lambda=\bO$. Here also $\pD^{\leq 0}(\Hk_{\cG}, \bO)$ is generated under extensions by the objects of the form $j^\lambda_! \underline{\bO}[\langle \lambda,2\rho \rangle]$, so that the same proof as above shows that $\CT_{\cB,\cG}$ is right t-exact. On the other hand, $\pD^{\geq 0}(\Hk_{\cG}, \bO)$ is generated under extensions by the objects of the form $j^\lambda_* \underline{\bO}[\langle \lambda,2\rho \rangle]$ or $j^\lambda_* \underline{\bk}[\langle \lambda,2\rho \rangle]$ for $\lambda \in \bbX_*(T)^+_I$. Now $\CT_{\cB,\cG}(j^\lambda_* \underline{\bk}[\langle \lambda,2\rho \rangle])$ is concentrated in nonnegative degrees by the case of $\bk$ treated above, and so is
\[
\bk \lotimes_{\bO} \CT_{\cB,\cG}(j^\lambda_* \underline{\bO}[\langle \lambda,2\rho \rangle]) \cong \CT_{\cB,\cG}(j^\lambda_* \underline{\bk}[\langle \lambda,2\rho \rangle]).
\]
By standard arguments this implies that $\CT_{\cB,\cG}(j^\lambda_* \underline{\bO}[\langle \lambda, 2\rho \rangle])$ is concentrated in nonnegative  degrees, and finishes the proof.
\end{proof}

\begin{rmk}
\label{rmk:exactness-CT-vanishing}
Concretely, what t-exactness of the functor $\CT_{\cB,\cG}$ means is that for $\scF \in \Perv(\Hk_\cG, \Lambda)$, $i \in \Z$ and $\mu \in \bbX_*(T)_I$ we have
\[
\mathsf{H}_{\mathrm{c}}^i(\rmS_\mu, \scF) = 0 \quad \text{unless $i= \langle \mu,2\rho \rangle$.}
\]
Similarly, by Proposition~\ref{prop:Braden}, we have
\[
\mathsf{H}^i_{\rmT_\mu}(\Gr_\cG, \scF) = 0 \quad \text{unless $i= \langle \mu,2\rho \rangle$,}
\]
where $\mathsf{H}^\bullet_{\rmT_\mu}(\Gr_\cG,-)$ denotes cohomology with support in $\rmT_\mu$.
\end{rmk}

Proposition~\ref{prop:exactness-CT} implies that $\CT_{\cP,\cG}$ restricts to an exact functor
\[
 \Perv(\Hk_{\cG}, \Lambda) \to \Perv(\Hk_{\cM},\Lambda),
\]
which will be denoted similarly.
These functors commute with the ``change of scalars'' functors for categories of perverse sheaves considered in~\S\ref{ss:Perv-Gr-Hk}.

\subsection{Compatibility with total cohomology}
\label{ss:total-cohomology}

Recall the pullback functor $h^*$ from~\eqref{eqn:pullback-Hk-Gr}.
We now define the functor
\[
\sF_\cG \colon \Dbc(\Hk_{\cG}, \Lambda) \to \modf_\Lambda
\]
(where $\modf_\Lambda$ is the category of finitely generated $\Lambda$-modules) by
\[
\sF_\cG(\scF) := \bigoplus_{n \in \Z} \sH^n(\Gr_{\cG}, h^*\scF).
\]
(We ignore the $\Z$-grading on the right-hand side.) For any parabolic subgroup $P \subset G$ containing $A$, with Levi factor containing $A$ denoted $M$, we can consider the analogous functor $\sF_\cM \colon \Dbc(\Hk_{\cM}, \Lambda) \to \modf_\Lambda$.

\begin{prop}
\label{prop:CT-fiber}
For any $P$ as above, there exists a canonical isomorphism 
\[
\sF_\cG \cong \sF_\cM \circ \CT_{\cP,\cG}
\]
of functors from $\Perv(\Hk_{\cG}, \Lambda)$ to $\modf_\Lambda$.
\end{prop}

\begin{proof}
Lemma~\ref{lem:transitivity-CT} reduces the proof to the case $P=B$. In this case $\Gr_{\cT}$ is discrete (see~\S\ref{ss:semi-inf-orbits}), and for $\scF$ in $\Perv(\Hk_{\cG}, \Lambda)$ we have
\[
\sF_\cT \circ \CT_{\cB,\cG}(\scF) = \bigoplus_{\lambda \in \bbX_*(T)^+_I} \sH_{\mathrm{c}}^\bullet(\rmS_\lambda, \scF'_{|\rmS_\lambda})
\]
where $\scF':=h^* \scF$.
By Remark~\ref{rmk:exactness-CT-vanishing} we moreover have that
\begin{equation}
\label{eqn:vanishing-weight-functors}
\sH_{\mathrm{c}}^n(\rmS_\lambda, \scF'_{|\rmS_\lambda})=0 \quad \text{unless $n=\langle \lambda, 2\rho \rangle$.}
\end{equation}
What we have to construct is therefore a canonical identification
\begin{equation}
\label{eqn:cohomology-wt-spaces}
\sF_\cG(\scF) = \bigoplus_{\lambda \in \bbX_*(T)^+_I} \sH_{\mathrm{c}}^{\langle \lambda,2\rho \rangle}(\rmS_\lambda, \scF'_{|\rmS_\lambda}).
\end{equation}

Decomposing $\scF$ as a direct sum of its restrictions to each connected component of $\Gr_{\cG}$, we can assume that it is supported on one such component $X$.
Choose a presentation $X = \mathrm{colim}_i X_i$ by $\Loop^+ \cG$-stable projective $k$-schemes, and some index $i$ such that $\scF$ is supported on $X_i$. Recall the decomposition
\[
| X_i | = \bigsqcup_{\lambda \in Z} | \rmS_\lambda \times_{\Gr_{\cG}} X_i |
\]
where $Z$ is the finite subset of $\bbX_*(T)^+_I$ consisting of the elements $\lambda$ such that $\rmS_\lambda \times_{\Gr_{\cG}} X_i \neq \varnothing$; see in particular~\eqref{eqn:attractor-X_i}. The parity of the integers $\langle \lambda,2\rho \rangle$ is constant on this set since $X_i$ is contained in a connected component of $\Gr_{\cG}$; to fix notation we will assume that they are all even. For $n \in \Z$ we then set
\[
X_i^n = \bigcup_{\substack{\lambda \in Z \\ \langle \lambda, 2\rho \rangle \leq 2n}} \overline{\rmS_\lambda \times_{\Gr_{\cG}} X_i}
\]
where $\overline{\rmS_\lambda \times_{\Gr_{\cG}} X_i}$ is the scheme-theoretic closure of $\rmS_\lambda \times_{\Gr_{\cG}} X_i$. Then for some integers ${n_1} < {n_2}$ we have a finite filtration
\[
\varnothing = X_i^{n_1} \subset X_i^{{n_1}+1} \subset \cdots \subset X_i^{{n_2}-1} \subset X_i^{n_2} = X_i
\]
by closed subschemes, and by~\eqref{eqn:closure-S} we have
\[
X_i^n \smallsetminus X_i^{n-1} = \bigsqcup_{\substack{\lambda \in Z \\ \langle \lambda, 2\rho \rangle = 2n}} \rmS_\lambda \times_{\Gr_{\cG}} X_i
\]
for any $n \in \{{n_1}+1, \ldots, {n_2}\}$. For $n \in \{{n_1}, \dots, {n_2}\}$, resp.~$n \in \{{n_1}+1, \ldots, {n_2}\}$, we will denote by
\[
a_{\leq n} \colon X_i^n \to X_i, \quad \text{resp.} \quad a_n \colon X_i^n \smallsetminus X_i^{n-1} \to X_i
\]
the immersions. Then for any $n \in \{{n_1}+1, \ldots, {n_2}\}$ we have a canonical distinguished triangle
\[
(a_n)_! a_n^* \scF' \to (a_{\leq n})_! a_{\leq n}^* \scF' \to (a_{\leq n-1})_! a_{\leq n-1}^* \scF' \xrightarrow{[1]}.
\]
Using~\eqref{eqn:vanishing-weight-functors} and parity arguments one proves by induction on $n$ that we have canonical isomorphisms
\[
\sH^q(X_i^n, \scF'_{|X_i^n}) = \begin{cases}
\bigoplus_{\substack{\lambda \in Z \\ \langle \lambda, 2\rho \rangle = q}} \sH^q_{\mathrm{c}}(\rmS_\lambda, \scF'_{| \rmS_\lambda}) & \text{if $q$ is even and $q \leq 2n$;} \\
0 & \text{otherwise.}
\end{cases}
\]
Taking $n={n_2}$ we deduce~\eqref{eqn:cohomology-wt-spaces}, which finishes the proof.
\end{proof}

\begin{rmk}
\phantomsection
\begin{enumerate}
\item
As in the usual geometric Satake context (see~\cite[Theorem~1.10.4]{br}), Proposition~\ref{prop:CT-fiber} (in case $P=B$) implies that the functor $\sF_\cG \colon \Perv(\Hk_{\cG}, \Lambda) \to \modf_\Lambda$ is exact.
\item
 Recall the ``change of scalars'' functors from~\S\ref{ss:Perv-Gr-Hk}. It is clear that for $\scF$ in $\Perv(\Hk_\cG,\bO)$ we have a canonical isomorphism
 \[
  \bK \otimes_{\bO} \sF_\cG(\scF) \cong \sF_\cG(\bK \otimes_{\bO} \scF).
 \]
On the other hand, there exists a canonical morphism
\begin{equation}
\label{eqn:mod-reduction-F}
  \bk \otimes_{\bO} \sF_\cG(\scF) \to \sF_\cG(\pH^0(\bk \otimes_{\bO} \scF)).
 \end{equation}
 It is clear that the analogous morphism for the category $\Perv(\Hk_\cT,\bO)$ is an isomorphism. Using Proposition~\ref{prop:CT-fiber} we deduce that~\eqref{eqn:mod-reduction-F} is an isomorphism.
 \end{enumerate}
\end{rmk}

The proof of Proposition~\ref{prop:CT-fiber} depends crucially on the fact that we are working with perverse sheaves; one cannot expect it to hold for general non-perverse objects in $\Dbc(\Hk_\cG,\Lambda)$.  Nevertheless, the following statement shows that Proposition~\ref{prop:CT-fiber} does generalize to direct sums of shifts of perverse sheaves.

\begin{cor}
\label{cor:CT-fiber-sum}
Let $M \in \Dbc(\mathrm{pt},\Lambda)$, and assume that $\sH^i(M) = 0$ for $i$ odd.
For $\scF \in \Perv(\Hk_\cG,\Lambda)$, there is a natural isomorphism
\[
\sF_\cG(\scF \lotimes_\Lambda M) \cong \sF_\cM(\CT_{\cP,\cG}(\scF \lotimes_\Lambda M))
\]
\end{cor}

\begin{proof}
For brevity, let $\scF' = h^*\scF$.  It is enough to prove the claim in the case where $\scF$ is supported on a single connected component of $\Gr_\cG$.  In this case, by (the proof of) Proposition~\ref{prop:CT-fiber}, $\sH^i(\Gr_\cG,\scF')$ can be nonzero only for $i$ of a single parity. To fix notation, we assume that $\sH^i(\Gr_\cG,\scF') \ne 0$ only for $i$ even.

By an appropriate version of the universal coefficient theorem, for any $n \in \Z$ there is a natural short exact sequence
\begin{multline*}
0 \to \bigoplus_{i+j=n} \sH^i(\Gr_\cG,\scF') \otimes_\Lambda \sH^j(M) \to \sH^n(\Gr_\cG, \scF' \lotimes_\Lambda M) \\
\to \bigoplus_{i+j=n+1} \Tor_1^\Lambda(\sH^{i}(\Gr_\cG,\scF'), \sH^j(M)) \to 0.
\end{multline*}
Our assumptions imply that the first term vanishes when $n$ is odd, and the last term vanishes when $n$ is even.  We deduce that
\[
\sH^n(\Gr_\cG, \scF' \lotimes_\Lambda M) \cong
\begin{cases}
\bigoplus_{i+j=n} \sH^i(\Gr_\cG,\scF') \otimes_\Lambda \sH^j(M)  & \text{if $n$ is even,} \\
\bigoplus_{i+j=n+1} \Tor_1^\Lambda(\sH^{i}(\Gr_\cG,\scF'), \sH^j(M)) & \text{if $n$ is odd,}
\end{cases}
\]
and hence that there is a natural isomorphism
\begin{equation}\label{eqn:CTfs1}
\sF_\cG(\scF \lotimes_\Lambda M) \cong (\sF_\cG(\scF) \otimes_\Lambda \sH^\bullet(M)) \oplus \Tor_1^\Lambda(\sF_\cG(\scF), \sH^\bullet(M)).
\end{equation}
Applying this isomorphism with the group $\cG$ replaced by $\cM$ and with $\scF$ replaced by $\CT_{\cP,\cG}(\scF)$, we obtain an isomorphism
\begin{multline}\label{eqn:CTfs2}
\sF_\cM(\CT_{\cP,\cG}(\scF \lotimes_\Lambda M)) \cong \sF_\cM(\CT_{\cP,\cG}(\scF) \lotimes_\Lambda M)  \\
 \cong (\sF_\cM(\CT_{\cP,\cG}(\scF)) \otimes_\Lambda \sH^\bullet(M)) \oplus \Tor_1^\Lambda(\sF_\cM(\CT_{\cP,\cG}(\scF)), \sH^\bullet(M)).
\end{multline}
By Proposition~\ref{prop:CT-fiber}, the right-hand sides of~\eqref{eqn:CTfs1} and~\eqref{eqn:CTfs2} are naturally isomorphic, so the left-hand sides are as well.
\end{proof}

\subsection{Constant term functors for convolution schemes}

Let $P \subset G$ be a parabolic subgroup containing $A$, and let $M \subset P$ be its Levi factor containing $A$.  Let $\cP$ and $\cM$ be the scheme-theoretic closures of $P$ and $M$, respectively, in $\cG$.  Then we have a diagram
\[
\Conv_\cG \xleftarrow{\tilde\imath_\cP} \Conv_\cP \xrightarrow{\tilde q_\cP} \Conv_\cM
\]
similar to~\eqref{eqn:ct-diagram-0},
along with the companion diagram
\[
\HkConv_\cG \xleftarrow{\tilde h_{\cM,\cG}} [\Loop^+\cM \backslash \Conv_\cG]_{\et} \xleftarrow{\tilde\imath_\cP} [\Loop^+\cM \backslash \Conv_\cP]_{\et} \xrightarrow{\tilde q_\cP} \HkConv_\cM.
\]
As in~\S\ref{ss:semi-inf-orbits-Conv} (where we considered the case $P=B$)
the connected components of $\Conv_\cM$ (and hence also those of $\Conv_\cP$) are in bijection with $(\bbX_*(T)/Q^\vee_M)_I \times (\bbX_*(T)/Q^\vee_M)_I$. (This parametrization differs from that obtained via the identification with $\Gr_\cM \times \Gr_\cM$ from~\eqref{eqn:Conv-product} by the map $(\lambda,\mu) \mapsto (\lambda,\lambda+\mu)$.) Define a function
\[
\widetilde{\corr}_{M,G} \colon \pi_0(\Conv_\cM) \to \Z
\]
as follows: if $X \subset \Conv_\cM$ is the connected component corresponding to $(\lambda,\mu) \in (\bbX_*(T)/Q^\vee_M)_I \times (\bbX_*(T)/Q^\vee_M)_I$, then
\[
\widetilde{\corr}_{M,G}(X) = \langle \lambda+\mu, 2\rho-2\rho_M \rangle.
\]
Define the \emph{convolution constant term functor}
\[
\widetilde{\CT}_{\cP,\cG} \colon \Dbc(\HkConv_\cG, \Lambda) \to \Dbc(\HkConv_\cM, \Lambda)
\]
by a recipe similar to that used for $\CT_{\cP,\cG}$: specifically, we set
\[
\widetilde{\CT}_{\cP,\cG} = \tilde q_{\cP!}\tilde\imath_\cP^* \tilde h_{\cM,\cG}^*(\scF)[\widetilde{\corr}_{M,G}].
\]

Using the opposite parabolic $P^-$ and its scheme-theoretic closure $\cP^-$, we have the following counterpart of~\eqref{eqn:ct-diagram-2}:
\[
\HkConv_\cG \xleftarrow{\tilde h_{\cM,\cG}} [\Loop^+\cM\backslash \Conv_\cG]_{\et} \xleftarrow{\tilde\imath_{\cP^-}} [\Loop^+\cM\backslash \Conv_{\cP^-}]_{\et} \xrightarrow{\tilde q_{\cP^-}} \HkConv_\cM.
\]

The proofs of the next two statements are essentially identical to those of Proposition~\ref{prop:Braden} and Lemma~\ref{lem:transitivity-CT},
and we omit them. 

\begin{prop}
\label{prop:Braden-conv}
For any $\scF$ in $\Dbc(\HkConv_{\cG}, \Lambda)$ there is a natural isomorphism
\[
(\tilde q_{\cP^-})_* \tilde\imath_{\cP^-}^! \tilde h_{\cM,\cG}^*(\scF)[\widetilde{\corr}_{M,G}] \simto \widetilde{\CT}_{\cP,\cG}(\scF).
\]
\end{prop}

\begin{lem}
\label{lem:transitivity-CT-conv}
Let $P \subset P'$ be parabolic subgroups of $G$ containing $A$, and let $M \subset M'$ be their respective Levi factors containing $A$.  There exists a canonical isomorphism of functors
\[
\widetilde{\CT}_{\cP \cap \cM', \cM'} \circ \widetilde{\CT}_{\cP',\cG} \simto \widetilde{\CT}_{\cP,\cG} \colon \Dbc(\HkConv_\cG,\Lambda) \to \Dbc(\HkConv_\cM,\Lambda).
\]
\end{lem}

For the next statement
we need to restrict the domain of $\widetilde{\CT}_{\cP,\cG}$
as follows: define
\[
\Dbc(\Conv_\cG,\Lambda)_\sph \subset \Dbc(\Conv_\cG,\Lambda)
\]
to be the full subcategory consisting of objects that are constructible with respect to the stratification 
\begin{equation}\label{eqn:conv-spherical}
| \Conv_\cG | = \bigsqcup_{\lambda,\mu \in \bbX_*(T)_I^+} | \Conv_\cG^{(\lambda,\mu)} |.
\end{equation}
Similarly, define
\[
\Dbc(\HkConv_\cG,\Lambda)_\sph \subset \Dbc(\HkConv_\cG,\Lambda)
\]
to be the full subcategory consisting of objects whose image in $\Dbc(\Conv_\cG,\Lambda)$ lies in the subcategory $\Dbc(\Conv_\cG,\Lambda)_\sph$.  The triangulated categories $\Dbc(\Conv_\cG,\Lambda)_\sph$ and $\Dbc(\HkConv_\cG,\Lambda)_\sph$ both inherit a perverse t-structure; their hearts are denoted by
\[
\Perv(\Conv_\cG,\Lambda)_\sph
\qquad\text{and}\qquad
\Perv(\HkConv_\cG,\Lambda)_\sph,
\]
respectively.

\begin{prop}
\label{prop:exactness-CT-conv}
The functor
\[
\widetilde{\CT}_{\cB,\cG} \colon \Dbc(\HkConv_\cG, \Lambda)_\sph \to \Dbc(\HkConv_\cT,\Lambda)
\]
is t-exact and conservative.
\end{prop}

\begin{proof}
The proof is essentially identical to that of Proposition~\ref{prop:exactness-CT}, replacing the reference to Lemma~\ref{lem:dim-estimate} by a reference to Lemma~\ref{lem:dim-estimate-Conv}.
\end{proof}

In concrete terms, Proposition~\ref{prop:exactness-CT-conv} says that for $\scF \in \Perv(\HkConv_\cG,\Lambda)_\sph$, $i \in \Z$ and $\lambda,\mu \in \bbX_*(T)_I$ we have
\begin{equation}\label{eqn:vanishing-wt-conv}
\sH^i_{\mathrm{c}}(\rmS_\lambda \wttimes \rmS_\mu, \tilde h^*\scF) = 0
\qquad\text{unless $i = \la \lambda + \mu, 2\rho \ra$,}
\end{equation}
where $\tilde h \colon \Conv_\cG\to \HkConv_\cG$ is the quotient map, see~\S\ref{ss:quotient-stacks}.

\begin{rmk}\label{rmk:exactness-CT-conv}
It is likely that Proposition~\ref{prop:exactness-CT-conv} holds for any $\widetilde{\CT}_{\cP,\cG}$, not just the special case $P = B$.  In fact, the general case would follow from Lemma~\ref{lem:transitivity-CT-conv} if we also had the following claim:
\begin{quote}
\it For any $\scF \in \Dbc(\HkConv_\cG,\Lambda)_\sph$, the object $\widetilde{\CT}_{\cP,\cG}(\scF)$ lies in $\Dbc(\HkConv_\cM,\Lambda)_\sph$.
\end{quote}
This claim is probably true, and not difficult to prove, but as we will not need it in the sequel we do not pursue it here.
\end{rmk}

In the next statement, we denote by $\pr_{1,\cG}$ and $\pr_{1,\cT}$ the two versions of the ``first projection'' map
\[
\HkConv_\cG \to \Hk_\cG
\qquad\text{and}\qquad
\HkConv_\cT \to \Hk_\cT,
\]
see~\S\ref{ss:convolution-schemes}--\ref{ss:quotient-stacks}.

\begin{prop}
\label{prop:CT-compare-pr}
For $\scF \in \Perv(\HkConv_\cG, \Lambda)_\sph$, there is a natural isomorphism
\[
\sF_\cT(\CT_{\cB,\cG}(\pr_{1,\cG!} \scF)) \cong \sF_\cT(\pr_{1,\cT!} \widetilde{\CT}_{\cB,\cG}(\scF)).
\]
\end{prop}

\begin{proof}
The proof is similar in spirit to that of Proposition~\ref{prop:CT-fiber}. We may assume without loss of generality that $\scF$ is supported on a single connected component $X$ of $\Conv_\cG$. As in that proof, we have
\begin{multline*}
\sF_\cT(\CT_{\cB,\cG}(\pr_{1,\cG!} \scF))
= 
\bigoplus_{\lambda \in \bbX_*(T)^+_I} \sH_{\mathrm{c}}^\bullet(\rmS_\lambda, (h^*\pr_{1,\cG!}\scF)_{|\rmS_\lambda} ) \\
= \bigoplus_{\lambda \in \bbX_*(T)^+_I} \sH_{\mathrm{c}}^\bullet(\pr_{1,\cG}^{-1}(\rmS_\lambda),\scF'_{|\pr_{1,\cG}^{-1}(\rmS_\lambda)} )
= \bigoplus_{\lambda \in \bbX_*(T)^+_I} \sH_{\mathrm{c}}^\bullet(\rmS_\lambda \wttimes \Gr_\cG,\scF'_{|\rmS_\lambda \wttimes \Gr_\cG} )
\end{multline*}
where $\scF' := \tilde h^*\scF$.

We now fix $\lambda \in \bbX_*(T)^+_I$. The subset $\rmS_\lambda \wttimes \Gr_\cG \subset \Conv_\cG$ is a union of subsets of the form $\rmS_\lambda \wttimes \rmS_\mu$. Moreover, the parity of $\la \lambda + \mu, 2\rho\ra$ is constant among the subsets of this form that are contained in $X$; to fix notation we assume that these numbers are even.

Consider the restriction $\scF'_{|\rmS_\lambda \wttimes \Gr_\cG}$ of $\scF'$ to $\rmS_\lambda \wttimes \Gr_\cG$.  Because $\scF'$ lies in $\Perv(\Conv_\cG,\Lambda)_\sph$, this object is supported on a subscheme of the form $\rmS_\lambda \wttimes X'$, where $X' \subset \Gr_\cG$ is some $\Loop^+\cG$-stable projective $\F$-scheme contained in a single component of $\Gr_\cG$. We set $Y = \rmS_\lambda \wttimes X'$.
We also let $Z$ be the finite subset of $\bbX_*(T)^+_I$ consisting of the elements $\mu \in \bbX_*(T)_I$ such that that $\rmS_\mu \times_{\Gr_\cG} X' \ne \varnothing$. 

For $n \in \Z$ we set
\[
Y^n = \bigcup_{\substack{\mu \in Z \\ \langle \lambda + \mu, 2\rho \rangle \leq 2n}} \rmS_\lambda \wttimes (\overline{\rmS_\mu \times_{\Gr_{\cG}} X'}).
\]
Then for some integers $n_1 < n_2$ we have a finite filtration
\[
\varnothing = Y^{n_1} \subset Y^{n_2} \subset \cdots \subset Y^{n_2-1} \subset Y^{n_2} = Y
\]
by closed subschemes, as well as a decomposition
\[
Y^n \smallsetminus Y^{n-1} = \bigsqcup_{\substack{\mu \in Z \\ \langle \lambda +\mu, 2\rho \rangle = 2n}} \rmS_\lambda \wttimes (\rmS_\mu \times_{\Gr_{\cG}} X').
\]
Let $a_{\le n} \colon Y^n \to Y$ and $a_n \colon Y^n \smallsetminus Y^{n-1} \to Y$ be the immersions, so that for $n \in \{n_1+1, \dots, n_2\}$ we have a distinguished triangle
\[
(a_n)_! a_n^* \scF' \to (a_{\leq n})_! a_{\leq n}^* \scF' \to (a_{\leq n-1})_! a_{\leq n-1}^* \scF' \xrightarrow{[1]}.
\]
Using~\eqref{eqn:vanishing-wt-conv} and parity arguments one proves by induction on $n$ that we have canonical isomorphisms
\begin{equation*}
\sH^q(Y^n, \scF'_{|Y^n}) = \begin{cases}
\bigoplus_{\substack{\mu \in Z_\lambda \\ \langle \lambda + \mu, 2\rho \rangle = q}} \sH^q(\rmS_\lambda \wttimes \rmS_\mu, \scF'_{|\rmS_\lambda \wttimes \rmS_\mu}) & \text{if $q$ is even and $q \leq 2n$;} \\
0 & \text{otherwise.}
\end{cases}
\end{equation*}
Taking $n={n_2}$ we deduce an isomorphism
\[
\sH_{\mathrm{c}}^\bullet(\rmS_\lambda \wttimes \Gr_\cG,\scF'_{|\rmS_\lambda \wttimes \Gr_\cG} ) = \bigoplus_{\mu \in \bbX_*(T)^+_I} \sH_{\mathrm{c}}^{\la \lambda + \mu, 2\rho \ra}(\rmS_\lambda \wttimes \rmS_\mu,\scF'_{|\rmS_\lambda \wttimes \rmS_\mu}).
\]

Summing the previous isomorphisms over $\lambda$ we deduce 
a natural isomorphism
\[
\sF_\cT(\CT_{\cB,\cG}(\pr_{1,\cG!} \scF))
= 
\bigoplus_{\lambda, \mu \in \bbX_*(T)^+_I} \sH_{\mathrm{c}}^{\la \lambda + \mu, 2\rho \ra}(\rmS_\lambda \wttimes \rmS_\mu,\scF'_{|\rmS_\lambda \wttimes \rmS_\mu}).
\]
The right-hand side identifies with $\sF_\cT(\pr_{1,\cT!} \widetilde{\CT}_{\cB,\cG}(\scF))$, so we are done.
\end{proof}

\begin{cor}
\label{cor:CT-conv-split}
Let $\scF, \scG \in \Perv(\Hk_\cG, \Lambda)$, and let $Y \subset \Conv_\cG$ be a locally closed sub-ind-scheme that is a union of subsets of the form $\rmS_\lambda \wttimes \rmS_\mu$.  Then there is a canonical isomorphism
\[
\sH^\bullet_{\mathrm{c}} \bigl( Y, (\tilde h^*p^*(\scF \lboxtimes_\Lambda \scG))_{|Y} \bigr) \cong
\bigoplus_{\substack{\lambda, \mu \in \bbX_*(T)^+_I\\ \rmS_\lambda \wttimes \rmS_\mu \subset Y}} \sH^{\la \lambda + \mu, 2\rho \ra}_{\mathrm{c}} \bigl( \rmS_\lambda \wttimes \rmS_\mu, (\tilde h^*p^*(\scF \lboxtimes_\Lambda \scG))_{|\rmS_\lambda \wttimes \rmS_\mu} \bigr).
\]
\end{cor}

\begin{proof}
For brevity, set $\scF' = h^* \scF$ and $\scG' = h^*\scG$.  It is sufficient to treat the case where $\scG'$ is supported on a single connected component of $\Gr_\cG$, and we assume this as well.  Let us first consider the special case where $Y = \Conv_\cG$.  By proper base change using the cartesian square~\eqref{eqn:ppr-cartesian}, we have
\[
\pr_{1,\cG!}p^*(\scF \lboxtimes_\Lambda \scG) \cong \scF \lotimes_\Lambda R\Gamma(\Gr_\cG, \scG').
\]
Since $\scG'$ is supported on a single connected component, as in the proof of Corollary~\ref{cor:CT-fiber-sum},  
the nonzero cohomology groups $\sH^i(\Gr_\cG,\scG')$ all have $i$ of the same parity.  Then, by  
Corollary~\ref{cor:CT-fiber-sum}, we obtain that
\begin{equation}\label{eqn:CT-conv-split1}
\sF_\cG \bigl( \pr_{1,\cG!}p^*(\scF \lboxtimes_\Lambda \scG) \bigr) \cong \sF_\cT \bigl( \CT_{\cB,\cG}(\pr_{1,\cG!}p^*(\scF \lboxtimes_\Lambda \scG)) \bigr).
\end{equation}
Combining this with
Proposition~\ref{prop:CT-compare-pr}, we obtain
\begin{multline}\label{eqn:CT-compare-rephrase}
\sH^\bullet_\mathrm{c}(\Conv_\cG, \tilde h^* p^*(\scF \lboxtimes_\Lambda \scG) ) 
= \sH^\bullet_\mathrm{c}(\Gr_\cG, h^*\pr_{1,\cG!}p^*(\scF \lboxtimes_\Lambda \scG)) \\
\cong \sF_\cT(\CT_{\cB,\cG}(\pr_{1,\cG!}p^*(\scF \lboxtimes_\Lambda \scG))) \\
\cong
\bigoplus_{\lambda, \mu \in \bbX_*(T)^+_I} \sH_{\mathrm{c}}^{\la \lambda + \mu, 2\rho \ra}(\rmS_\lambda \wttimes \rmS_\mu,\tilde h^*p^*(\scF \lboxtimes_\Lambda \scG)_{|\rmS_\lambda \wttimes \rmS_\mu}),
\end{multline}
which establishes the desired identification in this case.

Here is another interpretation of this equality.  For $\lambda, \mu \in \bbX_*(T)^+_I$, let $\rmS_{\not\ge (\lambda,\mu)} \subset \Conv_\cG$ be the closed subset defined by
\[
\rmS_{\not\ge(\lambda,\mu)} = \bigcup_{\substack{\lambda',\mu' \in \bbX_*(T)^+_I \\ \text{$\lambda' \not\ge \lambda$ or $\lambda'+\mu' \not\ge \lambda+\mu$}}} \rmS_{\lambda'} \wttimes \rmS_{\mu'}.
\]
This is a closed subset of $\Conv_\cG$.  There is a filtration $F_{\bullet,\bullet}$ of the $\Lambda$-module
\[
\sH^\bullet_{\mathrm{c}}(\Conv_\cG, \tilde h^* p^*(\scF \lboxtimes_\Lambda \scG) )
\]
indexed by $\bbX_*(T)^+_I \times \bbX_*(T)^+_I$, given by
\[
F_{\lambda,\mu} := \ker (\sH^\bullet_{\mathrm{c}}(\Conv_\cG, \tilde h^* p^*(\scF \lboxtimes_\Lambda \scG) ) \to \sH^\bullet_{\mathrm{c}}(\rmS_{\not\ge(\lambda,\mu)}, \tilde h^* p^*(\scF \lboxtimes_\Lambda \scG) _{|\rmS_{\not\ge(\lambda,\mu)}})).
\]
Then~\eqref{eqn:CT-compare-rephrase} says that this filtration admits a canonical splitting.

We now return to the setting of a general $Y$ as in the statement of the corollary.  Let $Z$ be the set of pairs $(\lambda, \mu) \in \bbX_*(T)^+_I \times \bbX_*(T)^+_I$ such that $\rmS_\lambda \wttimes \rmS_\mu \subset Y$, and let
\[
Z^- = \left\{ (\lambda, \mu) \in \bbX_*(T)^+_I \times \bbX_*(T)^+_I \,\Big|\,
\begin{array}{c}
\text{there exists $(\lambda',\mu') \in Z$ with} \\
\text{$\lambda \ge \lambda'$ and $\lambda+\mu \ge \lambda' + \mu'$}
\end{array} \right\}.
\]
The fact that $Y$ is locally closed implies that $Z^- \smallsetminus Z$ is an upper closed set with respect to the partial order on $\bbX_*(T)^+_I \times \bbX_*(T)^+_I$.  As a consequence, the module $\sH^\bullet_{\mathrm{c}}(Y, \tilde h^* p^*(\scF \lboxtimes_\Lambda \scG)_{|Y})$ can be identified with a subquotient of the filtration defined above: namely,
\[
\sH^\bullet_{\mathrm{c}}(Y, \tilde h^* p^*(\scF \lboxtimes_\Lambda \scG) _{|Y}) =
\sum_{(\lambda,\mu) \in Z^-} F_{\lambda,\mu} \Big/ \sum_{(\lambda,\mu) \in Z^- \smallsetminus Z} F_{\lambda,\mu}.
\]
Since the filtration is canonically split, the result follows.
\end{proof}

\begin{prop}\label{prop:CT-compare}
For $\scF,\scG \in \Perv(\Hk_\cG, \Lambda)$, there is a natural isomorphism
\[
\CT_{\cB,\cG}(\scF \star^0 \scG) \cong m_{\cT!} \widetilde{\CT}_{\cB,\cG}(p^*(\pH^0(\scF \lboxtimes_\Lambda \scG))).
\]
\end{prop}

\begin{proof}
The statement is a natural isomorphism in $\Perv(\Hk_\cT,\Lambda)$.  Below, we will prove a closely related statement: we will show that in $\Perv(\Gr_\cT,\Lambda)$, there is a natural isomorphism
\begin{equation}\label{eqn:CT-compare-mod}
h^*\CT_{\cB,\cG}(\scF \star \scG) \cong h^* m_{\cT!} \widetilde{\CT}_{\cB,\cG}(p^*(\scF \lboxtimes_\Lambda \scG)).
\end{equation}
Let us explain how to deduce the proposition from~\eqref{eqn:CT-compare-mod}.  By t-exactness of the various functors above (see Propositions~\ref{prop:semismall}, \ref{prop:exactness-CT}, and \ref{prop:exactness-CT-conv}, and also the discussion in~\S\ref{ss:semismall}), \eqref{eqn:CT-compare-mod} implies that
\[
h^*\CT_{\cB,\cG}(\scF \star^0 \scG) \cong h^* m_{\cT!} \widetilde{\CT}_{\cB,\cG}(p^*(\pH^0(\scF \lboxtimes_\Lambda \scG))),
\]
and this in turn implies the proposition because  
$h^* \colon \Perv(\Hk_\cT,\Lambda) \to \Perv(\Gr_\cT,\Lambda)$ is fully faithful. 

Let us prove~\eqref{eqn:CT-compare-mod}. Since $\Gr_\cT$ is discrete, we may further reduce the problem to that of comparing stalks of these complexes at each point of $\Gr_\cT$.  Let us fix an element $\nu \in \bbX_*(T)_I$, and focus on the stalks 
$(h^* \CT_{\cB,\cG}(\scF \star \scG))_{t^\nu}$ and $(h^* m_{\cT!} \widetilde{\CT}_{\cB,\cG}(p^*(\scF \lboxtimes_\Lambda \scG)))_{t^\nu}$.

Note that $m_{\cG}^{-1}(\rmS_\nu)$ is the union of all $\rmS_\lambda \wttimes \rmS_\mu$ where $\lambda + \mu = \nu$.  Therefore, using Corollary~\ref{cor:CT-conv-split}, we have
\begin{multline*}
(h^* \CT_{\cB,\cG}(\scF \star \scG))_{t^\nu} \cong \sH^{\la \nu, 2\rho \ra}_{\mathrm{c}}(\rmS_\nu, (m_{\cG!}p^*(\scF \lboxtimes_\Lambda \scG))_{|\rmS_\nu}) \\
\cong \sH^{\la \nu, 2\rho \ra}_{\mathrm{c}}(m_\cG^{-1}(\rmS_\nu), p^*(\scF \lboxtimes_\Lambda \scG)_{|m_\cG^{-1}(\rmS_\nu)}) \\\cong
\bigoplus_{\substack{\lambda, \mu \in \bbX_*(T)^+_I \\ \lambda + \mu = \nu}} \sH^{\la \nu, 2\rho \ra}_{\mathrm{c}}(\rmS_\lambda \wttimes \rmS_\mu, p^*(\scF \lboxtimes_\Lambda \scG)_{|\rmS_\lambda \wttimes \rmS_\mu}).
\end{multline*}

On the other hand, the stalk $(h^* m_{\cT!} \widetilde{\CT}_{\cB,\cG}(p^*(\scF \lboxtimes_\Lambda \scG)))_{t^\nu}$ 
is given by
\begin{multline*}
(h^* m_{\cT!} \widetilde{\CT}_{\cB,\cG}(p^*(\scF \lboxtimes_\Lambda \scG)))_{t^\nu} \cong \bigoplus_{\lambda+\mu = \nu} (h^* \widetilde{\CT}_{\cB,\cG}(p^*(\scF \lboxtimes_\Lambda \scG)))_{[\lambda,\mu]} \\
\cong \bigoplus_{\lambda+\mu = \nu} \sH^{\la \lambda+\mu, 2\rho \ra}_{\mathrm{c}}(\rmS_\lambda \wttimes \rmS_\mu, p^*(\scF \lboxtimes_\Lambda \scG)_{|\rmS_\lambda \wttimes \rmS_\mu})
\end{multline*}
where $[\lambda,\mu] \in \Conv_{\cT}$ is the point corresponding to $(\lambda,\lambda+\mu)$ under the identifications~\eqref{eqn:Conv-product} (for $\cT$) and~\eqref{eqn:GrT-X}.
The result follows.
\end{proof}

Combining Proposition~\ref{prop:CT-compare} with Proposition~\ref{prop:CT-fiber}, we obtain the following immediate consequence.

\begin{prop}
\label{prop:CT-fiber-conv}
For $\scF, \scG \in \Perv(\Hk_\cG,\Lambda)$, there is a natural isomorphism
\[
\sF_\cG(\scF \star^0 \scG) \cong \sF_\cT(m_{\cT!}\widetilde{\CT}_{\cB,\cG}(p^*(\pH^0(\scF \lboxtimes_\Lambda \scG)))).
\]
\end{prop}

\begin{rmk}\label{rmk:analytic-pres}
The preceding proposition depends on Corollary~\ref{cor:CT-fiber-sum}, whose proof makes crucial use of the fact that $\Lambda$ has global dimension${}\le1$.  In the analytic setting mentioned in Remark~\ref{rmk:analytic}, one may wish to consider coefficient rings of global dimension${}>1$ (but still finite).  To handle this situation, one can modify the argument as follows:
\begin{itemize}
\item First, prove Corollary~\ref{cor:CT-fiber-sum} under the additional assumption that the cohomology modules $\sH^i(M)$ are flat over $\Lambda$.
\item Then, prove Corollary~\ref{cor:CT-conv-split}, Proposition~\ref{prop:CT-compare}, and Proposition~\ref{prop:CT-fiber-conv} under the additional assumption that $\sF_\cG(\scG)$ is flat over $\Lambda$.
\item The results in Section~\ref{sec:construction} imply that every perverse sheaf admits a presentation $\scP_1 \to \scP_2 \twoheadrightarrow \scG$ where $\sF_\cG(\scP_1)$ and $\sF_\cG(\scP_2)$ are flat over $\Lambda$.  Because the functors in the statement of Proposition~\ref{prop:CT-compare} are right exact, one can uniquely fill in the dotted arrow in the diagram below:
\[
\hbox{\tiny$
\begin{tikzcd}[column sep=tiny]
\CT_{\cB,\cG}(\scF \star^0 \scP_1) \ar[r] \ar[d, "\begin{array}{r}\text{flat case of}\\ \text{Prop.~\ref{prop:CT-compare}}\end{array}"'] &
  \CT_{\cB,\cG}(\scF \star^0 \scP_2) \ar[r, two heads] \ar[d, "\begin{array}{r}\text{flat case of}\\ \text{Prop.~\ref{prop:CT-compare}}\end{array}"'] &  
  \CT_{\cB,\cG}(\scF \star^0 \scG) \ar[r] \ar[d, dashed] & 0 \\
m_{\cT!}\widetilde{\CT}_{\cB,\cG}(p^*\pH^0(\scF \lboxtimes_\Lambda \scP_1)) \ar[r] &
  m_{\cT!}\widetilde{\CT}_{\cB,\cG}(p^*\pH^0(\scF \lboxtimes_\Lambda \scP_2)) \ar[r, two heads] &
  m_{\cT!}\widetilde{\CT}_{\cB,\cG}(p^*\pH^0(\scF \lboxtimes_\Lambda \scG)) \ar[r] &
  0
\end{tikzcd}$}
\]
In this way, one can deduce Propositions~\ref{prop:CT-compare} and~\ref{prop:CT-fiber-conv} in general.
\end{itemize}
\end{rmk}

\section{Monoidality}
\label{sec:monoidality}

\subsection{Statement and strategy}
\label{ss:monoidality-statement}

The goal of this section is to equip the total cohomology functor
\[
\sF_\cG \colon \Perv(\Hk_\cG,\Lambda) \to \modf_\Lambda
\]
(see~\S\ref{ss:total-cohomology})
with a monoidal structure, i.e., with a natural isomorphism
\begin{equation}\label{eqn:monoidality-defn}
\phi \colon \sF_\cG(\scF) \otimes_\Lambda \sF_\cG(\scG) \simto \sF_\cG(\scF \star^0 \scG)
\end{equation}
satisfying appropriate associativity and identity equations.
In the ``classical'' geometric Satake context (see~\cite{mv,br}) the monoidal structure on the fiber functor is constructed using the Be{\u\i}linson--Drinfeld Grassmannian and the interpretation of convolution as fusion, see~\cite[\S 6]{mv} or~\cite[\S 1.8]{br} for details. In our present setting we have no analogue of the Be{\u\i}linson--Drinfeld Grassmannian; we therefore have to argue differently.

As a warm-up, we treat an easy special case: that in which $\cG$ is replaced by $\cT$.

\begin{lem}
\label{lem:F-monoidal-torus}
For $\scF , \scG \in \Perv(\Hk_\cT,\Lambda)$ there is a natural isomorphism
\[
\sF_\cT(\scF \star^0 \scG) \cong \sF_\cT(\scF) \otimes_\Lambda \sF_\cT(\scG)
\]
making $\sF_\cT \colon \Perv(\Hk_\cT,\Lambda) \to \modf_\Lambda$ into a monoidal functor.
\end{lem}

\begin{proof}
Recall that the underlying topological spaces of $\Gr_\cT$ and $\Conv_\cT$ are discrete, associated with the sets $\bbX_*(T)_I$ and $\bbX_*(T)_I \times \bbX_*(T)_I$ respectively; see in particular~\eqref{eqn:GrT-X}. Using the identifications~\eqref{eqn:Conv-ab} and~\eqref{eqn:HkConv-ab}, along with the commutative diagram~\eqref{eqn:HkConv-ab-commute}, we see that
\begin{multline*}
\sF_\cT(\scF \star^0 \scG) \cong \coH^0(\Gr_\cT, \pH^0(h^* m_!p^*(\scF \lboxtimes_\Lambda \scG))) \\
\cong \coH^0(\Gr_\cT \times \Gr_\cT, \pH^0((h^* \scF) \lboxtimes_\Lambda (h^* \scG))).
\end{multline*}
Using again that $\Gr_\cT$ is discrete we
see that the last expression is isomorphic to
\[
\coH^0(\Gr_\cT, h^* \scF) \otimes_\Lambda \coH^0(\Gr_\cT, h^* \scG) = \sF_\cT(\scF) \otimes_\Lambda \sF_\cT(\scG),
\]
as desired.
\end{proof}

In view of Lemma~\ref{lem:F-monoidal-torus} and Proposition~\ref{prop:CT-fiber} (for $P=B$), to construct a monoidal structure on $\sF_{\cG}$ it suffices to construct a monoidal structure on the functor
\[
 \CT_{\cB,\cG} \colon \Perv(\Hk_\cG,\Lambda) \to \Perv(\Hk_\cT,\Lambda)
\]
with respect to the convolution product $\star^0$. This is exactly what we do in the rest of this section.

\subsection{A K\"unneth lemma}

\begin{lem}
\label{lem:GrB-conv}
Consider the diagram
\[
\begin{tikzcd}[row sep=tiny]
&  {[\Loop^+\cT\backslash\Gr_\cB]_{\et}} \times \Hk_\cB \\
{[\Loop^+\cT\backslash\Conv_\cB]_{\et}} \ar[ur, "p_{\cB}"] \ar[dr, "\tilde q_\cB"'] &&
{[\Loop^+\cT\backslash\Gr_\cB]_{\et}} \times \Gr_\cB \ar[ul, "\id \times h_{\cB}"'] \ar[dl, "q_\cB \times q_\cB"] \\
& {\hspace{-4ex} \HkConv_\cT \overset{\eqref{eqn:HkConv-ab}}{=} \Hk_\cT \times \Gr_\cT. \hspace{-4ex}}
\end{tikzcd}
\]
For $\scF \in \Dbc([\Loop^+\cT\backslash\Gr_\cB]_{\et},\Lambda)$ and $\scG \in \Dbc(\Hk_\cB,\Lambda)$, there is a natural isomorphism
\[
\tilde q_{\cB!} p_{\cB}^*(\scF \lboxtimes_\Lambda \scG)  \cong  (q_{\cB!}\scF) \lboxtimes_\Lambda (q_{\cB!}h_{\cB}^* \scG).
\]
\end{lem}

\begin{proof}
 It is enough to prove this when $\scF$ and $\scG$ are each supported on a single connected component of $\Gr_\cB$, say $\rmS_\lambda$ and $\rmS_\mu$, respectively for some $\lambda, \mu \in \bbX_*(T)_I$.  
 Note that all maps in the diagram induce bijections on the sets of connected components.   
 Taking appropriate connected components, our diagram restricts to
\begin{equation*}
\begin{tikzcd}[row sep=tiny]
& {[\Loop^+\cT\backslash \rmS_\lambda]_\et} \times [\Loop^+\cB\backslash \rmS_\mu]_{\et} \\
{[\Loop^+\cT\backslash \rmS_\lambda \wttimes \rmS_\mu]_{\et}} \ar[ur, "p_{\cB}\;\;"] \ar[dr, "\tilde q_\cB"'] &&
{[\Loop^+\cT\backslash \rmS_\lambda]_{\et}} \times \rmS_\mu \ar[ul, "\;\;\;\id \times h_{\cB}"'] \ar[dl, "q_\cB \times q_\cB"] \\
& { {[\Loop^+\cT\backslash\mathrm{pt}]_{\et}}}
\end{tikzcd}
\end{equation*}
where $\mathrm{pt}$ is the point given by $(t^\lambda,t^\mu)$.
By proper base change using the cartesian square~\eqref{eqn:ppr-cartesian}, we have $\pr_{1,\cB!}(p_\cB^*(\scF \lboxtimes_\Lambda \scG)) \cong \scF \lotimes_\Lambda R\Gammac(\Gr_\cB, h_\cB^*\scG)$, and hence a natural isomorphism
\[
\tilde q_{\cB!} p_{\cB}^*(\scF \lboxtimes_\Lambda \scG)  \cong  (q_{\cB!}\scF) \lboxtimes_\Lambda (q_{\cB!}h_{\cB}^* \scG),
\]
as desired.
\end{proof}

\subsection{Constant term functors and external product}

From Lemma~\ref{lem:GrB-conv} we will deduce the following compatibility statement between the constant term functors and the external product.

\begin{cor}
\label{cor:CT-monoidal-pull}
For $\scF, \scG \in \Dbc(\Hk_\cG,\Lambda)$ there is a natural isomorphism
\begin{equation}\label{eqn:CT-monoidal1}
\widetilde{\CT}_{\cB,\cG} \bigl( p_\cG^*(\scF \lboxtimes_\Lambda \scG) \bigr) \cong p_\cT^* \bigl( \CT_{\cB,\cG}(\scF) \lboxtimes_\Lambda \CT_{\cB,\cG}(\scG) \bigr).
\end{equation}
For $\scF, \scG \in \Perv(\Hk_\cG, \Lambda)$, this induces a natural isomorphism
\begin{equation}\label{eqn:CT-monoidal2}
\widetilde{\CT}_{\cB,\cG} \bigl( p_\cG^*\pH^0(\scF \lboxtimes_\Lambda \scG) \bigr) \cong p_\cT^* \pH^0 \bigl( \CT_{\cB,\cG}(\scF) \lboxtimes_\Lambda \CT_{\cB,\cG}(\scG) \bigr).
\end{equation}
\end{cor}

\begin{proof}
Let $w \colon \Hk_\cB\to\Hk_\cG$ be the map induced by the inclusion $\cB\subset\cG$.
Our arguments will exploit the following diagram, in which the left and middle squares commute:
\[
\hbox{\tiny$\begin{tikzcd}
\Hk_\cG \times \Hk_\cG\; & \;{[\Loop^+\cT\backslash \Gr_\cG]_{\et}}\times\Hk_\cG \ar[l,"h_{\cT,\cG}\times\id"'] &  {[\Loop^+\cT\backslash \Gr_\cB]_{\et}}\times\Hk_\cB \ar[l,"i_\cB\times w"'] & {[\Loop^+\cT\backslash\Gr_\cB]_{\et}} \times \Gr_\cB \ar[l, "\id \times h_\cB"'] \ar[d, "q_\cB\times q_\cB"]\\
\HkConv_\cG \ar[u, "p_\cG"'] & {[\Loop^+\cT \backslash \Conv_\cG]}_{\et} \ar[l, "\tilde h_{\cT,\cG}"']\ar[u,"p_\cB"'] & {[\Loop^+\cT \backslash \Conv_\cB]}_{\et} \ar[l, "\tilde\imath_\cB"'] \ar[u,"p_\cB"']\ar[r, "\tilde q_\cB"] & \HkConv_\cT \overset{\eqref{eqn:HkConv-ab}}{=} \Hk_\cT \times \Gr_\cT. 
\end{tikzcd}$}
\]
Using the commutativity and Lemma~\ref{lem:GrB-conv} we obtain isomorphisms 
\begin{multline*}
\widetilde{\CT}_{\cB,\cG}(p_\cG^*(\scF \lboxtimes_\Lambda \scG))
:= \tilde q_{\cB!}\tilde\imath_\cB^* \tilde h_{\cT,\cG}^*p_\cG^*(\scF \lboxtimes_\Lambda \scG) \\
\cong \tilde q_{\cB!}p_\cB^* ((i_\cB^*h_{\cT,\cG}^*\scF) \lboxtimes_\Lambda w^*\scG)
\cong (q_{\cB!}i_\cB^*h_{\cT,\cG}^*\scF) \lboxtimes_\Lambda (q_{\cB!}h_\cB^*w^*\scG).
\end{multline*}
Next we observe that
\[
q_{\cB!}i_\cB^*h_{\cT,\cG}^*\scF =: \CT_{\cB,\cG}(\scF) \quad \text{and} \quad
q_{\cB!} h_{\cB}^* w^* \scG \cong h_{\cT}^* \CT_{\cB,\cG}(\scG).
\]
Finally, under the isomorphism $\HkConv_\cT \overset{\eqref{eqn:HkConv-ab}}{=} \Hk_\cT \times \Gr_\cT$ the map $p_\cT$ corresponds to $\id\times h_\cT$, which allows us to convert the isomorphisms above into~\eqref{eqn:CT-monoidal1} (and thereby to finish the proof).

In view of Proposition~\ref{prop:exactness-CT}, we obtain~\eqref{eqn:CT-monoidal2} by applying $\pH^0$ to~\eqref{eqn:CT-monoidal1}.
\end{proof}

\subsection{End of the construction}

We are now ready to exhibit a monoidal structure on $\CT_{\cB,\cG}$.  As explained in~\S\ref{ss:monoidality-statement}, from there one obtains a monoidal structure~\eqref{eqn:monoidality-defn} on the functor $\sF_\cG$.

\begin{prop}\label{prop:CT-monoidal}
For $\scF, \scG \in \Perv(\Hk_\cG,\Lambda)$, there is a natural isomorphism
\begin{equation}\label{eqn:CT-mon1}
\CT_{\cB,\cG}(\scF \star^0 \scG) \cong \CT_{\cB,\cG}(\scF) \star^0 \CT_{\cB,\cG}(\scG).
\end{equation}
that makes $\CT_{\cB,\cG}$ into a monoidal functor.  
\end{prop}
\begin{proof}
By Proposition~\ref{prop:CT-compare} and Corollary~\ref{cor:CT-monoidal-pull}, the following diagram commutes up to natural isomorphism:
\[
\begin{tikzcd}[column sep=4.2em]
\Perv(\Hk_\cG,\Lambda) \times \Perv(\Hk_\cG, \Lambda) \ar[r, "\CT_{\cB,\cG} \times \CT_{\cB,\cG}" ] \ar[d, "p_\cG^*\pH^0({-} \lboxtimes_\Lambda {-})"]  \ar[dd, bend right=78, "({-})\star^0 ({-})"'] \ar[dr, phantom, "\text{Cor.~\ref{cor:CT-monoidal-pull}}" description] &
  \Perv(\Hk_\cT,\Lambda) \times \Perv(\Hk_\cT, \Lambda) \ar[d, "p_\cT^*\pH^0({-} \lboxtimes_\Lambda {-})"'] \ar[dd, bend left=78, "({-})\star^0 ({-})"] \\
\Perv(\HkConv_\cG, \Lambda)_\sph \ar[r, "\widetilde{\CT}_{\cB,\cG}" description] \ar[dr, phantom, "\text{Prop.~\ref{prop:CT-compare}}" description] &
  \Perv(\HkConv_\cT, \Lambda)_\sph \ar[d, "m_{\cT!}" ] \\
\Perv(\Hk_\cG,\Lambda) \ar[r, "\CT_{\cB,\cG}"'] & \Perv(\Hk_\cT,\Lambda)
\end{tikzcd}
\]
(Note that Proposition~\ref{prop:CT-compare} asserts the commutativity of a pentagon in this diagram, and \emph{not} of the square that would be obtained by including the arrow $m_{\cG!}\colon \Perv(\HkConv_\cG,\Lambda)_\sph \to \Perv(\Hk_\cG,\Lambda)$.  For this reason, $m_{\cG!}$ is omitted from the picture.)   Considering the outer square of this diagram, we obtain the isomor\-phism~\eqref{eqn:CT-mon1}.  We leave it to the reader to check that~\eqref{eqn:CT-mon1} is compatible with the associativity and identity constraints.
\end{proof}

\begin{rmk}
\label{rmk:F-monoidal-summary}
By construction, the monoidal structure on $\sF_\cG$
is characterized by the fact that it makes the following diagram commute:
\[
\begin{tikzcd}[column sep=large]
\sF_\cT(m_{\cT!}\pH^0(\widetilde{\CT}_{\cB,\cG}(p_\cG^*(\scF \lboxtimes_\Lambda \scG)))) \ar[dd, "\text{Prop.~\ref{prop:CT-fiber-conv}}"', "\wr"] \ar[r, "\text{Cor.~\ref{cor:CT-monoidal-pull}}", "\sim"'] & \sF_\cT(\CT_{\cB,\cG}(\scF) \star^0 \CT_{\cB,\cG}(\scG)) \ar[d, "\text{Lem.~\ref{lem:F-monoidal-torus}}", "\wr"'] \\
& \sF_\cT(\CT_{\cB,\cG}(\scF)) \otimes_\Lambda \sF_\cT(\CT_{\cB,\cG}(\scG)) \ar[d, "\text{Prop.~\ref{prop:CT-fiber}}", "\wr"'] \\
\sF_\cG(\scF \star^0 \scG) 
\ar[r, "\sim"', "\eqref{eqn:monoidality-defn}"] &
\sF_\cG(\scF) \otimes_\Lambda \sF_\cG(\scG).
\end{tikzcd}
\]
\end{rmk}

\begin{rmk}\label{monoidality-remark}
There is another construction of a monoidal structure on the constant term functor which does not use the Be{\u\i}linson--Drinfeld Grassmannian in~\cite[Proof of Proposition~4.4]{yu-integral}. (The geometric setting considered in that reference is different from ours, but shares similar formal properties.)
Unfortunately, this proof is based on the false claim that the filtrations on the total cohomology arising from the semi-infinite orbits and their opposites are complementary to each other.\footnote{This problem was pointed out to the fourth named author by S.~Lysenko several years ago. It was also discussed during Scholze's geometrization lectures in the winter term 2020/21 where the third named author was present.}
This seems to be a common misconception in the literature and appears in several places including \cite[Proof of Theorem~3.6]{mv}, \cite[Proof of Theorem~5.3.9(3)]{zhu-notes} and \cite[Proof of Theorem~3.16]{HainesRicharz_TestFunctions}.
The proof of Proposition~\ref{prop:CT-monoidal} replaces the false argument with the direct computation in Corollary~\ref{cor:CT-conv-split}.
\end{rmk}

\section{A bialgebra governing perverse sheaves}
\label{sec:construction}

\subsection{Statement}

The main result of this section is
the following Theorem~\ref{thm:coalgebra}, whose proof will be finished in~\S\ref{ss:alg-structure}. Given a $\Lambda$-coalgebra $C$, we denote by $\comod_C$ its category of right comodules which are finitely generated over $\Lambda$. If $C$ is a bialgebra, then the tensor product $\otimes_\Lambda$ equips this category with a monoidal structure.

\begin{thm}
	\label{thm:coalgebra}
	For $\Lambda \in \{\bK,\bO,\bk\}$
	there exists a canonical $\Lambda$-bialgebra $\rmB_{\cG}(\Lambda)$ and an equivalence of monoidal categories
	\[
	\mathsf{S}_\cG \colon \bigl( \Perv(\Hk_{\cG},\Lambda), \star^0 \bigr) \simto \bigl( \comod_{\rmB_{\cG}(\Lambda)}, \otimes_\Lambda \bigr).
	\]
	Moreover, $\rmB_{\cG}(\bO)$ is flat over $\bO$ and 
	there exist canonical isomorphisms of $\bk$- and $\bK$-bialgebras
	\begin{equation}
		\label{eqn:BG-ext-scalars}
		\bk \otimes_\bO \rmB_{\cG}(\bO) \simto \rmB_{\cG}(\bk), \quad
		\bK \otimes_\bO \rmB_{\cG}(\bO) \simto \rmB_{\cG}(\bK)
	\end{equation}
	respectively, compatible with the change-of-scalars functors 
	\begin{gather*}
		\pH^0 (\bk \lotimes_\bO (-)) \colon \Perv(\Hk_{\cG},\bO) \to \Perv(\Hk_{\cG},\bk), \\
		\bK \otimes_{\bO} (-) \colon \Perv(\Hk_{\cG},\bO) \to \Perv(\Hk_{\cG},\bK), \\
		\Perv(\Hk_{\cG},\bk) \to \Perv(\Hk_{\cG},\bO)
	\end{gather*}
	in the natural way.
\end{thm}

The proof of Theorem~\ref{thm:coalgebra} is very similar to that of the corresponding claim in the context of the ``ordinary'' geometric Satake equivalence; see~\cite[Section~11]{mv} and~\cite[\S 1.13.1]{br}. We will not repeat the proofs that can be copied from these references.

\subsection{Weight functors}

Recall the functor
\[
\sF_\cG \colon \Perv(\Hk_\cG,\Lambda) \to \modf_\Lambda
\]
considered in~\S\ref{ss:total-cohomology}. We also have a similar functor
\[
\sF_\cT \colon \Perv(\Hk_\cT,\Lambda) \to \modf_\Lambda
\]
for the group $\cT$. Since $\Gr_\cT$ is discrete with underlying set $\bbX_*(T)_I$ (see~\S\ref{ss:semi-inf-orbits}), we have a canonical identification
\begin{equation}
	\label{eqn:Perv-Hk-T}
	\Perv(\Hk_\cT,\Lambda) = \modf_\Lambda^{\bbX_*(T)_I}
\end{equation}
where the right-hand side denotes the category of finitely generated $\bbX_*(T)_I$-graded $\Lambda$-modules. Via this identification, the functor $\sF_\cT$ sends an $\bbX_*(T)_I$-graded $\Lambda$-modules to the underlying $\Lambda$-module.

Using~\eqref{eqn:Perv-Hk-T}, the functor $\CT_{\cB,\cG}$ (see~\S\ref{ss:CT-functors}) can be seen as a functor
\[
\Perv(\Hk_\cG,\Lambda) \to \modf_\Lambda^{\bbX_*(T)_I}.
\]
If we denote, for any $\lambda \in \bbX_*(T)_I$, by $\sF_{\cG,\lambda}$ the composition of this functor with the functor $\modf_\Lambda^{\bbX_*(T)_I} \to \modf_\Lambda$ sending an $\bbX_*(T)_I$-graded $\Lambda$-module to its $\lambda$-component, then by Proposition~\ref{prop:CT-fiber} (for $\cP = \cB$) we have a canonical isomorphism of functors
\begin{equation}
	\label{eqn:FG-weight-decomposition}
	\sF_\cG \cong \bigoplus_{\lambda \in \bbX_*(T)_I} \sF_{\cG,\lambda}.
\end{equation}
The functor $\sF_{\cG,\lambda}$ is called the \emph{weight functor} associated with $\lambda$. Explicitly, for any $\lambda \in \bbX_*(T)_I$ and $\scF \in \Perv(\Hk_\cG,\Lambda)$ we have
\[
\sF_{\cG,\lambda}(\scF) = \mathsf{H}_c^{\langle \lambda, 2\rho \rangle}(\rmS_\lambda, (h^* \scF)_{|\rmS_\lambda}) \cong \mathsf{H}^{\langle \lambda, 2\rho \rangle}_{\rmT_\lambda}(h^* \scF),
\]
where 
the isomorphism is provided by Proposition~\ref{prop:Braden} (again, for $\cP = \cB$).

\subsection{Preliminaries on standard and costandard perverse sheaves}
\label{ss:preliminaries-standard-costandard}

Recall,
for $\mu \in \bbX_*(T)_I^+$, the objects
$\scJ_!(\mu,\Lambda)$ and
$\scJ_*(\mu,\Lambda)$
considered in~\S\ref{ss:Perv-Gr-Hk}.

\begin{lem}
	\label{lem:semisimplicity}
	In case $\Lambda=\bK$, 
	the category $\Perv(\Hk_\cG,\Lambda)$ is semisimple. In particular, the natural morphism $\scJ_!(\mu,\Lambda) \to \scJ_*(\mu,\Lambda)$ is an isomorphism.
\end{lem}

\begin{proof}
	Like in the setting of the ordinary geometric Satake equivalence (see~\cite[\S 1.4]{br}),
	the claim follows from the fact that the parity of the dimension of Schubert varieties is constant on each connected component of $\Gr_\cG$, see~\eqref{eqn:conn-comp-Gr} and~\eqref{eqn:dim-orbits}, and that the simple objects in $\Perv(\Hk_\cG,\Lambda)$ are parity complexes in the sense of~\cite{jmw}.
\end{proof}

\begin{lem}
	\phantomsection
	\label{lem:properties-J!*}
	\begin{enumerate}
		\item
		\label{it:properties-J!*-1}
		For $\mu \in \bbX_*(T)_I^+$ and $\nu \in \bbX_*(T)_I$, the $\Lambda$-module
		\[
		\sF_{\cG,\nu}(\scJ_!(\mu,\Lambda)), \quad \text{resp.} \quad \sF_{\cG,\nu}(\scJ_*(\mu,\Lambda)),
		\]
		is free with a canonical basis parametrized by the irreducible components of $\Gr_{\cG}^\mu \cap \rmS_\nu$, resp.~$\Gr_\cG^\mu \cap \rmT_\nu$.
		\item
		\label{it:properties-J!*-2}
		For any $\mu \in \bbX_*(T)_I^+$ 
		there exist canonical isomorphisms
		\begin{gather*}
			\scJ_!(\mu,\bK) \cong \bK \lotimes_{\bO} \scJ_!(\mu,\bO), \quad \scJ_*(\mu,\bK) \cong \bK \lotimes_{\bO} \scJ_*(\mu,\bO), \\
			\scJ_!(\mu,\bk) \cong \bk \lotimes_{\bO} \scJ_!(\mu,\bO), \quad \scJ_*(\mu,\bk) \cong \bk \lotimes_{\bO} \scJ_*(\mu,\bO).
		\end{gather*}
		\item
		\label{it:properties-J!*-3}
		In case $\Lambda=\bO$, for any $\mu \in \bbX_*(T)_I^+$ the canonical morphism
		\[
		\scJ_!(\mu,\Lambda) \to \scJ_*(\mu,\Lambda)
		\]
		is injective.
	\end{enumerate}
\end{lem}

\begin{proof}
	\eqref{it:properties-J!*-1}
	The proof is the same as for~\cite[Proposition~1.11.1]{br}.
	
	\eqref{it:properties-J!*-2}
	The proof is the same as for~\cite[Proposition~1.11.3]{br}.
	
	\eqref{it:properties-J!*-3}
	By the same considerations as in~\cite[Lemma~1.11.5]{br}, the claim follows from Lemma~\ref{lem:semisimplicity}.
\end{proof}

\subsection{Representability}
\label{ss:representability}

Consider a closed subscheme $Z \subset \Gr_\cG$ whose underlying topological subspace is a union of finitely many $\Loop^+ \cG$-orbits. We can then consider the quotient stack $[\Loop^+ \cG \backslash Z]_{\et}$, and the corresponding full subcategory
\[
\Perv([\Loop^+ \cG \backslash Z]_\et,\Lambda) \subset \Perv(\Hk_\cG,\Lambda).
\]
In fact, the action of $\Loop^+\cG$ on $Z$ factors through an action of $\Loop^+_n \cG$ for some $n \geq 0$, and we then have
\[
\Perv([\Loop^+ \cG \backslash Z]_{\et},\Lambda) = \Perv([\Loop^+_n \cG \backslash Z]_{\et},\Lambda).
\]

Fix some $\nu \in \bbX_*(T)_I$ such that $\rmT_\nu \cap Z \neq \varnothing$, and consider the immersion
\[
i_{Z,\nu} \colon \rmT_\nu \cap Z \to Z
\]
and the action and projection morphisms
\[
a_{Z,\nu}, p_{Z,\nu} \colon \Loop^+_n \cG \times Z \to Z.
\]
One checks as in~\cite[Proposition~1.12.1]{br} that the complex
\begin{equation}
	\label{eqn:complex-Pnu}
	(a_{Z,\nu})_! (p_{Z,\nu})^! (i_{Z,\nu})_! \underline{\Lambda}_{\rmT_\nu \cap Z}[-\langle \nu, 2\rho \rangle]
\end{equation}
is concentrated in nonpositive perverse degrees, and that
if we set
\[
\scP_Z(\nu,\Lambda) := \pH^0((a_{Z,\nu})_! (p_{Z,\nu})^! (i_{Z,\nu})_! \underline{\Lambda}_{\rmT_\nu \cap Z}[-\langle \nu, 2\rho \rangle]),
\]
the perverse sheaf $\scP_Z(\nu,\Lambda)$ is a projective object in $\Perv([\Loop^+ \cG \backslash Z]_\et,\Lambda)$ which represents the restriction of $\sF_{\cG,\nu}$ to this subcategory. In particular, this object does not depend on the choice of $n$.

Set
\[
\bbX_Z = \{ \nu \in \bbX_*(T)_I \mid Z \cap \rmT_\nu \neq \varnothing \} = 
\bigcup_{\substack{\lambda \in \bbX_*(T)^+_I \\ |\Gr^\lambda_\cG| \subset |Z|}} W_0 \cdot \lambda,
\]
where the equality follows from Lemma~\ref{lem:dim-estimate}. (This set is clearly finite.)
In view of~\eqref{eqn:FG-weight-decomposition}, setting
\[
\scP_Z(\Lambda) := \bigoplus_{\nu \in \bbX_Z} \scP_Z(\nu,\Lambda)
\]
one obtains a projective object in $\Perv([\Loop^+ \cG \backslash Z]_\et,\Lambda)$ which represents the restriction of $\sF_{\cG}$ to this subcategory. In fact, as in~\cite[\S 1.12.1]{br}, this object is a projective generator of the category $\Perv([\Loop^+ \cG \backslash Z]_\et,\Lambda)$.

We finish this subsection with the discussion of a property which will be used in Appendix~\ref{app:equiv-const}. Consider a \emph{locally closed} subscheme $X \subset \Gr_\cG$ whose underlying topological subspace is a union of finitely many $\Loop^+ \cG$-orbits. Choose closed subschemes $Y \subset Z \subset \Gr_\cG$ whose underlying topological subspace is a union of finitely many $\Loop^+ \cG$-orbits and such that $X = Z \smallsetminus Y$, and denote by $j \colon X \to Z$ the open embedding. If $\nu \in \bbX_Z \smallsetminus \bbX_Y$ (or, in other words, if $t^\nu \in |X|$), the same considerations as in~\cite[Remark~1.5.8(2)]{br} show that for any $\scG \in \Perv([\Loop^+ \cG \backslash X]_\et, \Lambda)$ we have 
\begin{equation}
	\label{eqn:vanishing-local-cohom}
	\mathsf{H}^{k}_{X \cap \rmT_\nu}(X, \scG) = 0 \quad \text{unless $k=\langle \nu, 2\rho \rangle$,}
\end{equation}
so that the functor
\[
\sF_{\cG,\nu}^X := \mathsf{H}^{\langle \nu, 2\rho \rangle}_{X \cap \rmT_\nu}(X, -) \colon \Perv([\Loop^+ \cG \backslash X]_\et, \Lambda) \to \modf_\Lambda
\]
is exact (by consideration of an appropriate long exact sequence).

\begin{lem}
	\label{lem:representability-F-X}
	The functor $\sF_{\cG,\nu}^X$ is represented by
	the perverse sheaf $j^* \scP_Z(\nu, \Lambda)$.
\end{lem}

\begin{proof}
	For $\scF \in \Perv([\Loop^+ \cG \backslash X]_\et, \Lambda)$, by adjunction we have
	\[
	\Hom(j^* \scP_Z(\nu, \Lambda), \scF) \cong \Hom(\scP_Z(\nu,\Lambda), \pH^0(j_* \scF)).
	\]
	Since $\scP_Z(\nu,\Lambda)$ represents $\sF_{\cG,\nu}$ on $\Perv([\Loop^+ \cG \backslash Z]_\et, \Lambda)$, the right-hand side identifies with
	\[
	\sF_{\cG,\nu}(\pH^0(j_* \scF)) = \mathsf{H}^{\langle \nu, 2\rho \rangle}_{\rmT_\nu}(\Gr_\cG, \pH^0(j_* \scF)).
	\]
	Now, by Remark~\ref{rmk:exactness-CT-vanishing} and standard considerations involving perverse truncation triangles (as in~\cite[Lemma~1.10.7]{br}), one sees that
	\[
	\mathsf{H}^{\langle \nu, 2\rho \rangle}_{\rmT_\nu}(\Gr_\cG, \pH^0(j_* \scF)) = \mathsf{H}^{\langle \nu, 2\rho \rangle}_{\rmT_\nu}(\Gr_\cG, j_* \scF).
	\]
	By base change the right-hand side identifies with
	$\mathsf{H}^{\langle \nu, 2\rho \rangle}_{X \cap \rmT_\nu}(X, \scF)$, which finishes the proof.
\end{proof}

It follows from Lemma~\ref{lem:representability-F-X} and the comments preceding it that the perverse sheaf
\[
\scP_X(\nu, \Lambda) = j^* \scP_Z(\nu, \Lambda)
\]
does not depend on the choice of $Y$ and $Z$, and is projective.

Consider now the projective object
\[
\scP_X(\Lambda) = \bigoplus_{\substack{\nu \in \bbX \\ t^\nu \in |X|}} \scP_X(\nu,\Lambda) \quad \in \Perv([\Loop^+ \cG \backslash X]_\et, \Lambda).
\]
The same arguments as in the proof of Proposition~\ref{prop:exactness-CT} show that the functor $\Hom(\scP_X(\Lambda), -)$ does not kill any nonzero object; it follows that $\scP$ is a projective generator of the category $\Perv([\Loop^+ \cG \backslash X]_\et, \Lambda)$.

\subsection{Structure of projective objects}
\label{ss:structure-projective}

Let us now consider two $\Loop^+ \cG$-stable closed subschemes $Y,Z \subset \Gr_\cG$ whose underlying topological subspace is a union of finitely many $\Loop^+ \cG$-orbits and such that $Y \subset Z$. If we denote by $i \colon Y \to Z$ the closed immersion, then as in~\cite[Proposition~1.12.2]{br} one checks that there exists a canonical isomorphism
\[
\scP_Y(\Lambda) \cong \pH^0(i^* \scP_Z(\Lambda)),
\]
and that the morphism
\[
\scP_Z(\Lambda) \to \pH^0(i_* i^* \scP_Z(\Lambda)) = i_* \scP_Y(\Lambda)
\]
induced by adjunction is surjective.

Next, using Lemma~\ref{lem:properties-J!*} and the same arguments as for~\cite[Proposition~1.12.3]{br}, one checks that, for any closed subscheme $Z \subset \Gr_\cG$ whose underlying topological subspace is a union of finitely many $\Loop^+ \cG$-orbits:
\begin{enumerate}
	\item
	\label{it:properties-P-1}
	the object $\scP_Z(\Lambda)$ admits a finite filtration with associated graded
	\[
	\bigoplus_{\substack{\mu \in \bbX_*(T)_I^+ \\ | \Gr_\cG^\mu | \subset | Z |}} \sF_\cG(\scJ_*(\mu,\Lambda)) \otimes_\Lambda \scJ_!(\mu,\Lambda);
	\]
	\item
	\label{it:properties-P-2}
	there exist canonical isomorphisms
	\[
	\scP_Z(\bK) \cong \bK \otimes_\bO \scP_Z(\bO), \quad
	\scP_Z(\bk) \cong \bk \lotimes_\bO \scP_Z(\bO);
	\]
	\item
	\label{it:properties-P-3}
	the $\bO$-module $\sF_\cG(\scP_Z(\bO))$ is free of finite rank, and 
	there exist canonical isomorphisms
	\[
	\sF_\cG(\scP_Z(\bK)) \cong \bK \otimes_\bO \sF_\cG(\scP_Z(\bO)), \quad
	\sF_\cG(\scP_Z(\bk)) \cong \bk \otimes_\bO \sF_\cG(\scP_Z(\bO)).
	\]
\end{enumerate}

Let us note for later use the following consequence of~\eqref{it:properties-P-2}.

\begin{lem}
	\label{lem:change-scalars-P}
	For any $\mu \in \bbX_*(T)^+_I$ such that $|\Gr_\cG^\mu| \subset |Z|$ 
	there exist canonical isomorphisms
	\[
	\scP_Z(\mu,\bK) \cong \bK \otimes_\bO \scP_Z(\mu,\bO), \quad
	\scP_Z(\mu,\bk) \cong \bk \lotimes_\bO \scP_Z(\mu,\bO).
	\]
\end{lem}

\begin{proof}
	The first isomorphism is clear from the construction of the objects $\scP_Z(\mu,\bK)$ and $\scP_Z(\mu,\bO)$ and t-exactness of the functor $\bK \otimes_{\bO} (-)$. For the second one, property~\eqref{it:properties-P-2} above implies that $\bk \lotimes_\bO \scP_Z(\mu,\bO)$ is a perverse sheaf. On the other hand, since the complex~\eqref{eqn:complex-Pnu} is concentrated in nonpositive perverse degrees, and since its version over $\bk$ is obtained from the version over $\bO$ by application of the functor $\bk \lotimes_\bO (-)$, we have $\scP_Z(\mu,\bk) \cong \pH^0 (\bk \lotimes_\bO \scP_Z(\mu,\bO))$. The desired claim follows.
\end{proof}

\begin{rmk}
	\label{rmk:hw-category}
	In case $\Lambda$ is a field (i.e.~if $\Lambda=\bK$ or $\Lambda=\bk$), as in~\cite[Proposition~1.12.4]{br} one can check that for any
	$\Loop^+ \cG$-stable closed subscheme $Z \subset \Gr_\cG$ whose underlying topological subspace is a union of finitely many $\Loop^+ \cG$-orbits the category $\Perv([\Loop^+\cG \backslash Z]_\et, \Lambda)$ is a highest weight category with weight poset
	\[
	\{\lambda \in \bbX_*(T)^+_I \mid |\Gr_\cG^\lambda| \subset |Z|\}
	\]
	and standard, resp.~costandard, objects the objects $\scJ_!(\lambda,\Lambda)$, resp.~$\scJ_*(\lambda,\Lambda)$. (For generalities on highest weight categories, see~\cite[\S 7]{riche-hab}.)
	
	More generally, for any $\Loop^+ \cG$-stable \emph{locally closed} subscheme $X \subset \Gr_\cG$ whose underlying topological subspace is a union of finitely many $\Loop^+ \cG$-orbits the category $\Perv([\Loop^+\cG \backslash X]_\et, \Lambda)$ is a highest weight category with weight poset
	\[
	\{\lambda \in \bbX_*(T)^+_I \mid |\Gr_\cG^\lambda| \subset |X|\}
	\]
	and standard, resp.~costandard, objects the objects $\scJ_!(\lambda,\Lambda)_{|X}$, resp.~$\scJ_*(\lambda,\Lambda)_{|X}$. In fact, writing $X=Z \smallsetminus Y$ with $Y \subset Z \subset \Gr_{\cG}$ as above, $\Perv([\Loop^+\cG \backslash X]_\et, \Lambda)$ is the Serre quotient of the category $\Perv([\Loop^+\cG \backslash Z]_\et, \Lambda)$ by the Serre subcategory $\Perv([\Loop^+\cG \backslash Y]_\et, \Lambda)$, so that the claim follows from~\cite[Lemma~7.8]{riche-hab}.
\end{rmk}

\subsection{Construction of \texorpdfstring{$\rmB_\cG(\Lambda)$}{BG}}
\label{ss:construction-BG}

If $Z$ is as in~\S\ref{ss:representability}, we set
\[
\rmA_Z(\Lambda) := \End_{\Perv([\Loop^+ \cG \backslash Z]_\et, \Lambda)}(\scP_Z(\Lambda))^{\mathrm{op}}.
\]
Then, since $\scP_Z(\Lambda)$ represents the restriction of $\sF_\cG$, we have
\[
\rmA_Z(\Lambda) \cong \sF_\cG(\scP_Z(\Lambda))
\]
as $\Lambda$-modules, hence $\rmA_Z(\Lambda)$ is free of finite rank over $\Lambda$ by~\eqref{it:properties-P-3} in~\S\ref{ss:structure-projective}. Since $\scP_Z(\Lambda)$ is a projective generator of $\Perv([\Loop^+ \cG \backslash Z]_\et, \Lambda)$, by a variant of the Gabriel--Popescu theorem (see e.g.~\cite[Proposition~1.13.1]{br}) one sees that there exists an equivalence of abelian categories
\[
\mathsf{S}'_Z \colon \Perv([\Loop^+ \cG \backslash Z]_\et, \Lambda) \simto \modf_{\rmA_Z(\Lambda)}
\]
(where the right-hand side is the category of finitely generated $\rmA_Z(\Lambda)$-modules) whose composition with the forgetful functor $\modf_{\rmA_Z(\Lambda)} \to \modf_\Lambda$ is the restriction of $\sF_\cG$. If we set
\[
\rmB_Z(\Lambda) := \Hom_{\Lambda}(\rmA_Z(\Lambda),\Lambda),
\]
we therefore obtain a $\Lambda$-coalgebra such that there exists a canonical equivalence of abelian categories
\[
\mathsf{S}_Z \colon \Perv([\Loop^+ \cG \backslash Z]_\et, \Lambda) \simto \comod_{\rmB_Z(\Lambda)}
\]
whose composition with the forgetful functor $\comod_{\rmB_Z(\Lambda)} \to \modf_\Lambda$ is the restriction of $\sF_\cG$.

In the setting of~\S\ref{ss:structure-projective}, if $Y \subset Z$ the functor $\pH^0 \circ i^*$ induces a morphism $f'_{Y,Z} \colon \rmA_Z(\Lambda) \to \rmA_Y(\Lambda)$ such that the diagram
\[
\begin{tikzcd}
	\Perv([\Loop^+ \cG \backslash Y]_{\et}, \Lambda) \ar[r, "\mathsf{S}_Y"] \ar[d, "i_*"'] & \modf_{\rmA_Y(\Lambda)} \ar[d] \\
	\Perv([\Loop^+ \cG \backslash Z]_{\et}, \Lambda) \ar[r, "\mathsf{S}_Z"] & \modf_{\rmA_Z(\Lambda)}
\end{tikzcd}
\]
commutes, where the right-hand vertical arrow is the restriction-of-scalars functor associated with $f'_{Y,Z}$. Passing to duals we deduce a morphism of coalgebras $f_{Y,Z} \colon \rmB_Y(\Lambda) \to \rmB_Z(\Lambda)$ such that the diagram
\[
\begin{tikzcd}
	\Perv([\Loop^+ \cG \backslash Y]_{\et}, \Lambda) \ar[r, "\mathsf{S}_Y"] \ar[d, "i_*"'] & \comod_{\rmB_Y(\Lambda)} \ar[d] \\
	\Perv([\Loop^+ \cG \backslash Z]_{\et}, \Lambda) \ar[r, "\mathsf{S}_Z"] & \comod_{\rmB_Z(\Lambda)}
\end{tikzcd}
\]
commutes, where the right-hand vertical arrow is the functor induced by $f_{Y,Z}$.

Finally we set
\[
\rmB_\cG(\Lambda) = \mathrm{colim}_Z \, \rmB_Z(\Lambda)
\]
where $Z$ runs over the closed subschemes of $\Gr_\cG$ whose underlying topological subspace is a finite union of $\Loop^+ \cG$-orbits. Then $\rmB_\cG(\Lambda)$ is a $\Lambda$-coalgebra, and since
\[
\Perv(\Hk_\cG,\Lambda) = \mathrm{colim}_Z \, \Perv([\Loop^+\cG \backslash Z]_{\et}, \Lambda)
\]
we obtain an equivalence of categories
\[
\mathsf{S}_\cG \colon \Perv(\Hk_\cG,\Lambda) \simto \comod_{\rmB_\cG(\Lambda)}
\]
whose composition with the forgetful functor $\comod_{\rmB_\cG(\Lambda)} \to \modf_\Lambda$ is $\sF_\cG$. Note that in case $\Lambda=\bO$, $\rmB_\cG(\bO)$ is flat over $\bO$, as a colimit of flat $\bO$-modules.

Property~\eqref{it:properties-P-3} in~\S\ref{ss:structure-projective} ensures that there exist canonical isomorphisms
\[
\bK \otimes_\bO \rmB_\cG(\bO) \simto \rmB_\cG(\bK), \quad \bk \otimes_\bO \rmB_\cG(\bO) \simto \rmB_\cG(\bk)
\]
such that the obvious diagrams involving the change-of-scalars functors
commute.

\subsection{Algebra structure}
\label{ss:alg-structure}

To conclude the proof of Theorem~\ref{thm:coalgebra}
we need to endow $\rmB_\cG(\Lambda)$ with an algebra structure such that $\mathsf{S}_\cG$ is monoidal. Fix three closed subschemes $X,Y,Z \subset \Gr_\cG$ whose underlying topological subspaces are unions of finitely many $\Loop^+\cG$-orbits, and such that the restriction of the morphism $m$ from~\S\ref{ss:convolution-schemes} to $X \wttimes Y$ factors through a morphism $X \wttimes Y \to Z$. The tensor product of identity morphisms provides a canonical element in
\begin{multline*}
	A_X(\Lambda) \otimes_\Lambda A_Y(\Lambda) \cong \sF_\cG(\scP_X(\Lambda)) \otimes_\Lambda \sF_\cG(\scP_Y(\Lambda)) \\
	\cong \sF_\cG(\scP_X(\Lambda) \star^0 \scP_Y(\Lambda)) \cong \Hom_{\Perv(\Hk_\cG,\Lambda)}(\scP_Z(\Lambda), \scP_X(\Lambda) \star^0 \scP_Y(\Lambda))
\end{multline*}
(where the second isomorphism is provided by the monoidal structure on $\sF_\cG$ constructed in Section~\ref{sec:monoidality}).
Applying the functor $\sF_\cG$ we deduce a morphism $A_Z(\Lambda) \to A_X(\Lambda) \otimes_\Lambda A_Y(\Lambda)$ and then, dualizing, a morphism
\[
\rmB_X(\Lambda) \otimes_\Lambda \rmB_Y(\Lambda) \to \rmB_Z(\Lambda).
\]
Passing to colimits we finally deduce a morphism
\[
\rmB_\cG(\Lambda) \otimes_\Lambda \rmB_\cG(\Lambda) \to \rmB_\cG(\Lambda).
\]
It is not difficult to check that this map defines an associative multiplication morphism, with unit element the image of the unit element in $\rmB_{\Gr_\cG^0}(\Lambda)=\Lambda$. Moreover, combined with the comultiplication considered above, this multiplication morphism endows $\rmB_\cG(\Lambda)$ with a bialgebra structure, such that the equivalence of categories $\mathsf{S}_\cG$ admits a canonical monoidal structure.

This finishes the proof of Theorem~\ref{thm:coalgebra}.

\subsection{Morphisms related to constant term functors}

We finish this section with the construction of morphisms of coalgebras induced by the constant term functors.

Let us consider a parabolic subgroup $P \subset G$ containing $A$, and its Levi factor $M$ containing $A$. 
We can then consider the $\Lambda$-coalgebra $\rmB_\cG(\Lambda)$ associated with $\cG$, in other words with $G$ and its special facet $\fa$, but also the $\Lambda$-coalgebra $\rmB_{\cM}(\Lambda)$ associated with $\cM$, i.e.~with the reductive group $M$ and its special facet $\fa_M$ (see~\S\ref{ss:group-theory}). Here, $\cM$ is the scheme-theoretic closure of $M$ in $\cG$, see~\eqref{eqn:parabolics-cocharacters-2}.

\begin{prop}
	\label{prop:morphism-coalg-Levi}
	For any parabolic subgroup $P$ containing $A$, there exists a canonical morphism of coalgebras
	\[
	\mathrm{res}_{\cP,\cG} \colon \rmB_{\cG}(\Lambda) \to \rmB_{\cM}(\Lambda)
	\]
	such that the diagram
	\[
	\begin{tikzcd}
		\Perv(\Hk_{\cG},\Lambda) \ar[r, "\mathsf{S}_\cG"] \ar[d, "\CT_{\cP,\cG}"'] & \comod_{\rmB_{\cG}(\Lambda)} \ar[d] \\
		\Perv(\Hk_{\cM},\Lambda) \ar[r, "\mathsf{S}_\cM"] & \comod_{\rmB_{\cM}(\Lambda)}
	\end{tikzcd}
	\]
	commutes, where the right vertical arrow is the functor induced by $\mathrm{res}_{\cP,\cG}$. Moreover these morphisms are compatible with change of scalars in the obvious way, and with parabolic restriction in the sense that given parabolic subgroups $P \subset P' \subset G$ containing $A$, with Levi factors $M \subset M'$ containing $A$, we have
	\begin{equation}
		\label{eqn:restriction-morphisms}
		\mathrm{res}_{\cP \cap \cM',\cM'} \circ \mathrm{res}_{\cP',\cG} = \mathrm{res}_{\cP,\cG},
	\end{equation}
	where $\cM'$ is the scheme-theoretic closure of $M'$ in $\cG$.
\end{prop}

\begin{proof}
	The existence of $\mathrm{res}_{\cP,\cG}$ follows from Proposition~\ref{prop:CT-fiber}, using the
	standard fact that any exact functor between categories of comodules (finitely generated over $\Lambda$) over some $\Lambda$-coalgebras compatible with the forgetful functor to $\modf_\Lambda$ is induced by a morphism of coalgebras; see e.g.~\cite[Proposition~1.2.6(2)]{br} (where the assumption that the coalgebras are defined over a field is not necessary). The compatibility with change of scalars is obvious. The equality~\eqref{eqn:restriction-morphisms} follows from Lemma~\ref{lem:transitivity-CT}.
\end{proof}

\begin{rmk}
	We emphasize that at this stage $\mathrm{res}_{\cP,\cG}$ is only a morphism of \emph{coalgebras}. We will later prove that it is also compatible with products, hence a morphism of bialgebras (see Proposition~\ref{prop:res-bialg}), but this fact is not clear for now.
\end{rmk}

\section{The absolute case}
\label{sec:monoidality-unramified}

\subsection{Absolute variants}
\label{ss:absolute-var}

Recall from~\S\ref{ss:loop-gps} that this paper treats two parallel geometric settings: the ``ramified'' case (including ``unramified groups'') involving the group $\cG$ (where geometric objects live over $\F$), and the ``absolute'' case involving the group $F^s \pot{z} \otimes_F G$ (where geometric objects live over the field $F^s$).  
So far in this paper, we have considered only the ramified case.
In fact, up to replacing $F^s$ by its algebraic closure,\footnote{Note that, if $F^{\mathrm{alg}}$ is an algebraic closure of $F^s$ the extension $F^s \to F^{\mathrm{alg}}$ is purely inseparable, hence the morphism $\Spec(F^{\mathrm{alg}}) \to \Spec(F^s)$ is a universal homeomorphism, see~\cite[\href{https://stacks.math.columbia.edu/tag/0BR5}{Tag 0BR5}]{stacks-project}; it therefore induces equivalences between the appropriate categories of sheaves.} the absolute case is a special case of the ramified one.

In this section, we consider the absolute setting.  
We point out that all the constructions and results of Sections \ref{sec:aff-Grass}--\ref{sec:construction} are applicable in the absolute case. 
We adapt notation from those sections by replacing the subscript ``$\cG$'' by ``$G$'' throughout.  
For instance, we write $\Gr_G$ instead of $\Gr_\cG$, $F_G$ instead of $F_\cG$, $\Loop^+ G$ instead of $\Loop^+ \cG$, and so on. (To be completely formal, these objects should be be denoted $\Gr_{F^s \pot{z} \otimes_F G}$, $\sF_{F^s \pot{z} \otimes_F G}$, $\Loop^+(F^s \pot{z} \otimes_F G)$, etc.)

The absolute counterparts of most statements from Sections \ref{sec:aff-Grass}--\ref{sec:construction} have appeared in the literature on the ``usual'' geometric Satake equivalence, in particular in~\cite{mv,br}.  
However, there is one significant point of departure: the monoidal structure on $\sF_G$ from Section~\ref{sec:monoidality} is constructed very differently from the one described in~\cite{mv}, and it is a priori not clear whether they coincide.

The main result of this section (Theorem~\ref{thm:monoidal-agree}) asserts that when $\Lambda = \bK$, they do coincide.  The most important consequence is that in this special case, the monoidal structure from Section~\ref{sec:monoidality} admits a \emph{commutativity constraint}.  We will see later in the paper (Corollary~\ref{cor:morph-coalg-Psi} and Proposition~\ref{prop:commutativity}) how to transfer this commutativity to the ramified case and to general $\Lambda$.

\subsection{The fusion monoidal structure via equivariant cohomology}
\label{ss:FG-monoidal-2}

Let us denote by 
\begin{equation}\label{eqn:monfus-defn}
	\phi_{\mathrm{fus}} \colon \sF_G(\scF) \otimes_\Lambda \sF_G(\scG) \simto \sF_G(\scF \star^0 \scG)
\end{equation}
the natural isomorphism constructed in~\cite[Proposition~6.4]{mv} or~\cite[Proposition~1.10.11]{br}.  The subscript ``fus'' refers to the ``fusion product,'' which is used in the construction of this isomorphism.  We will not review the details of the construction in general.

However, in the special case where $\Lambda = \bK$, $\phi_{\mathrm{fus}}$ has an alternative description in terms of equivariant cohomology, as explained in~\cite[\S3.3.4]{ar-book}.  In this subsection, we recall this alternative description.\footnote{In fact, the equivariant cohomology description is available for any $\Lambda$ in which the torsion primes of $G$ are invertible.  As the case $\Lambda = \bK$ is sufficient for our proof of Theorem~\ref{thm:monoidal-agree}, we will not consider more general rings here.}  

For the remainder of this subsection, we assume that $\Lambda = \bK$.  Let $R_{G,\bK}$ be the $\Loop^+G$-equivariant cohomology of a point, i.e., the graded $\bK$-algebra given by
\[
R_{G,\bK} := \bigoplus_{n \in \Z} \coH^n_{\Loop^+G}(\Spec(F^s), \bK) = \bigoplus_{n \in \Z} \Hom_{\Dbc([\Spec(F^s)/\Loop^+G]_{\et}, \bK)}(\underline{\bK}, \underline{\bK}[n]).
\]
The ring $\bK$, considered as a graded object concentrated in degree $0$, has a natural structure of graded $R_{G,\bK}$-module.

More generally, for any Levi subgroup $M \subset G$, we can define the ring $R_{M,\bK}$ in the same way, and there is an obvious map of equivariant cohomology rings
\begin{equation}\label{eqn:eqcohom-restrict}
	R_{G,\bK} \to R_{M,\bK}.
\end{equation}
It is well known that in the case of a torus, we have that $R_{T,\bK}$ is the graded symmetric algebra on 
\[
\coH^2_{\Loop^+T}(\Spec(F^s), \bK) \cong \bK \otimes_{\Z} \bbX^*(T).
\]
The (absolute) Weyl group $W_\abs$ of $(G_{F^s}, T_{F^s})$ acts on this ring, and the map~\eqref{eqn:eqcohom-restrict} (in the special case $M=T$) induces an isomorphism
\[
R_{G,\bK} \simto R_{T,\bK}^{W_\abs}.
\]

For $\scF \in \Dbc(\Hk_G,\bK)$ and $n \in \Z$ we set 
\[
\sH^n_{\Loop^+G}(\Gr_{G}, \scF) := \Hom_{\Dbc(\Hk_G, \bK)} (\underline{\bK}_X, \scF[n]),
\]
where $\underline{\bK}_X$ is (the object whose image in $\Dbc(\Gr_G,\bK)$ is) the constant sheaf on some $\Loop^+G$-stable closed subscheme $X$ that contains the support of $\scF$. As above, the same definition applies when $G$ is replaced by a Levi subgroup $M$, e.g.~by $T$.

Let $\gMod_{R_{G,\bK}}$ be the category of graded $R_{G,\bK}$-modules.  Define the functor
\[
\sFeq_G \colon \Dbc(\Hk_G, \bK) \to \gMod_{R_{G,\bK}}
\]
by
\[
\sFeq_G(\scF) := \bigoplus_{n \in \Z} \sH^n_{\Loop^+G}(\Gr_{G}, \scF).
\]

Let $M \subset G$ be a Levi subgroup, and recall the natural morphism
\[
h_{M,G} \colon [\Loop^+ M \backslash \Gr_G]_{\et} \to \Hk_G.
\]
Then we can also consider $\Loop^+ M$-equivariant cohomology of complexes on $\Hk_G$: for $\scF$ in $\Dbc(\Hk_G, \bK)$ we set
\[
\sF^{\mathrm{eq},M}_G(\scF) := \bigoplus_{n \in \Z} \sH^n_{\Loop^+M}(\Gr_{G}, h_{M,G}^* \scF).
\]

The next statements are well known; see e.g.~\cite[\S5.2]{zhu-notes} (see also~\cite[\S 3.3.4]{ar-book} for analogues in an ``analytic'' setting).  

\begin{lem}
	\label{lem:Feq-properties}
	For any $\scF \in \Perv(\Hk_G,\bK)$, $\sFeq_G(\scF)$ is a finitely generated projective $R_{G,\bK}$-module. Moreover, there is a natural isomorphism
	\begin{equation}
		\label{eqn:cohom-equiv-cohom}
		\sF_G(\scF) \cong \bK \otimes_{R_{G,\bK}} \sFeq_G(\scF)
	\end{equation}
	and, for any Levi subgroup $M \subset G$, a natural isomorphism
	\begin{equation}
		\label{eqn:equiv-cohom-restriction}
		R_{M,\bK} \otimes_{R_{G,\bK}} \sFeq_G(\scF) \cong \sF^{\mathrm{eq},M}_G(h_{M,G}^* \scF).
	\end{equation}
\end{lem}

Let $\gproj_{R_{G,\bK}}$ be the category of finitely generated projective $R_{G,\bK}$-modules. In view of Lemma~\ref{lem:Feq-properties},
one can regard $\sFeq_G$ as a functor
\[
\sFeq_G \colon \Perv(\Hk_G, \bK) \to \gproj_{R_{G,\bK}}.
\]

\begin{prop}
	\label{prop:sfeq-monoidal}
	For $\scF, \scG \in \Perv(\Hk_G,\bK)$, there is a natural isomorphism
	\[
	\sFeq_G(\scF \star \scG) \cong \sFeq_G(\scF) \otimes_{R_{G,\bK}} \sFeq_G(\scG),
	\]
	making $\sFeq_G \colon \Dbc(\Hk_G, \bK) \to \gproj_{R_{G,\bK}}$ into a monoidal functor.
\end{prop}

The preceding two results give rise to a monoidal structure on
\[
\sF_G \colon (\Perv(\Hk_G,\bK), \star) \to (\modf_\bK, \otimes_\bK)
\]
as follows:  Lemma~\ref{lem:Feq-properties} expresses $\sF_G$ as the composition of the functor
\[
\bK \otimes_{R_{G,\bK}} ({-}) \colon \gproj_{R_{G,\bK}} \to \modf_\bK
\]
with $\sFeq_G \colon \Perv(\Hk_G, \bK) \to \gproj_{R_{G,\bK}}$.  The former is obviously monoidal, and the latter is monoidal by Proposition~\ref{prop:sfeq-monoidal}, so $\sF_G$ is monoidal as well.

\subsection{Comparison of the monoidal structures}
\label{ss:comparison-monoidal-structures}

In this subsection we compare the monoidal structures on $\sF_G$ constructed in~\eqref{eqn:monoidality-defn} and~\eqref{eqn:monfus-defn}.

\begin{prop}
	\label{prop:CT-fibereq}
	For $\scF \in \Perv(\Hk_G,\bK)$, there is a natural isomorphism
	\begin{equation}\label{eqn:CT-fibereq}
		R_{T,\bK} \otimes_{R_{G,\bK}} \sFeq_G(\scF) \cong \sFeq_T(\CT_{B,G}(\scF)).
	\end{equation}
	Moreover, this isomorphism has the property that the following diagram commutes, where the upper vertical arrows are induced by the augmentation map $R_{T,\bK} \to \bK$:
	\[
	\begin{tikzcd}
		R_{T,\bK} \otimes_{R_{G,\bK}} \sFeq_G(\scF) \ar[d] \ar[r, "\sim"', "\eqref{eqn:CT-fibereq}"] &
		\sFeq_T(\CT_{B,G}(\scF)) \ar[d] \\
		\bK \otimes_{R_{G,\bK}} \sFeq_G(\scF) \ar[d, "\wr", "\eqref{eqn:cohom-equiv-cohom}"'] &
		\bK \otimes_{R_{T,\bK}} \sFeq_T(\CT_{B,G}(\scF)) \ar[d, "\wr"', "\eqref{eqn:cohom-equiv-cohom}"] \\
		\sF_G(\scF) \ar[r, "\sim"', "\text{\normalfont Prop.~\ref{prop:CT-fiber}}"] &
		\sF_T(\CT_{B,G}(\scF)).
	\end{tikzcd}
	\]
\end{prop}

\begin{proof}
	Recall the functor $\sF_G^{\mathrm{eq},T}$ considered above.
	In view of~\eqref{eqn:equiv-cohom-restriction},
	to prove the proposition it suffices to establish a version of the commutative diagram above in which the upper left-hand corner is replaced by $\sF_G^{\mathrm{eq},T}(\scF)$.  The desired isomorphism 
	\[
	\sF_G^{\mathrm{eq},T}(\scF) \cong \sFeq_T(\CT_{B,G}(\scF))
	\]
	can be obtained by copying the proof of Proposition~\ref{prop:CT-fiber} and replacing all occurrences of ordinary cohomology 
	by $\Loop^+T$-equivariant cohomology.
	(See~\cite[Lemma~2.2]{yun-zhu} for similar considerations.) The commutativity of the diagram is then immediate from the construction.
\end{proof}

\begin{prop}
	\label{prop:CT-fibereq-conv}
	For $\scF, \scG \in \Perv(\Hk_G,\bK)$, there is a natural isomorphism
	\begin{equation}\label{eqn:CT-fibereq-conv}
		R_{T,\bK} \otimes_{R_{G,\bK}} \sFeq_G(\scF \star \scG) \cong\sFeq_T(m_{T!}\widetilde{\CT}_{B,G}(p^*(\scF \lboxtimes_\bK \scG))).
	\end{equation}
	Moreover, this isomorphism has the property that the following diagram commutes:
	\[
	\begin{tikzcd}
		R_{T,\bK} \otimes_{R_{G,\bK}} \sFeq_G(\scF \star \scG) \ar[d] \ar[r, "\sim"', "\eqref{eqn:CT-fibereq-conv}"] &
		\sFeq_T(m_{T!}\widetilde{\CT}_{B,G}(p^*(\scF \lboxtimes_\bK \scG))) \ar[d] \\
		\bK \otimes_{R_{G,\bK}} \sFeq_G(\scF \star \scG) \ar[d, "\wr", "\eqref{eqn:cohom-equiv-cohom}"'] &
		\bK \otimes_{R_{T,\bK}} \sFeq_T(m_{T!}\widetilde{\CT}_{B,G}(p^*(\scF \lboxtimes_\bK \scG))) \ar[d, "\wr"', "\eqref{eqn:cohom-equiv-cohom}"] \\
		\sF_G(\scF \star \scG) \ar[r, "\sim"', "\text{\normalfont Prop.~\ref{prop:CT-fiber-conv}}"] &
		\sF_T(m_{T!}\widetilde{\CT}_{B,G}(p^*(\scF \lboxtimes_\bK \scG))).
	\end{tikzcd}
	\]
\end{prop}

\begin{proof}
	This is very similar to the proof of Proposition~\ref{prop:CT-fibereq}, replacing the details from the proof of Proposition~\ref{prop:CT-fiber} by those from Proposition~\ref{prop:CT-fiber-conv}.
\end{proof}

The following lemma is clear from constructions.

\begin{lem}
	\label{lem:monoidal-compare}
	For $\scF, \scG \in \Perv(\Hk_G,\bK)$, the following diagram commutes:
	\[
	\begin{tikzcd}
		\sFeq_T(m_{T!}\widetilde{\CT}_{B,G}(p_G^*(\scF \lboxtimes_{\bK} \scG))) \ar[dd, "\text{\normalfont Prop.~\ref{prop:CT-fibereq-conv}}"', "\wr"] \ar[r, "\text{\normalfont Cor.~\ref{cor:CT-monoidal-pull}}", "\sim"'] & \sFeq_T(\CT_{B,G}(\scF) \star^0 \CT_{B,G}(\scG)) \ar[d, "\text{\normalfont Prop.~\ref{prop:sfeq-monoidal}}", "\wr"'] \\
		& \sFeq_T(\CT_{B,G}(\scF)) \otimes_{R_{T,\bK}} \sFeq_T(\CT_{B,G}(\scG)) \ar[d, "\text{\normalfont Prop.~\ref{prop:CT-fibereq}}", "\wr"'] \\
		R_{T,\bK} \otimes_{R_{G,\bK}} \sFeq_G(\scF \star \scG) \ar[r, "\text{\normalfont Prop.~\ref{prop:sfeq-monoidal}}", "\sim"'] &
		R_{T,\bK} \otimes_{R_{G,\bK}} \sFeq_G(\scF) \otimes_{R_{G,\bK}} \sFeq_G(\scG).
	\end{tikzcd}
	\]
\end{lem}

For the next statement, we come back to the setting of a general ring of coefficients $\Lambda$ as in~\S\ref{ss:Perv-Gr-Hk}.

\begin{thm}\label{thm:monoidal-agree}
	The monoidal structures on $\sF_G \colon \Perv(\Hk_G,\Lambda) \to \modf_\Lambda$ considered in~\eqref{eqn:monoidality-defn} and~\eqref{eqn:monfus-defn} agree.
\end{thm}

\begin{proof}
	We must show that for $\scF, \scG \in \Perv(\Hk_G,\Lambda)$, the two isomorphisms
	\[
	\phi, \phi_{\mathrm{fus}} \colon \sF_G(\scF) \otimes_\Lambda \sF_G(\scG) \simto \sF_G(\scF \star^0 \scG)
	\]
	are equal.  Before proving this in general, we consider a number of special cases.
	
	\textit{Case 1. $\Lambda = \bK$ and $G = T$.} In this case, 
	the monoidal structure in~\eqref{eqn:monfus-defn} is provided by Lemma~\ref{lem:F-monoidal-torus}, and the agreement with the structure provided by Proposition~\ref{prop:sfeq-monoidal} is clear.
	
	\textit{Case 2. $\Lambda = \bK$, and $G$ is arbitrary.}
	Apply the functor $\bK \otimes_{R_{T,\bK}} ({-})$ to the commutative diagram in Lemma~\ref{lem:monoidal-compare}.  Using Lemma~\ref{lem:Feq-properties}, the commutative diagrams in Propositions~\ref{prop:CT-fibereq} and~\ref{prop:CT-fibereq-conv}, and Case~1 above, the resulting diagram can be written as
	\[
	\begin{tikzcd}
		\sF_T(m_{T!}\widetilde{\CT}_{B,G}(p_G^*(\scF \lboxtimes_\bK \scG))) \ar[dd, "\text{Prop.~\ref{prop:CT-fiber-conv}}"', "\wr"] \ar[r, "\text{Cor.~\ref{cor:CT-monoidal-pull}}", "\sim"'] & \sF_T(\CT_{B,G}(\scF) \star^0 \CT_{B,G}(\scG)) \ar[d, "\text{Lem.~\ref{lem:F-monoidal-torus}}", "\wr"'] \\
		& \sF_T(\CT_{B,G}(\scF)) \otimes_\bK \sF_T(\CT_{B,G}(\scG)) \ar[d, "\text{Prop.~\ref{prop:CT-fiber}}", "\wr"'] \\
		\sF_G(\scF \star \scG) \ar[r, "\text{\eqref{eqn:monfus-defn}}", "\sim"'] &
		\sF_G(\scF) \otimes_\bK \sF_G(\scG).
	\end{tikzcd}
	\]
	The result follows by comparison with the commutative diagram in Remark~\ref{rmk:F-monoidal-summary}.  
	
	\textit{Case 3. $\Lambda = \bO$, and $\sF_G(\scF)$ and $\sF_G(\scG)$ are free over $\bO$.}  This case follows from Case~2 by compatibility of our constructions with appropriate change-of-scalars functors.
	
	\textit{Case 4. $\Lambda = \bO$, and $\scF$ and $\scG$ are both direct sums of objects of the form $\scP_Z(\bO)$.}  By Property~\eqref{it:properties-P-3} in~\S\ref{ss:structure-projective}, this is a special case of Case~3.
	
	\textit{Case 5. $\Lambda = \bk$, and $\scF$ and $\scG$ are both direct sums of objects of the form $\scP_Z(\bk)$.}  This case follows from Case~4 by Property~\eqref{it:properties-P-2} in~\S\ref{ss:structure-projective} and compatibility with change of scalars.
	
	\textit{Case 6. $\Lambda = \bO$ or $\bk$, and $\scF$ and $\scG$ are arbitrary.}  For $\Lambda = \bO$ or $\bk$, any object in $\Perv(\Hk_G,\Lambda)$ is a quotient of a direct sum of objects of the form $\scP_Z(\Lambda)$, so the result in this case follows from Cases~4 and~5.
\end{proof}

\subsection{``Absolute'' geometric Satake equivalence}
\label{ss:absolute-Satake}

As mentioned already in~\S\ref{ss:absolute-var}
the structures we have considered on the category $\Dbc(\Hk_{\cG},\Lambda)$ above have well known (and older) counterparts for the group $F^s \pot{z} \otimes_{F^s} G_{F^s}$.
In particular we have the stack $\Hk_G$ over $\Spec(F^s)$, 
the convolution product $\star$ on the category $\Dbc(\Hk_G, \Lambda)$ defining a monoidal structure, the induced monoidal structure on the subcategory $\Perv(\Hk_G, \Lambda)$ of perverse sheaves obtained by setting
\[
\scF \star^0 \scG := \pH^0(\scF \star \scG),
\]
and the total cohomology functor $\sF_G$.
Given a parabolic subgroup $Q \subset G_{F^s}$ containing the maximal torus $T_{F^s}$, with Levi factor containing $T_{F^s}$ denoted $L$, we also have
the constant term functor
\[
\CT_{Q,G} \colon \Dbc(\Hk_G,\Lambda) \to \Dbc(\Hk_L,\Lambda)
\]
which is t-exact and admits a canonical monoidal structure.

On the other hand, denote by $G^\vee_{\mathbb{Z}}$ the unique pinned reductive group scheme over $\mathbb{Z}$ whose root datum is the dual of the 
root datum of $G_{F^s}$. 
Then, we set 
\[
G^\vee_\Lambda := \Lambda \otimes_{\mathbb{Z}} G^\vee_{\mathbb{Z}}.
\]
We will denote by $\Rep(G^\vee_\Lambda)$ the category of (algebraic) representations of this group scheme on finitely generated $\Lambda$-modules.

\begin{thm}
	\label{thm:Satake-abs}
	Fix a compatible system of $\ell^n$-th roots of unity in $\mathbb F$ for all $n\geq 1$.
	Then, there exists a canonical equivalence of monoidal categories
	\[
	\bigl( \Perv(\Hk_G, \Lambda), \star^0 \bigr) \cong \bigl( \Rep(G^\vee_\Lambda), \otimes_\Lambda \bigr)
	\]
	under which the functor $\sF_G$ corresponds to the obvious forgetful functor $\Rep(G^\vee_\Lambda) \to \modf_\Lambda$. Moreover, under these equivalences the ``change of scalars'' functors analogous so those considered in~\S\ref{ss:Perv-Gr-Hk} correspond to the obvious functors on categories of representations.
\end{thm}

\begin{proof}[Proof sketch]
	This statement is essentially the usual version of the geometric Satake equivalence from~\cite{mv}.
	Let us indicate the main steps of the proof.
	
	First, for any connected reductive group $H$ over $F^s$ with Borel subgroup $B_H \subset H$ and maximal torus $T_H \subset B_H$, the absolute version of Theorem~\ref{thm:coalgebra} gives us a bialgebra $\rmB_H(\Lambda)$ and an equivalence of monoidal categories
	\[
	\bigl( \Perv(\Hk_{F^s\pot{z} \otimes_{F^s} H},\Lambda), \star^0 \bigr) \simto \bigl( \comod_{\rmB_H(\Lambda)}, \otimes_\Lambda \bigr).
	\]
	It is shown by Mirkovi\'c--Vilonen~\cite[\S11]{mv} (see also~\cite[\S1.13.2]{br}) that $\rmB_H(\Lambda)$ is in fact a commutative Hopf algebra, so that
	\[
	\widehat{H}_\Lambda := \Spec(\rmB_H(\Lambda))
	\]
	is a flat affine group scheme over $\Lambda$.  Thus, Theorem~\ref{thm:coalgebra} can be restated as 
	an equivalence of monoidal categories
	\[
	\bigl( \Perv(\Hk_{F^s \pot{z} \otimes_{F^s} H}, \Lambda), \star^0 \bigr) \cong \bigl( \Rep(\widehat{H}_\Lambda), \otimes_\Lambda \bigr).
	\]
	This construction is compatible with change of scalars in the sense that 
	there are canonical identifications
	\[
	\widehat{H}_{\bK} \cong \bK \otimes_{\bO} \widehat{H}_\bO, \quad
	\widehat{H}_{\bk} \cong \bk \otimes_{\bO} \widehat{H}_\bO.
	\]
	
	Next, in~\cite[\S12]{mv} (see also~\cite[\S1.14]{br}), Mirkovi\'c--Vilonen construct 
	(using the constant term functor $\CT_{B_H,H}$) a canonical subgroup $\widehat{T}_{H,\Lambda}$ of $\widehat{H}_\Lambda$ canonically isomorphic to the diagonalizable group $\mathrm{D}_\Lambda(\bbX_*(T_H))$ over $\Lambda$ associated with the lattice $\bbX_*(T_H)$ of cocharacters of $T_H$; this construction is also compatible with extension of scalars in the same sense as above.
	They then check that $\widehat{H}_{\bO}$ is a split reductive group scheme over $\bO$ with maximal torus $\widehat{T}_{H,\bO}$, and that the root datum of $(\widehat{H}_{\bO},\widehat{T}_{H,\bO})$ is dual to that of $(H, T_H)$. The theorem will follow from this construction (applied to the group $H=G_{F^s}$, its Borel subgroup $B_{F^s}$ and the maximal torus $T_{F^s}$) once we explain how to construct a canonical pinning on $\widehat{H}_{\bO}$. 
	
	This can be done as follows, following e.g.~\cite[\S VI.11]{fs}.\footnote{This construction was known for a long time, but~\cite{fs} provides a convenient and explicit construction.} First, 
	one has a canonical Borel subgroup $\widehat{B}_{H,\bO}$ (i.e.~a choice of a system of positive roots) obtained as the subgroup stabilizing the filtration of the fiber functor $\sF_H$ on the category $\Perv(\Hk_{F^s \pot{z} \otimes_{F^s} H}, \bO)$ given by
	\[
	(\sF_H)^{\geq m}(\scF) = \bigoplus_{i \geq m} \sH^i(\Gr_{F^s \pot{z} \otimes_{F^s} H}, \scF).
	\]
	The only remaining structure we have to construct is a basis of each weight space in the Lie algebra of $\widehat{H}_{\bO}$ corresponding to a simple root. Using an appropriate constant term functor one reduces the construction to the case where $H$ is of semisimple rank $1$. Then $H':=H/Z(H)$ is isomorphic to $\mathrm{PGL}_{2,F^s}$, hence its affine Grassmannian $\Gr_{F^s \pot{z} \otimes_{F^s} H'}$ has a unique $1$-dimensional Schubert variety isomorphic to $\mathbb P^1_{F^s}$, and $\widehat{H'}_{\bO}$ identifies with the special linear group of the total cohomology of this variety, see~\cite[Comments after Lemma~VI.11.2]{fs}.
	The zeroth cohomology of this orbit has a canonical basis, and using the system of $\ell^n$-th roots of unity in order to trivialize Tate twists, so does the second cohomology. 
	We get an identification $\widehat{H'}_{\bO} = \mathrm{SL}_{2,\bO}$, and in particular a canonical pinning on $\widehat{H'}_{\bO}$. 
	Finally, we have
	\[
	(\Gr_{F^s \pot{z} \otimes_{F^s} H})_{\mathrm{red}} \cong \pi_1(H) \times_{\pi_1(H')} (\Gr_{F^s \pot{z} \otimes_{F^s} H'})_{\mathrm{red}}
	\]
	(where $\pi_1(H)$ is the algebraic fundamental group of $H$, and similarly for $H'$), which 
	provides an isomorphism 
	\[
	\widehat{H}_{\bO} = \mathrm{D}_{\bO}(\pi_1(H)) \times^{\mu_{2,\bO}} \widehat{H'}_{\bO}
	\]
	compatible in the natural way with the canonical maximal torus on each side. (See the proof of Proposition~\ref{prop:commutativity} below for a detailed version of similar considerations.) We deduce the desired pinning of $\widehat{H}_{\bO}$, which finishes the proof.
\end{proof}

\subsection{Galois action}
\label{ss:Galois}

In~\S\ref{ss:absolute-Satake} we have considered the ``usual'' affine Grassmannian $\Gr_G$ of the $F^s$-group $G_{F^s}$, an ind-scheme over $F^s$. In fact one can also consider the affine Grassmannian $\Gr_{F\pot{z} \otimes_F G}$ associated with the (reductive, but possibly nonsplit) $F$-group $G$ (an ind-scheme over $F$), and we have a canonical identification
\[
F^s \otimes_F \Gr_{F\pot{z} \otimes_F G} \simto \Gr_G,
\]
see~\eqref{eqn:Gr-base-change}.
The Galois group $I=\mathrm{Gal}(F^s / F)$ acts on the left-hand side via its action on $\Spec(F^s)$, which provides an action on $\Gr_G$. (Here $I$ acts by automorphisms as $F$-ind-scheme, but not as $F^s$-ind-scheme.) 

From the action of the group $I$ on $\Gr_G$
we deduce an action on the categories $\Dbc(\Hk_G, \Lambda)$ and $\Perv(\Hk_G, \Lambda)$, which is easily seen to be compatible with the convolution products $\star$ and $\star^0$ respectively. The functor $\sF_G$ is also invariant under these actions. By Tannakian formalism, and in view of Theorem~\ref{thm:Satake-abs}, this means that there exists an action of $I$ on $G^\vee_\Lambda$ (by automorphisms of group scheme over $\Lambda$) which induces the corresponding action on $\Perv(\Hk_G, \Lambda)$.

\begin{lem}
	The action of $I$ on $G^\vee_\Lambda$ preserves the canonical pinning constructed in the course of the proof of Theorem~\ref{thm:Satake-abs}, and factors through an action of a finite quotient.
\end{lem}

\begin{proof}
	The fact that the canonical maximal torus of $G^\vee_\Lambda$ is preserved by the action of $I$ follows from the fact that this torus identifies with the centralizer of the cocharacter provided by the cohomological grading. (Alternatively, one may prove this independence by noting that the maximal torus does not depend on the choices involved in its construction, see~\cite[\S 1.5.5]{br}.) The induced action on the canonical maximal torus $\mathrm{D}_{\Lambda}(\bbX_*(T))$ is via the natural $I$-action on $\bbX_*(T)$. Since $T$ splits over a finite extension of $F$, this implies in particular that there exists a normal subgroup $I' \subset I$ of finite index which acts trivially on this torus.
	
	The canonical Borel subgroup of $G^\vee_\Lambda$ is also manifestly stable under the $I$-action. The fact that the canonical simple root vectors are permuted is clear from their construction, since $I$ permutes the Levi subgroups of $G_{F^s}$ of semisimple rank $1$ according to its permutation of absolute simple roots. Finally, since a pinning-preserving automorphism is determined by the induced action on the root datum, the subgroup $I'$ considered above acts trivially on $G^\vee_\Lambda$, showing that the action of $I$ factors through a finite quotient.
\end{proof}

It is clear that the action of $I$ is compatible with extension-of-scalars, in the sense that the actions on $G^\vee_\bK$ and $G^\vee_\bk$ are induced by the action on $G^\vee_\bO$ via the canonical identifications
\[
G^\vee_\bK = \bK \otimes_{\bO} G^\vee_\bO, \quad G^\vee_\bk = \bk \otimes_{\bO} G^\vee_\bO.
\]

\section{Nearby cycles}
\label{sec:nc}

\subsection{Be\texorpdfstring{{\u\i}}{i}linson--Drinfeld Grassmannians}
\label{ss:BD-Gr}

We now consider a ``Be{\u\i}linson--Drin\-feld Grassmannian'' relating the ind-schemes $\Gr_\cG$ and $\Gr_G$.
We follow the purely local definition from \cite[\S 0.3]{richarz-erratum}.
For any $O_F$-algebra $R$ we equip $R\pot{z}$ with the $O_F$-algebra structure given by the unique $\F$-algebra map $O_F=\F\pot{t}\to R\pot{z}$ such that $t\mapsto z+t$.

Set $S:=\Spec(O_F)$. 
The {\it Be{\u\i}linson--Drinfeld Grassmannian} $\Gr_{\cG,S}$ associated with $\cG$ is the functor from the category of $O_F$-algebras to sets sending $R$ to the set of isomorphism classes of pairs $(\cE, \beta)$ where:
\begin{itemize}
	\item
	$\cE$ is a $\cG \times_S \Spec(R\pot{z})$-torsor on $\Spec(R\pot{z})$;
	\item
	$\beta$ is a trivialization of $\cE$ on $\Spec(R\rpot{z})$.
\end{itemize}
It is proven in \cite[Theorem~2.19, using Lemma~3.1]{richarz} that this functor is represented by an ind-projective ind-scheme over $S$.
Furthermore, it admits a loop uniformization, i.e. it can be written as the étale quotient of the full loop group $\Loop_S\cG$ by the positive loop group $\Loop_S^+\cG$.
Here $\Loop_S\cG$ and $\Loop_S^+\cG$ are the group-valued functors on the category of $O_F$-algebras $R$ given by $\cG(R\rpot{z})$ and $\cG(R\pot{z})$, respectively. 

\begin{rmk}
	\label{rmk:spreading}
	As explained in~\cite[\S 0.3]{richarz-erratum}, the ind-scheme $\Gr_{\cG, S}$ agrees with the base change of the Be{\u\i}linson--Drinfeld Grassmannian constructed using a spreading of $\cG$ over some curve as follows. 
	By \cite[Lemma~3.1]{richarz} (or \cite[\S 5.1.1]{HainesRicharz_TestFunctions}) there exists a smooth affine connected $\F$-curve $X$ with a point $x_0\in X(\F)$, an identification $\hat{\mathcal O}_{X,x_0}=O_F$ on completed local rings, and a smooth affine $X$-group scheme $\ucG$ together with an identification of $O_F$-group schemes $\ucG\times_XS= \cG$. 
	One then has the Be{\u\i}linson--Drinfeld Grassmannian $\Gr_{\ucG,X}$ associated with $\ucG$ and $X$. 
	The base change $\Gr_{\ucG,X}\times_XS$ agrees with $\Gr_{\cG,S}$ as defined above.
\end{rmk}

\subsection{A nearby cycles setting}
\label{ss:nc-setting}

Let us describe the fibers of $\Gr_{\cG,S}$ in more detail, following~\cite[\S 2.2]{richarz}. 
Let $s=\Spec(\F)$, resp.~$\eta=\Spec(F)$, be the special, resp.~generic, point of $S$.
As explained in~\cite[Corollary~2.14]{richarz}, we have a canonical identification
\begin{equation}
	\label{eqn:fibers-Gr}
	s \times_{S} \Gr_{\cG, S} \cong \Gr_{\cG}.
\end{equation}

On the other hand, 
$F\pot{z} \otimes_{O_F} \cG$ is a reductive group scheme over $F\pot{z}$, and by~\cite[Lemma~0.2]{richarz-erratum} there exists an isomorphism
\[
F\pot{z} \otimes_{O_F} \cG \cong F\pot{z} \otimes_{F} G
\]
where in the right-hand side we consider the obvious $F$-algebra morphism $F \to F\pot{z}$;
we fix once and for all such an isomorphism that maps
\[
F\pot{z} \otimes_{O_F} \cA, \quad \text{resp.} \quad F\pot{z} \otimes_{O_F} \cT, \quad \text{resp.} \quad F\pot{z} \otimes_{O_F} \cB
\]
into
\[
F\pot{z} \otimes_{F} A, \quad \text{resp.} \quad F\pot{z} \otimes_{F} T, \quad \text{resp.} \quad F\pot{z} \otimes_{F} B.
\]
Similarly to~\eqref{eqn:fibers-Gr}, this provides an identification
\[
\eta \times_S \Gr_{\cG,S} \cong \Gr_{F\pot{z} \otimes_F G}.
\]

Let also $\overline{S}$ be the spectrum of the normalization $\overline{O_F}$ of $O_F$ in $F^s$, and set
\[
\Gr_{\cG, \overline{S}} := \Gr_{\cG, S} \times_S \overline{S}.
\]
Then $\overline{O_F}$ is a valuation ring with fraction field $F^s$ and residue field $\F$. Let $\overline{\eta}=\Spec(F^s)$ and $\overline{s}=\Spec(\F)$ be the generic and special points of $\overline{S}$, respectively. We have
\[
\overline{s} \times_{\overline{S}} \Gr_{\cG, \overline{S}} = s \times_S \Gr_{\cG, S} = \Gr_\cG
\]
and
\[
\overline{\eta} \times_{\overline{S}} \Gr_{\cG,\overline{S}} = \overline{\eta} \times_\eta (\eta \times_S \Gr_{\cG,S}) = F^s \otimes_F \Gr_{F\pot{z} \otimes_F G} \overset{\eqref{eqn:Gr-base-change}}{\cong} \Gr_{F^s\pot{z} \otimes_{F^s} G_{F^s}}.
\]
\subsection{Actions via cocharacters}
\label{ss:BD-Gr-cocharacters}

Let $\lambda \colon \mathbb{G}_{\mathrm{m},S} \to \cA$ be a cocharacter over $S$. 
This cocharacter defines an action of $\mathbb{G}_{\mathrm{m},S}$ on $\cG$ over $S$, hence we can consider the associated fixed points and attractor schemes $\cM_\lambda:=\cG^0$ and $\cP_\lambda:=\cG^+$ as in \S\ref{ss:group-theory}.
The same considerations as in~\S\ref{ss:BD-Gr} allow to define the $S$-ind-schemes $\Gr_{\cM_\lambda, S}$ and $\Gr_{\cP_\lambda, S}$, and the obvious analogues of the statements in~\S\ref{ss:nc-setting} hold in this case also.

The cocharacter $\lambda$ induces a $\mathbb{G}_{\mathrm{m},S}$-action on
$\Gr_{\cG,S}$,  
and we can consider the associated fixed points and attractor functors $(\Gr_{\cG,{S}})^0$ and $(\Gr_{\cG,{S}})^+$. 
It follows from~\cite[Lemma~5.3]{HainesRicharz_TestFunctions}
that these functors are ind-schemes, and by naturality we obtain again morphisms 
\begin{equation}
	\label{eqn:fixed-pts-attractor-map-BD0}
	\Gr_{\cM_\lambda, {S}} \to (\Gr_{\cG, {S}})^0 \quad \text{and} \quad
	\Gr_{\cP_\lambda, {S}}\to (\Gr_{\cG,{S}})^+.
\end{equation}

\begin{prop}
	\label{prop:att-fix-Gr-oS}
	The maps \eqref{eqn:fixed-pts-attractor-map-BD0} are isomorphisms.
\end{prop}

\begin{proof}
	The proof is similar to the fiberwise case (see Proposition~\ref{prop:attractors-fixed-pts}). 
	Namely, we already know that both maps are closed immersions
	by~\cite[Theorem 5.6]{HainesRicharz_TestFunctions}. 
	By compatibility of fixed points and attractors with base change, Proposition~\ref{prop:attractors-fixed-pts} and its absolute analogue imply that the maps under consideration are isomorphisms over $\eta$ (see also \cite[Theorem 5.6]{HainesRicharz_TestFunctions}) and over $s$. 
	In particular, they 
	induce isomorphisms on the underlying reduced sub-ind-schemes after any base change. 
	To show that they are in fact isomorphisms, we use~\cite[Lemma 8.6 and Remark~8.7]{HLR} and repeat the proof of Proposition~\ref{prop:attractors-fixed-pts} with slight modifications adapted to the more general situation.
	Namely, to conclude the proof it suffices to show that for any $O_F$-algebra $R$ which is a strictly henselian local artinian ring, the induced maps $\Gr_{\cM_\lambda, S}(R) \to (\Gr_{\cG, S})^0(R)$ and $\Gr_{\cP_\lambda, S}(R) \to (\Gr_{\cG,S})^+(R)$ are bijective.
	We treat the second case, leaving the first one to the reader.
	
	The proof proceeds by induction on $\mathrm{n}(R)$, where as in the proof of Proposition~\ref{prop:attractors-fixed-pts} $\mathrm{n}(R)$ is the minimal positive integer $n$ such that ${\rm rad}(R)^n=0$. 
	If $\mathrm{n}(R)=1$, then $R$ is an algebraically closed field. 
	The map $O_F \to R$ factors either through a map $\F \to R$ or $F \to R$ (depending on whether $t$ maps to zero or not). 
	As the maps in~\eqref{eqn:fixed-pts-attractor-map-BD0} are isomorphisms over $\eta$ 
	and over $s$, this finishes the case $\mathrm{n}(R)=1$.
	
	Assume now that $\mathrm{n}(R)>1$, and let $J \subset R$ be an ideal of square $0$ such that $\mathrm{n}(R/J) < \mathrm{n}(R)$. 
	By induction we can assume that the map $\Gr_{\cP_\lambda, S}(R/J) \to (\Gr_{\cG, S})^+(R/J)$ is bijective. 
	Using this and formal smoothness, we are reduced to showing that any $x \in (\Gr_{\cG, S})^+(R)$ whose image in $(\Gr_{\cG, S})^+(R/J)$ is the image of the ``base point'' section $e_S \in \Gr_{\cG, S}(O_F)$ is in the image of $\Gr_{\cP_\lambda, S}(R)$, i.e., $x$ corresponds to an element of $T_{e_S,J}((\Gr_{\cG, S})^+)$.
	Here for an $S$-ind-scheme $X=\colim_{i\in I}X_i$ together with an $S$-point $e\in X(S)$ we consider the $R$-module 
	\begin{equation}
		\label{eqn:att-fix-Gr-oS0}
		T_{e,J}(X):= \colim_{i\in I}\Hom_{\Mod_{R}}(\omega_{X_i,e}\otimes_{O_F}R,J),
	\end{equation}
	where $\omega_{X_i,e}=\Gamma(S,e^*\Omega_{X_i/S})$ is the global sections of the conormal sheaf associated with the immersion $e$, see also the proof of Proposition~\ref{prop:attractors-fixed-pts}, and $\Mod_{R}$ is the category of $R$-modules.
	The $R$-module $T_{e,J}(X)$ does not depend on the chosen presentation of $X$ as an ind-scheme over $S$, and its underlying set agrees with $p^{-1}(e_{R/J})$ where $p\colon X(R)\to X(R/J)$ is the obvious map and $e_{R/J}$ denotes the image of $e$ under $X(S)\to X(R/J)$. 
	To finish the proof, it suffices to show that the map 
	\begin{equation}
		\label{eqn:att-fix-Gr-oS1}
		T_{e_S,J}(\Gr_{\cP_\lambda, S}) \to T_{e_S,J}((\Gr_{\cG, S})^+)
	\end{equation}
	is an isomorphism. 
	The proof is completely analogous to the proof of Lemma~\ref{lem:tangent-spaces} but replacing every occurrence of the tangent space or the Lie algebra by the $R$-module \eqref{eqn:att-fix-Gr-oS0}, and the loop functors $\Loop$ and $\Loop^+$ by their relative versions $\Loop_S$ and $\Loop_S^+$ respectively. 
	We recall some key steps.
	Since $R$ is strictly henselian, one has $\Gr_{\cG,S}(R)=\Loop_S\cG(R)/\Loop_S^+\cG(R)$ which implies the equality
	\[
	T_{e_S,J}(\Gr_{\cG,S})=T_{e_S,J}(\Loop_S\cG)/T_{e_S,J}(\Loop_S^+\cG),
	\]
	and similarly for $\cG$ replaced by $\cP_\lambda$.
	Next, the $\mathbb G_{m,S}$-action induces a $\Z$-grading on $e^*\Omega_{X/S}$ for each $\mathbb G_{m,S}$-invariant closed subscheme $X$ in $\Gr_{\cG,S}$ and $\Loop_S\cG$ respectively, hence a $\Z$-grading on $T_{e_S,J}(\Gr_{\cG,S})$, $T_{e_S,J}(\Loop_S\cG)$ and $T_{e_S,J}(\Loop_S^+\cG)$ respectively. 
	We get isomorphisms
	\[
	T_{e_S,J}((\Gr_{\cG,S})^+)=T_{e_S,J}(\Gr_{\cG,S})^+=T_{e_S,J}(\Loop_S\cG)^+/T_{e_S,J}(\Loop_S^+\cG)^+.
	\] 
	So, the claim \eqref{eqn:att-fix-Gr-oS1} is equivalent to proving the equality
	\[
	T_{e_S,J}(\Loop_S\cG)^+/T_{e_S,J}(\Loop_S^+\cG)^+=T_{e_S,J}(\Loop_S\cP_\lambda)/T_{e_S,J}(\Loop_S^+\cP_\lambda).
	\]
	As in \eqref{eqn:big-cell-cG} we use the big cell to prove the equalities $T_{e_S,J}(\Loop_S\cG)^+=T_{e_S,J}(\Loop_S\cP_\lambda)$ and $T_{e_S,J}(\Loop_S^+\cG)^+=T_{e_S,J}(\Loop_S^+\cP_\lambda)$, which finishes the proof.
\end{proof}

\begin{cor}
	\label{cor:fixed-pts-attractor-map-BD}
	The maps 
	\begin{equation}
		\label{eqn:fixed-pts-attractor-map-BD}
		\Gr_{\cM_\lambda, \overline{S}} \to (\Gr_{\cG, \overline{S}})^0 \quad \text{and} \quad
		\Gr_{\cP_\lambda, \overline{S}}\to (\Gr_{\cG,\overline{S}})^+.
	\end{equation}
	are isomorphisms.
\end{cor}
\begin{proof}
	This follows from Proposition \ref{prop:att-fix-Gr-oS} and compatibility of fixed points and attractors with base change along $\overline S\to S$.
\end{proof}

\subsection{Definition of the functor}
\label{ss:definition-nc}

We continue with the geometric setting of~\S\ref{ss:nc-setting}; in particular we have an ind-scheme $\Gr_{\cG,S} \to S$ where $S=\Spec(O_F)$, whose fiber over the special point $s=\Spec(\F)$, resp.~over the geometric generic point $\overline{\eta}=\Spec(F^s)$, identifies with $\Gr_{\cG}$, resp.~with the ``traditional'' affine Grassmannian
\[
\Gr_{G} := \Gr_{F^s\pot{z} \otimes_{F^s} G_{F^s}}
\]
considered in Section~\ref{sec:monoidality-unramified}.
Consider the natural embeddings
\[
\Gr_{G} \xrightarrow{j} \Gr_{\cG, \overline{S}} \xleftarrow{i} \Gr_{\cG}.
\]
The main player in this section will be the t-exact ``nearby cycles'' functor
\[
\Psi_\cG := i^* j_* \colon \Dbc(\Gr_{G}, \Lambda) \to \Dbc(\Gr_{\cG}, \Lambda)
\]
constructed as in~\S\ref{sss:const-nc}.

\begin{rmk}
	\label{rmk:nc}
	The ``true'' nearby cycles functor associated with the data above is rather the composition of $\Psi_\cG$ with the pullback functor under the natural morphism $\Gr_{G} \to (\Gr_{\cG,S})_\eta$ where $(\Gr_{\cG,S})_\eta$ is the fiber of $\Gr_{\cG,S}$ over $\eta$.
	The results mentioned in~\S\ref{sss:const-nc}
	of course have antecedents in the literature; see e.g.~\cite[Appendix]{bbd}. However these statements are usually given for the ``true'' nearby cycles functor 
	rather than for the version we consider, which justifies our references to~\cite{hs}.
\end{rmk}

As for the affine Grassmannian $\Gr_\cG$ in~\S\ref{ss:Iwahori-Weyl-Schubert},
the ind-scheme $\Gr_{G}$ has a stratification given by the Cartan decomposition, with strata parametrized by $\bbX_*(T)^+$. Namely, each $\lambda \in \bbX_*(T)^+$ determines an $F^s$-point in $\Gr_G$, and the $\Loop^+G$-orbit of this point is a quasi-projective subscheme $\Gr_G^{\lambda}$ of $\Gr_G$. Moreover, we have
\[
|\Gr_{G}| = \bigsqcup_{\lambda \in \bbX_*(T)^+} |\Gr_G^{\lambda}|.
\]
The intersection cohomology complex associated with the constant local system on $\Gr_G^{\lambda}$ will be denoted $\scJ_{!*}^{\abs}(\lambda, \Lambda)$. (Here ``$\abs$'' stands for ``absolute.'')

\begin{lem}
	\label{lem:Psi-IC}
	Let $\lambda \in \bbX_*(T)^+$, and denote by $\overline{\lambda}$ its image in $\bbX_*(T)^+_I$ (see Lem\-ma~\ref{lem:dominant-coweights}\eqref{it:dominant-coweights-lifting}). Then we have
	\[
	(j^{\overline{\lambda}})^* \Psi_{\cG}(\scJ_{!*}^{\abs}(\lambda, \Lambda)) \cong \underline{\Lambda}_{\Gr_{\cG}^{\overline{\lambda}}}[\langle \lambda, 2\rho \rangle].
	\]
\end{lem}

\begin{proof}
	The proof is the same as that of~\cite[Lemma~2.6]{zhu} (see also~\cite[p.~3755]{richarz}). Namely, consider the ``global Schubert variety'' $M_\lambda$ of~\cite[Definition~3.5]{richarz}. By~\cite[Corollary~3.14]{richarz} we have an open subscheme $\overset{\circ}{M}_\lambda \subset M_\lambda$ which is smooth over $\overline{S}$ and contains $\Gr_G^{\lambda}$, resp.~$\Gr_\cG^{\overline{\lambda}}$, in its generic, resp.~special, fiber. The desired claim follows, by compatibility of nearby cycles with smooth pullback.
\end{proof}

Note that
in case $\Lambda$ is a field, Lemma~\ref{lem:Psi-IC} implies that $\scJ_{!*}(\overline{\lambda}, \Lambda)$ is a composition factor of the perverse sheaf $\Psi_\cG(\scJ_{!*}^{\abs}(\lambda, \Lambda))$.

\subsection{Compatibility with convolution}

The functor $\Psi_\cG$ of~\S\ref{ss:definition-nc} admits an ``equivariant version;'' more specifically there exists a canonical functor
\begin{equation}
	\label{eqn:Psi-equiv}
	\Psi_\cG \colon \Dbc(\Hk_G, \Lambda) \to \Dbc(\Hk_{\cG}, \Lambda)
\end{equation}
which is related to the functor of~\S\ref{ss:definition-nc} by the obvious commutative diagram involving pullback functors to sheaves on the respective affine Grassmannians. 
Moreover this functor is t-exact, hence restricts to an exact functor
\begin{equation}
	\label{eqn:Psi-equiv-Perv}
	\Perv(\Hk_G, \Lambda) \to \Perv(\Hk_{\cG}, \Lambda).
\end{equation}

\begin{prop}
	\label{prop:Psi-monoidal}
	The functor~\eqref{eqn:Psi-equiv} admits a canonical monoidal structure with respect to the convolution products~$\star$. As a consequence, the functor~\eqref{eqn:Psi-equiv-Perv} admits a canonical monoidal structure with respect to the convolution products $\star^0$.
\end{prop}

\begin{proof}
	The proof is similar to that given in~\cite[Theorem-Definition~3.1]{zhu} or~\cite[\S 3.4]{ar-book} (which, itself, essentially goes back to~\cite{gaitsgory}). The idea is to consider a deformation of $\HkConv_G$ to $\HkConv_{\cG}$; more specifically, given $X$ and $\ucG$ as in Remark~\ref{rmk:spreading}, one considers (the restriction to $S$ of) the stack $\HkConv_{\cG,X}$ over $X$ defined as follows. Given an $\bF$-algebra $R$ and $x \in X(R)$, we denote by $\widehat{\Gamma}_x$ the spectrum of the completion of the ring $R \otimes_\F \scO(X)$ with respect to the ideal defining the graph $\Gamma_x \subset X \otimes_\F R$ of $x$. Then $\HkConv_{\cG,X}(R)$ is defined as the category of tuples $(x,\cE_1,\cE_2,\cE_3,\alpha,\beta)$ where $x \in X(R)$, $\cE_1,\cE_2, \cE_3$ are $\ucG$-bundles on $\widehat{\Gamma}_x$, and $\alpha$, resp.~$\beta$, is an isomorphism between the restrictions of $\cE_1$ and $\cE_2$, resp.~$\cE_2$ and $\cE_3$, to $\widehat{\Gamma}_x \smallsetminus \Gamma_x$. As in~\S\ref{ss:definition-nc}, we have
	natural embeddings
	\[
	\HkConv_{G} \xrightarrow{\tilde\jmath} \HkConv_{\cG, \overline{S}} \xleftarrow{\tilde\imath} \HkConv_{\cG}
	\]
	and a nearby cycles functor
	\[
	\widetilde{\Psi}_\cG :=\tilde\imath^* \tilde \jmath_* \colon \Dbc(\HkConv_{G}, \Lambda) \to \Dbc(\HkConv_{\cG}, \Lambda).
	\]
	By compatibility of nearby cycles with external tensor product, smooth pullback, and proper push-forward, the following diagram commutes up to natural isomorphism:
	\[
	\begin{tikzcd}[column sep=4.2em]
		\Dbc(\Hk_G,\Lambda) \times \Dbc(\Hk_G, \Lambda) \ar[r, "\Psi_\cG \times \Psi_\cG" ] \ar[d, "p_G^*({-} \lboxtimes_\Lambda {-})"']  \ar[dd, bend right=78, "\star"']  &
		\Dbc(\Hk_\cG,\Lambda) \times \Dbc(\Hk_\cG, \Lambda) \ar[d, "p_\cG^*({-} \lboxtimes_\Lambda {-})"] \ar[dd, bend left=78, "\star"] \\
		\Dbc(\HkConv_G, \Lambda) \ar[r, "\widetilde{\Psi}_{\cG}" description] \ar[d, "m_{G!}"' ]  &
		\Dbc(\HkConv_\cG, \Lambda) \ar[d, "m_{\cG!}" ] \\
		\Dbc(\Hk_G,\Lambda) \ar[r, "\Psi_{\cG}"'] & \Dbc(\Hk_\cG,\Lambda),
	\end{tikzcd}
	\]
	and this yields the desired monoidal structure on $\Psi_\cG$.
\end{proof}

\subsection{Compatibility with constant term functors}
\label{ss:compatibility-Psi-CT}

Recall the 
constant term functors from~\S\ref{ss:CT-functors}, and their ``absolute'' analogues considered in~\S\ref{ss:absolute-Satake}.
Fix $\lambda \in \bbX_*(A)$. This cocharacter determines a parabolic subgroup $P_\lambda \subset G$ (containing $A$) and its Levi factor $M_\lambda$, see~\S\ref{ss:group-theory}. By extension of scalars, it also determines a cocharacter $\lambda_\abs \in \bbX_*(T)$, hence a parabolic subgroup $P_\lambda^\abs \subset G_{F^s}$ (containing $T_{F^s}$) and its Levi factor $M_\lambda^\abs$. 

By compatibility of attractors and fixed points with base change, we have
\[
P_\lambda^\abs = P_\lambda \otimes_F F^s \quad \text{and} \quad M_\lambda^\abs = M_\lambda \otimes_F F^s.
\]
More specifically, assume that $P_\lambda$ is standard with respect to $B$, i.e.~that $\lambda$ is dominant. Then $P_\lambda$ is the standard parabolic subgroup of $G$ determined by the subset
\[
\Phi_\lambda^{\mathrm{s}} := \{\alpha \in \Phi^{\mathrm{s}} \mid \langle \lambda, \alpha \rangle = 0\} \subset \Phi^{\mathrm{s}}
\]
(see~\cite[\S 5.12]{bot}), and $P_\lambda^\abs$ is the standard parabolic subgroup of $G_{F^s}$ determined by the subset
\[
\{\alpha \in \Phi_\abs^{\mathrm{s}} \mid \langle \lambda, \alpha \rangle = 0\} \subset \Phi_\abs^{\mathrm{s}},
\]
i.e.~by the inverse image of $\Phi_\lambda^{\mathrm{s}}$ under~\eqref{eqn:surjection-Phis}.

Let $\cP_\lambda$ and $\cM_\lambda$ be as in~\S\ref{ss:group-theory}, and consider the action of $\mathbb{G}_{\mathrm{m},\overline{S}}$ on the ind-scheme $\Gr_{\cG, \overline{S}}$ via the cocharacter of $\cA$ naturally attached to $\lambda$ (see~\S\ref{ss:BD-Gr-cocharacters}). 
By 
Corollary~\ref{cor:fixed-pts-attractor-map-BD}, 
the fiber over $\overline{s}$, resp.~$\overline{\eta}$, of $(\Gr_{\cG, \overline{S}})^+$ identifies with $\Gr_{\cP_\lambda}$, resp.~$\Gr_{F^s\pot{z} \otimes_{F^s} P_\lambda^\abs}$, and 
the fiber over $\overline{s}$, resp.~$\overline{\eta}$, of $(\Gr_{\cG,\overline{S}})^0$ identifies with $\Gr_{\cM_\lambda}$, resp.~$\Gr_{F^s\pot{z} \otimes_{F^s} M_\lambda^\abs}$. 

\begin{prop}
	\label{prop:Psi-CT}
	In the setting above, there exists a canonical isomorphism
	\begin{equation}\label{eqn:Psi-CT}
		\Psi_{\cM_\lambda} \circ \CT_{P_\lambda^\abs, G} \cong \CT_{\cP_\lambda,\cG} \circ \Psi_\cG
	\end{equation}
	of functors from $\Dbc(\Hk_G,\Lambda)$ to $\Dbc(\Hk_{\cM_\lambda},\Lambda)$.
	In case $P_\lambda=B$, the restriction of this isomorphism to $\Perv(\Hk_G,\Lambda)$ is compatible with the monoidal structures on $\CT_{\cB,\cG}$ and $\CT_{B,G}$ from Proposition~\ref{prop:CT-monoidal}, and the monoidal structures on $\Psi_\cG$ and $\Psi_\cT$ from Proposition~\ref{prop:Psi-monoidal}.
\end{prop}

\begin{proof}
	The isomorphism~\eqref{eqn:Psi-CT} follows from the compatibility of hyperbolic localization with nearby cycles, see~\cite[Theorem~3.3]{richarz-Gm}.
	
	When $P_\lambda = B$, we have the following commutative diagram, in which the front and rear faces come from the proof of Proposition~\ref{prop:CT-monoidal}, and the left and right faces are obtained by applying $\pH^0$ to the diagram in the proof Proposition~\ref{prop:Psi-monoidal}.  (We omit most subscripts to avoid cumbersome notation, and write $\mathsf{P}$ for $\Perv$ to save space.)
	\[
	\hbox{\scriptsize$\begin{tikzcd}[column sep=-5pt, row sep=small]
			\mathsf{P}(\Hk_G,\Lambda) \times \mathsf{P}(\Hk_G, \Lambda) \ar[rr, "\CT \times \CT" ] \ar[dd, "({-}) \star^0 ({-})"'] \ar[dr, "\Psi \times \Psi"]  &&
			\mathsf{P}(\Hk_T,\Lambda) \times \mathsf{P}(\Hk_T, \Lambda) \ar[dd, "({-}) \star^0 ({-})" near start] \ar[dr, "\Psi \times \Psi"]  \\ 
			& \mathsf{P}(\Hk_\cG,\Lambda) \times \mathsf{P}(\Hk_\cG, \Lambda) \ar[rr, "\CT \times \CT"  near start, crossing over ]  &&
			\mathsf{P}(\Hk_\cT,\Lambda) \times \mathsf{P}(\Hk_\cT, \Lambda) \ar[dd, "({-}) \star^0 ({-})"]  \\ 
			\mathsf{P}(\Hk_G, \Lambda) \ar[dr, "\Psi"] \ar[rr, "\CT" near start] &&
			\mathsf{P}(\Hk_T, \Lambda) \ar[dr, "\Psi"]  \\ 
			& \mathsf{P}(\Hk_\cG, \Lambda) \ar[uu, leftarrow, crossing over, "({-}) \star^0 ({-})"' near end] \ar[rr, "\CT" near start, crossing over] &&
			\mathsf{P}(\Hk_\cT, \Lambda)  \\
			\end{tikzcd}$}
	\]
		The commutativity of this diagram implies that for $\scF, \scG \in \Perv(\Hk_G, \Lambda)$ the following diagram commutes, where we write $B^\abs$ for $B_{F^s}$:
		\[
		\begin{tikzcd}
			\Psi_{\cT}(\CT_{B^\abs, G}(\scF \star^0 \scG)) \ar[r, "\text{Prop.~\ref{prop:CT-monoidal}}"] \ar[d,  "\eqref{eqn:Psi-CT}"'] &
			\Psi_\cT(\CT_{B^\abs, G}(\scF) \star^0 \CT_{B^\abs,G}(\scG)) \ar[d, "\text{Prop.~\ref{prop:Psi-monoidal}}"] \\
			\CT_{\cB,\cG}(\Psi_\cG(\scF \star^0 \scG)) \ar[d, "\text{Prop.~\ref{prop:Psi-monoidal}}"']  
			& \Psi_\cT(\CT_{B^\abs, G}(\scF)) \star^0 \Psi_\cT(\CT_{B^\abs, G}(\scG)) \ar[d, "\eqref{eqn:Psi-CT}"] \\
			\CT_{\cB,\cG}(\Psi_\cG(\scF) \star^0 \Psi_\cG(\scG)) \ar[r, "\text{Prop.~\ref{prop:CT-monoidal}}"']
			& \CT_{\cB,\cG}(\Psi_\cG(\scF)) \star^0 \CT_{\cB,\cG}(\Psi_\cG(\scG)).
		\end{tikzcd}
		\]
		Thus, in this case,~\eqref{eqn:Psi-CT} is an isomorphism of monoidal functors.
	\end{proof}
	
	\subsection{Compatibility with fiber functors}
	
	Recall the functor $\sF_{\cG}$ of~\S\ref{ss:total-cohomology}, and its analogue $\sF_G$ considered in Section~\ref{sec:monoidality-unramified}.
	
	\begin{lem}
		\label{lem:F-Psi}
		There exists a canonical isomorphism
		\[
		\sF_{\cG} \circ \Psi_{\cG} \cong \sF_G
		\]
		of monoidal functors from $\Perv(\Hk_G, \Lambda)$ to $\modf_\Lambda$. 
	\end{lem}
	
	\begin{proof}
		The isomorphism follows from compatibility of nearby cycles with proper pushforward, applied to the proper morphism $\Gr_{\cG,\overline{S}} \to \overline{S}$. By construction of the monoidal structures on $\sF_\cG$ and $\sF_G$ (see also the discussion in~\S\ref{ss:comparison-monoidal-structures}), and in view of Proposition~\ref{prop:Psi-CT}, to prove the compatibility with monoidal structures in general it suffices to prove it in case $\cG=\cT$, where it is obvious.
	\end{proof}
	
	\subsection{Application to coalgebras}
	\label{ss:application-coalg}
	
	Recall from Theorem~\ref{thm:coalgebra}, resp.~Theorem~\ref{thm:Satake-abs}, that we have a canonical equivalence of categories
	\[
	\Perv(\Hk_{\cG},\Lambda) \simto \comod_{\rmB_\cG(\Lambda)}, \quad \text{resp.} \quad \Perv(\Hk_G,\Lambda) \simto \Rep(G^\vee_\Lambda)=\comod_{\scO(G^\vee_\Lambda)}
	\]
	under which the functor $\sF_{\cG}$, resp.~$\sF_G$, corresponds to the natural forgetful functor to $\modf_\Lambda$. 
	
	\begin{cor}
		\label{cor:morph-coalg-Psi}
		There exists a canonical morphism of $\Lambda$-bialgebras
		\[
		f_{\cG,\Lambda} \colon \scO(G^\vee_\Lambda) \to \rmB_\cG(\Lambda)
		\]
		such that the diagram
		\[
		\begin{tikzcd}[column sep=large]
		\Perv(\Hk_G,\Lambda) \ar[r, "\text{Thm.~\ref{thm:Satake-abs}}", "\sim"'] \ar[d, "\Psi_\cG"'] & \Rep(G^\vee_\Lambda) \ar[d] \\
		\Perv(\Hk_\cG,\Lambda) \ar[r, "\text{Thm.~\ref{thm:coalgebra}}", "\sim"'] & \comod_{\rmB_\cG(\Lambda)}
		\end{tikzcd}
		\]
		commutes, where the right vertical arrow is the functor induced by $f_{\cG,\Lambda}$. Moreover, this morphism factors through a morphism of $\Lambda$-bialgebras
		\[
		\tilde{f}_{\cG,\Lambda} \colon \scO((G^\vee_\Lambda)^I) \to \rmB_\cG(\Lambda).
		\]
	\end{cor}
	
	\begin{proof}
		The existence of $f_{\cG,\Lambda}$ follows from the same considerations as for Proposition~\ref{prop:morphism-coalg-Levi}. (In this case the compatibility with products uses the monoidal structure on the functor $\Psi_{\cG}$.) To prove that this morphism factors through the invariants $\scO((G^\vee_\Lambda)^I)$, as e.g.~in~\cite[Lemma~4.5]{zhu} one has to check that for any $\gamma \in I$ and any $\scF \in \Perv(\Hk_G,\Lambda)$ we have a canonical isomorphism
		\[
		\Psi_\cG(\gamma \cdot \scF) \cong \Psi_\cG(\scF).
		\]
		Now we have $\Gr_{\cG,\overline{S}} = \Gr_{\cG,S} \times_S \overline{S}$, and the $I$-action on $F^s$ stabilizes $\overline{O_F}$ and induces the trivial action on $\F$. Hence the action of $I$ on $\overline{\eta} \times_{\overline{S}} \Gr_{\cG,\overline{S}} = \Gr_G$ extends to an action on $\Gr_{\cG,\overline{S}}$ which restricts to the trivial action on $\overline{s} \times_{\overline{S}} \Gr_{\cG,\overline{S}}$. The desired property follows.
	\end{proof}
		
	Let us now consider the setting of~\S\ref{ss:compatibility-Psi-CT}.
	We can consider the morphism of coalgebras $f_{\cG,\Lambda}$ from Corollary~\ref{cor:morph-coalg-Psi}, and also the analogous morphism $f_{\cM_\lambda,\Lambda}$ for the reductive group $M_\lambda$ and its parahoric group scheme $\cM_\lambda$. On the other hand we have the morphism
	\[
	\res_{\cP_\lambda,\cG} \colon \rmB_\cG(\Lambda) \to \rmB_{\cM_\lambda}(\Lambda)
	\]
	of Proposition~\ref{prop:morphism-coalg-Levi}. 
	The same considerations provide a morphism of Hopf algebras
	\[
	\res_{P_\lambda,G} \colon \scO(G^\vee_\Lambda) \to \scO((M_\lambda)^\vee_\Lambda)
	\]
	which identifies $(M_\lambda)^\vee_\Lambda$ with the Levi subgroup of $G^\vee_\Lambda$ which is Langlands dual to the Levi subgroup $F^s \otimes_F M_\lambda \subset G_{F^s}$; see e.g.~\cite[\S 1.15.2]{br}. 
	
	The following claim is a direct consequence of Proposition~\ref{prop:Psi-CT}.
	
	\begin{cor}
		\label{cor:compatibility-f-r}
		In the setting of Proposition~\ref{prop:Psi-CT}, we have
		\[
		f_{\cM_\lambda,\Lambda} \circ \res_{P_\lambda,G} = \res_{\cP_\lambda,\cG} \circ f_{\cG,\Lambda}.
		\]
	\end{cor}
	
	In this setting we have actions of $I$ both on $G^\vee_\Lambda$ and on $(M_\lambda)^\vee_\Lambda$ (see~\S\ref{ss:Galois}), and it is easily seen that the morphism $\res_{P_\lambda,G}$ is $I$-equivariant. It therefore induces a morphism
	\[
	\widetilde{\res}_{P_\lambda,G} \colon \scO((G^\vee_\Lambda)^I) \to \scO(((M_\lambda)^\vee_\Lambda)^I)
	\]
	which satisfies
	\[
	\widetilde{f}_{\cM_\lambda, \Lambda} \circ \widetilde{\res}_{P_\lambda,G} = \res_{\cP_\lambda,\cG} \circ \widetilde{f}_{\cG,\Lambda}.
	\]
		
	\section{The ramified geometric Satake equivalence}
	\label{sec:equivalence}
	
	\subsection{Commutativity}
	
	We will now prove that $\rmB_{\cG}(\Lambda)$ is 
	the coordinate ring of a group scheme over $\Lambda$.
	
	\begin{prop}
		\label{prop:commutativity}
		The $\Lambda$-bialgebra $\rmB_{\cG}(\Lambda)$ is a commutative Hopf algebra. In particular it makes sense to consider the spectrum
		\[
		\cG^\vee_\Lambda := \Spec(\rmB_\cG(\Lambda)),
		\]
		and this affine scheme has a canonical structure of flat group scheme over $\Lambda$.
	\end{prop}
	
	\begin{proof}
		Assume for a moment that $G$ is semisimple of adjoint type (i.e.~has trivial center). By Lemma~\ref{lem:semisimplicity}, 
		the category $\Perv(\Hk_{\cG}, \bK)$ is semisimple; its simple objects are the intersection cohomology complexes $\scJ_{!*}(\lambda,\bK)$ for $\lambda \in \bbX_*(T)^+_I$.
		Lemma~\ref{lem:dominant-coweights}\eqref{it:dominant-coweights-lifting} and Lemma~\ref{lem:Psi-IC} imply that any of these simple objects is a subquotient of an object in the image of $\Psi_{\cG}$; using~\cite[Lemma~2.2.13]{schauenburg} we deduce that the morphism $f_\cG$ from Corollary~\ref{cor:morph-coalg-Psi} is surjective.
		Since $\scO(G^\vee_\bK)$ is commutative, this implies that $\rmB_\cG(\bK)$ is commutative. By flatness of $\rmB_\cG(\bO)$ and the isomorphisms~\eqref{eqn:BG-ext-scalars}, we deduce that $\rmB_\cG(\Lambda)$ is commutative for any $\Lambda$. In particular, we can therefore consider the semigroup scheme
		\[
		\cG^\vee_\Lambda := \Spec(\rmB_\cG(\Lambda))
		\]
		over $\Lambda$. The fact that this semigroup scheme is a group scheme, i.e.~that $\rmB_\cG(\Lambda)$ admits an antipode, can be checked as in~\cite[Proposition~1.13.4]{br}. Since $\rmB_\cG(\Lambda)$ is flat over $\Lambda$, this group scheme is flat.
		
		Now we drop the assumption that $G$ is semisimple of adjoint type, and consider the quotient morphism $G \to G_{\ad}$ as in the proof of Lemma~\ref{lem:dominant-coweights}\eqref{it:dominant-coweights-orbits}.
		We have an induced equivariant map of buildings $\mathscr{B}(G,F)\to \mathscr{B}(G_{\ad},F)$ that induces a bijection between facets. 
		Let $\fa_\ad \subset \mathscr{B}(G_{\mathrm{ad}},F)$ be the facet corresponding to $\fa$; then we can consider the parahoric group scheme $\cG_{\mathrm{ad}}$ for $G_{\mathrm{ad}}$ corresponding to $\fa_\ad$.
		Since parahoric group schemes are smooth, \cite[Proposition 1.7.6]{BT84} yields a morphism $\cG \to \cG_{\mathrm{ad}}$ of group schemes over $O_F$, which induces a morphism of ind-schemes
		\begin{equation}
			\label{eqn:morph-Gr-Gad}
			\Gr_{\cG} \to \Gr_{\cG_{\mathrm{ad}}}.
		\end{equation}
		On the other hand, the Kottwitz morphism provides a map
		\[
		\Gr_{\cG} \to \underline{\pi_1(G)_I}
		\]
		where $\pi_1(G)_I$ is as in~\S\ref{ss:tori-weights} and $\underline{\pi_1(G)_I}$ is the associated ind-scheme over $\F$. We also have the similar ind-scheme $\underline{\pi_1(G_{\mathrm{ad}})_I}$ (which in this case is a scheme since $\pi_1(G_{\mathrm{ad}})_I$ is finite), and morphisms
		\[
		\underline{\pi_1(G)_I} \to \underline{\pi_1(G_{\mathrm{ad}})_I} \quad \text{and} \quad \Gr_{\cG_{\mathrm{ad}}} \to \underline{\pi_1(G_{\mathrm{ad}})_I}.
		\]
		Together, these morphisms induce a map
		\begin{equation*}
			\label{eqn:isom-Gr-Gad}
			\Gr_{\cG} \to \Gr_{\cG_{\mathrm{ad}}} \times_{\underline{\pi_1(G_{\mathrm{ad}})_I}} \underline{\pi_1(G)_I}.
		\end{equation*}
		
		The fibers of the morphism $\Gr_{\cG_{\mathrm{ad}}} \to \underline{\pi_1(G_{\mathrm{ad}})_I}$ are the connected components of $\Gr_{\cG_{\mathrm{ad}}}$; hence any object of $\Perv(\Hk_{\cG_{\mathrm{ad}}},\Lambda)$ admits a canonical grading by $\pi_1(G_{\mathrm{ad}})_I$. By the same considerations as in the proof of Proposition~\ref{prop:morphism-coalg-Levi}, we deduce a morphism of $\Lambda$-group schemes
		\[
		\Diag_\Lambda(\pi_1(G_{\mathrm{ad}})_I) \to (\cG_{\mathrm{ad}})_\Lambda^\vee.
		\]
		Corollary~\ref{cor:passing-to-adjoint} implies that the datum of an object in $\Perv(\Hk_{\cG},\Lambda)$ is equivalent to the datum of an object in $\Perv(\Hk_{\cG_{\mathrm{ad}}},\Lambda)$ together with a ``lift'' of the associated $\pi_1(G_{\mathrm{ad}})_I$-grading to a $\pi_1(G)_I$-grading, or in other words of a representation of the group scheme $(\cG_{\mathrm{ad}})_\Lambda^\vee$ together with an ``extension'' of the action of the diagonalizable group scheme $\Diag_\Lambda(\pi_1(G_{\mathrm{ad}})_I)$ to an action of the diagonalizable group scheme $\Diag_\Lambda(\pi_1(G)_I)$. It follows that $\rmB_\cG(\Lambda)$ is the structure algebra of the flat $\Lambda$-group scheme
		\[
		\Diag_\Lambda(\pi_1(G)_I) \times_{\Spec(\Lambda)}^{\Diag_\Lambda(\pi_1(G_{\mathrm{ad}})_I)} (\cG_{\mathrm{ad}})^\vee_\Lambda,
		\]
		which finishes the proof.
	\end{proof}
	
	Now we consider a parabolic subgroup $P \subset G$ containing $A$, and the associated morphism $\res_{\cP,\cG}$ from Proposition~\ref{prop:morphism-coalg-Levi}.
	
	\begin{prop}
		\label{prop:res-bialg}
		For any parabolic subgroup $P \subset G$ containing $A$, $\res_{\cP,\cG}$ is a morphism of Hopf algebras.
	\end{prop}
	
	\begin{proof}
		The proof is similar to that of Proposition~\ref{prop:commutativity}. Namely, in case $G$ is semisimple of adjoint type the claim follows from the similar claim for $\mathrm{res}_{P,G}$ in case $\Lambda=\bK$, which is well known, see~\S\ref{ss:application-coalg}. (Note that in this case any Levi subgroup of $G$ satisfies the conditions in Lemma~\ref{lem:dominant-coweights}\eqref{it:dominant-coweights-lifting}, so that the associated morphism of Corollary~\ref{cor:morph-coalg-Psi} is surjective.) The general case can then be reduced to this one.
	\end{proof}
		
	\subsection{Identification of the dual group}
	\label{ss:identification-dual-gp}
	
	Passing to spectra, the morphism $\widetilde{f}_\cG$ of Corollary~\ref{cor:morph-coalg-Psi} provides a canonical morphism of $\Lambda$-group schemes
	\begin{equation*}
		\varphi_{\cG,\Lambda} \colon \cG^\vee_\Lambda \to (G^\vee_\Lambda)^I.
	\end{equation*}
	By~\eqref{eqn:BG-ext-scalars}, 
	there exist canonical isomorphisms
	\begin{equation}
		\label{eqn:base-change-Tann-gp}
		\cG^\vee_{\bK} \simto \bK \otimes_\bO \cG^\vee_\bO, \quad
		\cG^\vee_{\bk} \simto \bk \otimes_\bO \cG^\vee_\bO.
	\end{equation}
	On the other hand, by compatibility of fixed points with base change (see~\cite[(2.1)]{alrr}),
	we also have canonical isomorphisms
	\begin{equation}
		\label{eqn:base-change-Tann-gp-2}
		(G^\vee_{\bK})^I \simto \bK \otimes_\bO (G^\vee_\bO)^I, \quad
		(G^\vee_{\bk})^I \simto \bk \otimes_\bO (G^\vee_\bO)^I.
	\end{equation}
	These identifications are compatible with the corresponding morphisms $\varphi_{\cG,\Lambda}$ in the obvious way.
		
	The following theorem is the main result of this paper. Its proof will occupy the rest of the section.
	
	\begin{thm}
		\label{thm_dual_group_fixpts}
		For any $\Lambda$, the morphism $\varphi_{\cG,\Lambda}$ is an isomorphism.
	\end{thm}
	
	We start with a reduction of this statement to semisimple groups of adjoint type. 
	Recall from the proof of Proposition~\ref{prop:commutativity} that we have a canonical integral model $\cG_{\mathrm{ad}}$ of the adjoint quotient $G_{\mathrm{ad}}$ of $G$ induced by the parahoric model $\cG$. Of course, the choice of $B$ determines a Borel subgroup in $G_{\mathrm{ad}}$, and as noted in the proof of Lemma~\ref{lem:dominant-coweights}\eqref{it:dominant-coweights-orbits} the choice of $A$ determines a maximal split torus in $G_{\mathrm{ad}}$. 
	Hence we have all the ingredients to run the above constructions for the group $\cG_{\mathrm{ad}}$.
	
	\begin{lem}
		\label{lem:reduction-adjoint}
		If Theorem~\ref{thm_dual_group_fixpts} holds for the group $\cG_{\mathrm{ad}}$, then it holds for $\cG$.
	\end{lem}
	
	\begin{proof}
		As noted in the course of the proof of Proposition~\ref{prop:commutativity}, we have
		\begin{equation}
			\label{eqn:dual-group-G-Gad}
			\cG^\vee_\Lambda = \Diag_\Lambda(\pi_1(G)_I) \times_{\Spec(\Lambda)}^{\Diag_\Lambda(\pi_1(G_{\mathrm{ad}})_I)} (\cG_{\mathrm{ad}})^\vee_\Lambda.
		\end{equation}
		On the other hand, if $\Cent(G^\vee_\Lambda)$ is the scheme-theoretic center of $G^\vee_\Lambda$, then there is a natural identification
		$\Cent(G^\vee_\Lambda) = \Diag_\Lambda(\pi_1(G))$, which provides an isomorphism
		\[
		(\Cent(G^\vee_\Lambda))^I = \Diag_\Lambda(\pi_1(G)_I),
		\]
		see~\cite[Lemma~2.2]{alrr}. Similarly, for $G_{\mathrm{ad}}$ we have
		\[
		(\Cent((G_{\mathrm{ad}})^\vee_\Lambda))^I = \Diag_\Lambda(\pi_1(G_{\mathrm{ad}})_I),
		\]
		and by~\cite[Proposition~6.8]{alrr} the natural morphism
		\begin{equation}
			\label{eqn:fixed-points-G-Gad}
			(\Cent(G^\vee_\Lambda))^I \times^{(\Cent((G_{\mathrm{ad}})^\vee_\Lambda))^I} ((G_{\mathrm{ad}})^\vee_\Lambda)^I \to (G^\vee_\Lambda)^I
		\end{equation}
		is an isomorphism, which finishes the proof. Comparing~\eqref{eqn:dual-group-G-Gad} and~\eqref{eqn:fixed-points-G-Gad} together with the description of fixed points in centers we deduce the claim.
	\end{proof}
	
	Since the adjoint group of a torus is the trivial group, for which the theorem obviously holds, we deduce in particular from Lemma~\ref{lem:reduction-adjoint} that Theorem~\ref{thm_dual_group_fixpts} holds when $G$ is a torus. Once this is known, for general $G$ and $\cG$, the morphism $\res_{\cB,\cG}$ of Proposition~\ref{prop:morphism-coalg-Levi} defines a morphism of group schemes $(T^\vee_\Lambda)^I \to \cG^\vee_\Lambda$ whose composition with $\varphi_{\cG,\Lambda}$ is the natural closed immersion $(T^\vee_\Lambda)^I \to (G^\vee_\Lambda)^I$. The former morphism is therefore also a closed immersion.
		
	Now, let us outline the strategy of the proof of Theorem~\ref{thm_dual_group_fixpts}.
	By Lemma~\ref{lem:reduction-adjoint} we can assume that $G$ is semisimple of adjoint type. 
	The cases $\Lambda=\bK,\bk$ follow from the case $\Lambda=\bO$ by base change using~\eqref{eqn:base-change-Tann-gp} and~\eqref{eqn:base-change-Tann-gp-2}.
	If $\Lambda=\bO$, then we aim to apply the following statement to the map $\tilde{f}_{\cG,\bO} \colon \scO((G^\vee_\bO)^I)\to \scO(\cG_\bO^\vee)$.
	
	\begin{lem}[{\cite[Lemma~VI.11.3]{fs}}]
		\label{lem:fs}
		Let $M,N$ be flat $\bO$-modules, and let 
		\[
		f \colon M \to N
		\]
		be a morphism of $\bO$-modules such that $\bk \otimes_\bO f$ is injective and $\bK \otimes_\bO f$ is an isomorphism. Then $f$ is an isomorphism.
	\end{lem}
	
	Both $\bO$-modules $\scO((G^\vee_\bO)^I)$ and $\scO(\cG_\bO^\vee)=B_\cG(\bO)$ are flat by \cite[Theorem~5.1(1)]{alrr} and Theorem~\ref{thm:coalgebra} respectively.
	By the proof of Proposition~\ref{prop:commutativity}, we know that $\varphi_{\cG,\bO}\otimes_\bO\bK=\varphi_{\cG,\bK}$ is a closed immersion.
	Thus, to conclude it suffices to check that $\tilde{f}_{\cG,\Lambda}$ is injective for both $\Lambda=\bK,\bk$, which by~\cite[\href{https://stacks.math.columbia.edu/tag/056A}{Tag 056A}]{stacks-project} amounts to showing that the scheme-theoretic image of $\varphi_{\cG,\bk}$ is $(G^\vee_\Lambda)^I$ for these choices of coefficients.
	This is checked in Subsection~\ref{sec:ss-rank-1} when $G$ has semisimple $F$-rank $1$. 
	The general case is proven in Subsection~\ref{sec:general-case} using constant term functors to construct enough ``semisimple-rank-$1$ Levi subgroups'' in $\cG_\Lambda^\vee$ to ensure the schematical dominance of $\varphi_{\cG,\Lambda}$ for both $\Lambda=\bK,\bk$.
	Here the required group theory is supplied by \cite{alrr}, where the case $\Lambda=\bk$ is particularly interesting in characteristic $2$.
	
	\subsection{Groups of semisimple \texorpdfstring{$F$}{F}-rank 1}\label{sec:ss-rank-1}
	
	In this subsection we prove Theorem~\ref{thm_dual_group_fixpts} in case $G$ has semisimple $F$-rank $1$, i.e.~$\# \Phi^{\mathrm{s}}=1$. 
	
	The next statement is probably well known, although we could not find a proof in the literature. (It is somehow implicit in~\cite[Lemma~VI.11.2]{fs}.)
	Here, for any field $k$, we denote by:
	\begin{itemize}
		\item 
		$\mathrm{T}_{2,k}$, resp.~$\mathrm{B}^+_{2,k}$, resp.~$\mathrm{B}^-_{2,k}$, the subgroup of $\mathrm{SL}_{2,k}$ consisting of diagonal matrices, resp.~upper triangular matrices, resp.~lower triangular matrices;
		\item
		$\underline{\mathrm{T}}_{2,k}$, resp.~$\underline{\mathrm{B}}^+_{2,k}$, resp.~$\underline{\mathrm{B}}^-_{2,k}$, the subgroup of $\mathrm{PGL}_{2,k}$ consisting of (images of) diagonal matrices, resp.~upper triangular matrices, resp.~lower triangular matrices
	\end{itemize}
	
	\begin{lem}
		\label{lem:subgroups-SL2}
		Let $k$ be a field.
		\begin{enumerate}
			\item 
			If $K$ is a smooth connected closed subgroup of $\mathrm{SL}_{2,k}$ containing $\mathrm{T}_{2,k}$, then $K$ is one of $\mathrm{T}_{2,k}$, $\mathrm{B}^\pm_{2,k}$, or $\mathrm{SL}_{2,k}$;
			\item
			If $K$ is a smooth connected closed subgroup of $\mathrm{PGL}_{2,k}$ containing $\underline{\mathrm{T}}_{2,k}$, then $K$ is one of $\underline{\mathrm{T}}_{2,k}$, $\underline{\mathrm{B}}^\pm_{2,k}$, or $\mathrm{PGL}_{2,k}$.
		\end{enumerate}
	\end{lem}
	
	\begin{proof}
		We explain the case of $\mathrm{SL}_{2,k}$; that of $\mathrm{PGL}_{2,k}$ can be treated similarly, or deduced using the quotient morphism $\mathrm{SL}_{2,k} \to \mathrm{PGL}_{2,k}$.
		We denote by $\mathrm{U}^\pm_{2,k}$ the unipotent radical of $\mathrm{B}^\pm_{2,k}$. The multiplication map
		\[
		\mathrm{U}^-_{2,k} \times \mathrm{B}^+_{2,k} \to \mathrm{SL}_{2,k}
		\]
		is an open immersion. Its image $\mathrm{C}_{2,k}$ (the ``big cell'') intersects $K$, hence by connectedness $\mathrm{C}_{2,k} \cap K$ is open and dense in $K$. On the other hand, by~\cite[Proposition~2.1.8(3)]{CGP10} the multiplication morphism
		\[
		(\mathrm{U}^-_{2,k} \cap K) \times (\mathrm{B}^+_{2,k} \cap K) \to \mathrm{C}_{2,k} \cap K
		\]
		is an isomorphism, and since $K$ contains $\mathrm{T}_{2,k}$ it follows that multiplication induces an isomorphism
		\[
		(\mathrm{U}^-_{2,k} \cap K) \times \mathrm{T}_{2,k} \times (\mathrm{U}^+_{2,k} \cap K) \simto \mathrm{C}_{2,k} \cap K.
		\]
		The other statement in~\cite[Proposition~2.1.8(3)]{CGP10} guarantees that $\mathrm{U}^\pm_{2,k} \cap K$ is smooth. It is a $\mathbb{G}_{\mathrm{m},k}$-stable subgroup of $\mathrm{U}^\pm_{2,k}$, hence is either trivial or equal to $\mathrm{U}^\pm_{2,k}$. 
		The resulting four possible cases lead to the four cases $\mathrm{T}_{2,k}$, $\mathrm{B}^\pm_{2,k}$, or $\mathrm{SL}_{2,k}$.
	\end{proof}
	
	Finally we come to the main result of the subsection.
	
	\begin{prop}
		\label{prop:dual.gp.sl2.case}
		Theorem~\ref{thm_dual_group_fixpts} holds in case $G$ is semisimple of $F$-rank $1$.
	\end{prop} 
	
	\begin{proof}
		By Lemma~\ref{lem:reduction-adjoint} we can (and will) assume that $G$ is
		furthermore semisimple of adjoint type. Then $G_{F^s}$ is also semisimple of adjoint type, hence a product of simple groups (of adjoint type). Since $\# \Phi^{\mathrm{s}}=1$, there are two possibilities:
		\begin{itemize}
			\item[(a)]
			either $G_{F^s}$ is a product of copies of $\mathrm{PGL}_{2,F^s}$ and $I$ acts by a transitive permutation of the factors;
			\item[(b)]
			or $G_{F^s}$ is a product of copies of $\mathrm{PGL}_{3,F^s}$, $I$ permutes transitively the factors, and the stabilizer of each factor acts non trivially on that factor (via the unique nontrivial diagram automorphism).
		\end{itemize}
		In case (a), $G^\vee_\Lambda$ is a product of copies of $\mathrm{SL}_{2,\Lambda}$ permuted transitively by $I$, so that we have
		\begin{equation}
			\label{eqn:fixed-pts-rk1-1}
			(G^\vee_\Lambda)^I \cong \mathrm{SL}_{2,\Lambda}.
		\end{equation}
		In case (b), $G^\vee_\Lambda$ is a product of copies of $\mathrm{SL}_{3,\Lambda}$, and we have
		\begin{equation}
			\label{eqn:fixed-pts-rk1-2}
			(G^\vee_\Lambda)^I \cong (\mathrm{SL}_{3,\Lambda})^{\Z/2\Z}
		\end{equation}
		where in the right-hand side the action is that considered in~\cite[\S 2.3]{alrr}. (See~\cite[\S 2.3]{alrr} for details.) Moreover, if $2$ is invertible in $\Lambda$ then the group in~\eqref{eqn:fixed-pts-rk1-2} identifies with $\mathrm{PGL}_{2,\Lambda}$, see~\cite[Example~5.9(1)]{alrr}.
		
		Now we treat the case $\Lambda=\bK$. 
		As seen in the course of the proof of Proposition~\ref{prop:commutativity}, in this case we know that $\varphi_{\cG,\bK}$ is a closed immersion. Using tannakian formalism one checks as in~\cite[Lemma~1.9.3]{br} that $\cG^\vee_\bK$ is connected (using the fact that $\bbX_*(T)_I$ is free over $\Z$ under our present assumptions), and as in~\cite[Lemma~1.9.4]{br} (using Lemma~\ref{lem:semisimplicity}) that it is reductive. This group is not a torus since it admits simple representations whose dimension is at least $2$. (This property can e.g.~be checked using the results of~\S\ref{ss:preliminaries-standard-costandard}.) Hence $\varphi_{\cG,\bK}$ is an isomorphism.
		
		Now we assume that $\Lambda=\bk$, and denote by 
		$H$ the scheme-theoretic image of $\varphi_{\cG,\bk}$, or in other words 
		the spectrum of the image of $\widetilde{f}_{\cG,\bk}$ (see the comments at the end of~\S\ref{ss:identification-dual-gp}). Then $H$ is a closed subgroup scheme of $(G^\vee_\bk)^I$.
		The morphism $\scO(H) \to \scO(\cG^\vee_{\bk})$ is injective, hence $H$ is a quotient of $\cG^\vee_{\bk}$; in particular, we have a fully faithful monoidal functor
		\[
		\Rep(H) \to \Rep(\cG^\vee_{\bk}) \cong \Perv(\Hk_\cG,\bk)
		\]
		whose essential image is stable under subquotients.
		Consider the reduced subgroup $(G^\vee_{\bk})^I_{\mathrm{red}}$.
		We will now prove that
		\begin{equation}
			\label{eqn:description-H-red}
			H \supset (G^\vee_{\bk})^I_{\mathrm{red}}.
		\end{equation}
		(Here the right-hand side is isomorphic to $\mathrm{SL}_{2,\bk}$ in case (a) and in case (b) when $\ell=2$, and to $\mathrm{PGL}_{2,\bk}$ otherwise.)
		
		First, we claim that $H$ is connected. In fact, if it were not then $\Rep(H)$ would contain a subcategory stable under tensor products and containing finitely many isomorphism classes of simple objects (see e.g.~\cite[Proposition~1.2.11(2)]{br}). Hence the same would be true for $\Perv(\Hk_\cG,\bk)$. As in the ``classical'' case (see~\cite[Lemma~1.9.3]{br}) this is impossible because $\bbX_*(T)_I$ is free.
		
		Let us note also that since the morphism $(T^\vee_\bk)^I \to (G^\vee_{\bk})^I$ factors through $\cG^\vee_\bk$ (see~\S\ref{ss:identification-dual-gp}), $H$ contains a maximal torus of $(G^\vee_{\bk})^I_{\mathrm{red}}$.
		If $H$ does not contain $(G^\vee_{\bk})^I_{\mathrm{red}}$, then $H_{\mathrm{red}}$ is a strict smooth connected subgroup containing a maximal torus; by Lemma~\ref{lem:subgroups-SL2} it is therefore either equal to the latter subgroup, or to one of the Borel subgroups containing it. In any case, any simple representation of $H_{\mathrm{red}}$ is invertible in the monoidal category $\Rep(H_{\mathrm{red}})$. On the other hand, as in the proof of~\cite[Lemma~1.14.6]{br}, for $n \gg 0$ the $n$-th Frobenius morphism $\mathrm{Fr}^n_H$ of $H$ can be written as a composition
		\[
		H \xrightarrow{\mathrm{Fr}^{n \prime}_H} (H_{\mathrm{red}})^{(n)} \hookrightarrow H^{(n)}
		\]
		where $\mathrm{Fr}^{n \prime}_H$ is a quotient morphism.
		Choosing a nontrivial simple representation of $(H_{\mathrm{red}})^{(n)}$, pulling it back to $H$ and then to $\cG^\vee_{\bk}$, we obtain a nontrivial simple representation of $\cG^\vee_{\bk}$ which is invertible in the monoidal category
		\[
		\bigl( \Rep(\cG^\vee_{\bk}), \otimes \bigr) \cong \bigl( \Perv(\Hk_{\cG},\bk), \star^0 \bigr).
		\]
		By the same considerations as in~\cite[Beginning of~\S 1.9]{br} one sees that the invertible objects in $\Perv(\Hk_{\cG},\bk)$ are the objects $\scJ_{!*}(\lambda, \bk)$ where $\lambda \in \bbX_*(T)_I$ is $W_0$-stable. Since $G$ is assumed to be semisimple of adjoint type this implies that $\lambda=0$; in other words $\Perv(\Hk_{\cG},\bk)$
		has no nontrivial simple invertible object, which provides a contradiction.
		
		Now, assume that we are in case (a), or otherwise in case (b) with $\ell \neq 2$. Then the group scheme $(G^\vee_{\bk})^I$ is smooth, so that~\eqref{eqn:description-H-red} implies that $H=(G^\vee_{\bk})^I$. Lemma~\ref{lem:fs} applied to the morphism $\tilde{f}_{\cG,\bO}$ implies that this morphism is an isomorphism. It follows that $\tilde{f}_{\cG,\bk} = \bk \otimes_{\bO} \tilde{f}_{\cG,\bO}$ is an isomorphism too, which finishes the proof.
		
		Finally we consider case (b) when $\ell=2$. If $\bK=\Q_2$, i.e.~$\bO=\Z_2$, then~\cite[Proposition~6.9]{alrr} implies that we in fact have $H=(G^\vee_{\F_2})^I$, which allows us to conclude as above. For a general $\bO$, we have a continuous morphism of rings $\Z_2 \to \bO$; using the associated change-of-scalars functors (see~\S\ref{ss:const-scalars}) we deduce the desired claim from the case of $\Z_2$.
	\end{proof}
		
	\subsection{The general case}\label{sec:general-case}
		
	We can finally complete the proof of Theorem~\ref{thm_dual_group_fixpts}.
	
	\begin{proof}[Proof of Theorem~\ref{thm_dual_group_fixpts}]
		By Lemma~\ref{thm_dual_group_fixpts} we can (and will) assume that $G$ is semisimple of adjoint type. As explained at the end of~\S\ref{ss:identification-dual-gp},
		to conclude it suffices to prove that 
		the scheme-theoretic image of $\varphi_{\cG,\Lambda}$ is $(G^\vee_\Lambda)^I$ for $\Lambda=\bK$ or $\bk$.
		
		So, we finally assume that $\Lambda$ is either $\bK$ or $\bk$, and denote by $H$
		the scheme-theoretic image of $\varphi_{\cG,\Lambda}$.
		As explained in~\S\ref{ss:application-coalg}, the formalism of constant term functors gives us morphisms $\cM^\vee_\Lambda\to \cG^\vee_\Lambda$ for every standard Levi subgroup $M\subset G$. (Here $\cM$ is the scheme-theoretic closure of $M$; see~\eqref{eqn:parabolics-cocharacters-2}.) Applying this construction to each standard Levi subgroup $M_\gamma$ attached to a simple relative root $\gamma$ and
		using Proposition~\ref{prop:dual.gp.sl2.case}, we deduce that $H$ contains each $((M_\gamma)^\vee_{\Lambda})^I$. Here, $(M_\gamma)^\vee_{\Lambda}$ identifies with the Levi subgroup of $G^\vee_\Lambda$ attached to the subset of simple roots whose coroots are the inverse image of $\gamma$ under~\eqref{eqn:surjection-Phis}, see~\S\ref{ss:application-coalg}. Therefore, the desired claim follows from~\cite[Corollary~6.6]{alrr}.
	\end{proof}
	
	\appendix

	\section{Equivariant versus constructible perverse sheaves}
	\label{app:equiv-const}
	
	In the setting considered in the body of the paper, let
	\[
	\Db_{(\Loop^+ \cG)}(\Gr_\cG,\Lambda)
	\]
	be the full subcategory of $\Dbc(\Gr_\cG,\Lambda)$ whose objects are the complexes $\scF$ such that for any $i \in \Z$ and $\lambda \in \bbX_*(T)_I^+$ the sheaf $\cH^i(\scF_{|\Gr_\cG^\lambda})$ is constant. Lemma~\ref{lem:constant-loc-sys} below implies that this subcategory is triangulated. It is easily seen that the perverse t-structure on $\Dbc(\Gr_\cG,\Lambda)$ restricts to a t-structure on $\Db_{(\Loop^+ \cG)}(\Gr_\cG,\Lambda)$, whose heart will be denoted
	\[
	\Perv_{(\Loop^+ \cG)}(\Gr_\cG,\Lambda).
	\]
	In case $\Lambda$ is a field, $\Perv_{(\Loop^+ \cG)}(\Gr_\cG,\Lambda)$ is the Serre subcategory of the category of perverse sheaves on $\Gr_\cG$ generated by the objects $\scJ_{!*}(\lambda,\Lambda)$ for $\lambda \in \bbX_*(T)_I^+$. It is clear that the functor $h^*$ (see~\eqref{eqn:pullback-Hk-Gr}) factors through an exact functor
	\begin{equation}
		\label{eqn:For-Perv}
		\Perv(\Hk_\cG,\Lambda) \to \Perv_{(\Loop^+ \cG)}(\Gr_\cG,\Lambda).
	\end{equation}
	Our goal in this appendix is to prove that~\eqref{eqn:For-Perv} is an equivalence of categories.  This will be achieved in Proposition~\ref{prop:Perv-equiv-const} below. (In the body of the paper, we use this statement through its consequence proved in Corollary~\ref{cor:passing-to-adjoint}.)
	
	\begin{rmk}
		\begin{enumerate}
			\item 
			Proposition~\ref{prop:Perv-equiv-const} is a ramified analogue of~\cite[Proposition~2.1]{mv}. Unfortunately, a direct adaptation of the proof of that result (given in~\cite[Appendix~A]{mv}; see also~\cite[\S 1.10.2]{br}) presents some technical difficulties. To bypass this problem we will give a slightly different proof of this property (based on similar considerations) which applies in both settings. 
			None of the details introduced in this proof are required elsewhere in the paper.
			\item
			In case $\Lambda=\bK$, the fact that~\eqref{eqn:For-Perv} is an equivalence can be obtained in a much simpler way, by observing that the same arguments as for Lemma~\ref{lem:semisimplicity} show that the category $\Perv_{(\Loop^+ \cG)}(\Gr_\cG,\bK)$ is semisimple.
		\end{enumerate}
	\end{rmk}
	
	\begin{lem}
		\label{lem:constant-loc-sys}
		Let $\lambda \in \bbX_*(T)^+_I$. The subcategory of the category of sheaves on $\Gr_\cG^\lambda$ whose objects are the constant local systems is stable under subquotients and extensions.
	\end{lem}
	
	\begin{proof}
		Stability under subquotients is a classical fact, known for any connected scheme of finite type. For stability under extensions, what we need to prove is that for any finitely generated $\Lambda$-modules $M,N$ the natural morphism
		\begin{equation}
			\label{eqn:Ext1-loc-sys}
			\Ext^1_\Lambda(M,N) \to \Hom_{\Dbc(\Gr_\cG^\lambda,\Lambda)}(\underline{M}_{\Gr_\cG^\lambda},\underline{N}_{\Gr_\cG^\lambda}[1])
		\end{equation}
		is an isomorphism.
		
		First, assume that $\Lambda=\bK$ or $\Lambda=\bk$. Then it suffices to treat the case $M=N=\Lambda$, and the $\Ext^1$-space vanishes. On the other hand we have
		\[
		\Hom_{\Dbc(\Gr_\cG^\lambda,\Lambda)}(\underline{\Lambda}_{\Gr_\cG^\lambda},\underline{\Lambda}_{\Gr_\cG^\lambda}[1]) = \mathsf{H}^1(\Gr_\cG^\lambda; \Lambda),
		\]
		which vanishes because $\Gr_\cG^\lambda$ is smooth (so that cohomology is dual to cohomology with compact support up to shift) and admits a paving by affine spaces (see~\S\ref{ss:Iwahori-Weyl-Schubert}).
		
		Now we consider the case $\Lambda=\bO$. 
		We have
		\[
		\Hom_{\Dbc(\Gr_\cG^\lambda,\bO)}(\underline{M}_{\Gr_\cG^\lambda},\underline{N}_{\Gr_\cG^\lambda}[1]) = \mathsf{H}^1(R\Gamma(\Gr_\cG^\lambda, \underline{R\Hom_{\bO}(M,N)})).
		\]
		Here the complex $R\Hom_\bO(M,N)$ is concentrated in degrees $0$ and $1$; more explicitly we have a distinguished triangle
		\[
		\Hom_\bO(M,N) \to R\Hom_\bO(M,N) \to \Ext^1_\bO(M,N)[-1] \xrightarrow{[1]}.
		\]
		Using the fact that
		\[
		\mathsf{H}^1(R\Gamma(\Gr_\cG^\lambda,\underline{\Hom_\bO(M,N)}))=0
		\]
		(for the same reason as above) and that $\Gr_\cG^\lambda$ is connected, applying the functor $R\Gamma(\Gr_\cG^\lambda,\underline{?})$ to this triangle we obtain an embedding
		\[
		\mathsf{H}^1(R\Gamma(\Gr_\cG^\lambda, \underline{R\Hom_{\bO}(M,N)})) \hookrightarrow \Ext^1_\bO(M,N).
		\]
		It is clear that the composition of~\eqref{eqn:Ext1-loc-sys} with this morphism is the identity morphism of $\Ext^1_\bO(M,N)$. Hence both of this morphisms are isomorphisms, which finishes the proof.
	\end{proof}

	\begin{prop}
		\label{prop:Perv-equiv-const}
		The functor~\eqref{eqn:For-Perv} is an equivalence of categories.
	\end{prop}
	
	\begin{proof}
		For any $\Loop^+\cG$-stable locally closed subscheme $X \subset \Gr_\cG$ 
		such that $|X|$
		has finitely many $\Loop^+\cG$-orbits, we can consider as above the categories
		\[
		\Perv([\Loop^+\cG \backslash X]_{\et},\Lambda) \quad \text{and} \quad \Perv_{(\Loop^+\cG)}(X,\Lambda)
		\]
		of perverse sheaves which are $\Loop^+\cG$-equivariant and constant on each $\Loop^+\cG$-orbit, respectively, and the natural exact ``forgetful'' functor
		\[
		\Perv([\Loop^+\cG \backslash X]_{\et},\Lambda) \to \Perv_{(\Loop^+\cG)}(X,\Lambda).
		\]
		We will prove that this functor is an equivalence of categories for any $X$, which will imply the proposition. Note that the general theory of perverse sheaves implies that this functor is fully faithful, and that its essential image is stable under subquotients. What we have to prove is that it is also essentially surjective, and for that it suffices to show that any object in $\Perv_{(\Loop^+\cG)}(X,\Lambda)$ is a quotient of an object in the image of $\Perv([\Loop^+\cG \backslash X]_\et,\Lambda)$. In case $\Lambda = \bK$ or $\Lambda=\bk$, the category $\Perv_{(\Loop^+\cG)}(X,\Lambda)$ is a finite-length abelian category, whose simple objects are the perverse sheaves $\scJ_{!*}(\lambda,\Lambda)_{|X}$ where $\lambda \in \bbX_*(T)^+_I$ is such that $|\Gr_\cG^\lambda| \subset |X|$. 
		In case $\Lambda=\bO$, the same arguments as in~\cite[Lemma~2.1.4]{rsw} show that any object is an extension of objects which are quotients of objects of the form $\scJ_{!*}(\lambda,\Lambda)_{|X}$ where $\lambda$ satisfies the same conditions. Hence, in any case, to conclude it suffices to prove that for any such $\lambda$ the object $\scJ_{!*}(\lambda,\Lambda)_{|X}$ is a quotient of an object in $\Perv([\Loop^+\cG \backslash X]_\et,\Lambda)$ whose image in $\Perv_{(\Loop^+\cG)}(X,\Lambda)$ is projective.
		
		We will proceed by induction on the number of orbits contained in $|X|$.
		Of course, there is nothing to prove in case $X$ is empty. So, we consider a nonempty $X$ as above, and
		choose some $\mu \in \bbX_*(T)^+_I$ such that $\Gr_\cG^\mu$ is closed in $X$ and such that the claim is known for the complement $X'$ of $\Gr_\cG^\mu$ in $X$. The open immersion $X' \to X$ will be denoted $j$. If $\lambda \in \bbX_*(T)^+_I$ is such that $|\Gr_\cG^\lambda| \subset |X|$ and $\lambda \neq \mu$ (i.e.~$|\Gr_\cG^\lambda| \subset |X'|$), then as explained in~\S\ref{ss:representability} there exists a projective object $\scP$ in $\Perv([\Loop^+\cG \backslash X']_{\et},\Lambda)$ and a surjection
		\[
		\scP \twoheadrightarrow \scJ_{!*}(\lambda,\Lambda)_{|X'} = j^*(\scJ_{!*}(\lambda,\Lambda)_{|X}).
		\]
		By induction the image of $\scP$ in $\Perv_{(\Loop^+\cG)}(X',\Lambda)$ is projective, and adjunction provides a morphism
		\[
		\pH^0(j_! \scP) \to \scJ_{!*}(\lambda,\Lambda)_{|X}.
		\]
		This morphism can be written as a composition
		\[
		\pH^0(j_! \scP) \to j_{!*} \scP \to j_{!*} \bigl( \scJ_{!*}(\lambda,\Lambda)_{|X'} \bigr) = \scJ_{!*}(\lambda,\Lambda)_{|X}
		\]
		where the first morphism is surjective by definition, and the second one is also surjective because intermediate extension functors preserve surjections.
		Since the functor $j^!=j^* \colon \Perv_{(\Loop^+\cG)}(X,\Lambda) \to \Perv_{(\Loop^+\cG)}(X',\Lambda)$ is exact, by adjunction again the object $\pH^0(j_! \scP)$ is projective; the problem is therefore solved for these $\lambda$'s.
		
		It remains to treat the case $\lambda=\mu$. In this case we have
		\[
		\scJ_{!*}(\mu,\Lambda)_{|X} = \underline{\Lambda}_{\Gr_\cG^\mu}[\dim(\Gr_\cG^\mu)].
		\]
		Using the notation of~\S\ref{ss:representability},
		we claim that there exists a surjective morphism
		\[
		\scP_X(\mu,\Lambda) \to \scJ_{!*}(\mu,\Lambda)_{|X}.
		\]
		In fact, using Lemma~\ref{lem:representability-F-X} we see that
		\[
		\Hom(\scP_X(\mu,\Lambda), \scJ_{!*}(\mu,\Lambda)_{|X}) \cong \sF^X_{\cG,\mu}(\scJ_{!*}(\mu,\Lambda)_{|X}) = \Lambda.
		\]
		The element $1 \in \Lambda$ defines a morphism $\scP_X(\mu,\Lambda) \to \scJ_{!*}(\mu,\Lambda)_{|X}$, which is surjective in case $\Lambda$ is a field because its codomain is simple, and in case $\Lambda=\bO$ because its image under derived tensor product with $\bk$ is surjective (see Lemma~\ref{lem:change-scalars-P}).
		
		We claim that the functor
		\[
		\mathsf{H}^{\langle \mu, 2\rho \rangle}_{\rmT_\mu \cap X}(X,-) \colon \Perv_{(\Loop^+\cG)}(X,\Lambda) \to \modf_\Lambda
		\]
		is exact.
		In fact, as noted in~\S\ref{ss:representability},
		for any $\scF \in \Perv([\Loop^+\cG \backslash X]_{\et},\Lambda)$ we have
		\[
		\mathsf{H}^{k}_{\rmT_\mu \cap X}(X,\scF) = 0
		\]
		unless $k=\langle \mu, 2\rho \rangle$. Since any object in $\Perv_{(\Loop^+\cG)}(X,\Lambda)$ is an extension of objects in $\Perv([\Loop^+\cG \backslash X]_\et,\Lambda)$, the same property holds for $\scF \in \Perv_{(\Loop^+\cG)}(X,\Lambda)$, which implies the desired exactness. For simplicity, this functor will also be denoted $\sF_{\cG,\mu}^X$.
		
		We now consider the category $\mathscr{C}$ constructed as in~\cite[\S 1]{vilonen} out of the following data:
		\begin{itemize}
			\item
			the categories are $\mathscr{A}=\Perv_{(\Loop^+\cG)}(X',\Lambda)= \Perv([\Loop^+\cG \backslash X']_\et,\Lambda)$ and $\mathscr{B}=\modf_\Lambda$;
			\item
			the functors $\mathbf{F},\mathbf{G} \colon \mathscr{A} \to \mathscr{B}$ are 
			\[
			\mathbf{F}= \sF_{\cG,\mu}^X \circ \pH^0 \circ j_! \quad \text{and} \quad \mathbf{G}= \sF_{\cG,\mu}^X \circ \pH^0 \circ j_*; 
			\]
			\item
			the morphism $\mathbf{T} \colon \mathbf{F} \to \mathbf{G}$ is the morphism induced by the canonical morphism $j_! \to j_*$.
		\end{itemize}
		Explicitly,
		the objects in $\mathscr{C}$ are the quadruples $(\scF,V,m,n)$ where $\scF$ is an object in $\Perv_{(\Loop^+\cG)}(X',\Lambda) = \Perv([\Loop^+\cG \backslash X']_\et,\Lambda)$, $V$ is an object in $\modf_\Lambda$, and
		\[
		m \colon \sF_{\cG,\mu}^X( \pH^0 (j_!(\scF))) \to V, \quad n \colon V \to \sF_{\cG,\mu}^X( \pH^0 ( j_*(\scF)))
		\]
		are morphisms such that $nm$ is induced by our morphism of functors $\mathbf{T}$. The morphisms in this category are pairs consisting of a morphism in $\Perv_{(\Loop^+\cG)}(X',\Lambda)$ and a morphism in $\modf_\Lambda$, which make the obvious diagram commutative. For later use we note that the functor $\pH^0 \circ j_*$ takes values in the subcategory $\Perv([\Loop^+\cG \backslash X]_\et,\Lambda)$ of $\Perv_{(\Loop^+\cG)}( X,\Lambda)$; for $\scF \in \Perv_{(\Loop^+\cG)}(X',\Lambda)$ we therefore have
		\[
		\sF_{\cG,\mu}^X \circ \pH^0 \circ j_*(\scF) = \Hom(\scP_X(\mu,\Lambda), \pH^0 ( j_*(\scF))) = \Hom(j^* \scP_X(\mu,\Lambda), \scF);
		\]
		in other words, the functor $\sF_{\cG,\mu}^X \circ \pH^0 \circ j_*$ is represented by $j^* \scP_X(\mu,\Lambda)$.
		
		By~\cite[Proposition~1.1(a)]{vilonen}, $\mathscr{C}$ is an abelian category and, by~\cite[Proposition~1.2]{vilonen}, the functor
		\[
		E \colon \Perv_{(\Loop^+\cG)}(X,\Lambda) \to \mathscr{C}
		\]
		sending $\scG$ to the quadruple
		\[
		(j^* \scG, \sF_{\cG,\mu}^X(\scG),m,n)
		\]
		where $m,n$ are the obvious morphisms (provided by adjunction) is fully faithful and exact. (We apply this proposition with $\widetilde{\mathscr{B}}$ the full subcategory of perverse sheaves supported on $\Gr_\cG^\mu$, which is equivalent to the category of constant sheaves on $\Gr_\cG^\mu$, see Lemma~\ref{lem:constant-loc-sys}.) To prove that the image of $\scP_X(\mu,\Lambda)$ in $\Perv_{(\Loop^+\cG)}(X,\Lambda)$ is projective, it therefore suffices to prove that its image in $\mathscr{C}$ is projective. What we will show is that this image represents the functor
		\begin{equation}
			\label{eqn:proj-C-modf}
			\mathscr{C} \to \modf_\Lambda
		\end{equation}
		given by $(\scF,V,m,n) \to V$. This functor is clearly exact, which will imply projectivity and finish the proof.
		
		In more concrete terms we will prove that
		\[
		E(\scP_X(\mu,\Lambda)) = (j^*\scP_X(\mu,\Lambda), \sF_{\cG,\mu}^X( \pH^0 (j_! j^* \scP_X(\mu,\Lambda))) \oplus \Lambda, m, n)
		\]
		where $m$ is the obvious embedding and $n$ is the sum of the morphism induced by $\mathbf{T}$ with the morphism
		\[
		\Lambda \to \sF_{\cG,\mu}^X (\pH^0 ( j_* j^* \scP_X(\mu,\Lambda))) = \End(j^* \scP_X(\mu,\Lambda))
		\]
		sending $1 \in \Lambda$ to the identity morphism. From this description, and using the fact that $j^* \scP_X(\mu,\Lambda)$ represents $\sF_{\cG,\mu}^X \circ \pH^0 \circ j_*$, it is not difficult to check that $E(\scP_X(\mu,\Lambda))$ indeed represents the functor~\eqref{eqn:proj-C-modf}. (In case $\Lambda$ is a field, this object
		is the projective cover of the simple object $(0,\Lambda,0,0)$ described in~\cite[p.~667]{vilonen}, see also~\cite[p.~317]{mirollo-vilonen}.) By definition, the second component in $E(\scP_X(\mu,\Lambda))$ is
		\[
		\sF^X_{\cG,\mu}(\scP_X(\mu,\Lambda)) = \End(\scP_X(\mu,\Lambda)),
		\]
		where we use Lemma~\ref{lem:representability-F-X}.
		To prove the above claim it therefore suffices to prove that the natural morphism
		\begin{multline*}
			\sF_{\cG,\mu}^X( \pH^0 (j_! j^* \scP_X(\mu,\Lambda))) \oplus \Lambda \\
			= \Hom(\scP_X(\mu,\Lambda), \pH^0 (j_! j^* \scP_X(\mu,\Lambda))) \oplus \Lambda \to \End(\scP_X(\mu,\Lambda))
		\end{multline*}
		given by the sum of the morphism induced by the adjunction morphism
		\[
		\pH^0 (j_! j^* \scP_X(\mu,\Lambda)) \to \scP_X(\mu,\Lambda) 
		\]
		and the morphism $\Lambda \to \End(\scP_X(\mu,\Lambda))$ sending $1 \in \Lambda$ to the identity morphism is an isomorphism.
		
		First we consider the case when $\Lambda$ is a field, i.e.~$\Lambda=\bK$ or $\Lambda=\bk$.
		The projective object $\scP_{X}(\mu,\Lambda)$ in the highest weight category $\Perv([\Loop^+\cG \backslash X]_\et, \Lambda)$ (see Remark~\ref{rmk:hw-category}) admits a standard filtration; hence there exists an exact sequence
		\begin{equation}
			\label{eqn:standard-filtration-P}
			\scF_1 \hookrightarrow \scP_{X}(\mu,\Lambda) \twoheadrightarrow \scF_2
		\end{equation}
		where $\scF_1$ is an extension of objects $\scJ_!(\lambda, \Lambda)_{|X}$ where $\lambda \in \bbX_*(T)^+_I$ is such that $|\Gr_\cG^\lambda| \subset |X|$ and $\lambda \neq \mu$, and $\scF_2$ is a direct sum of copies of $\scJ_!(\mu, \Lambda)_{|X} = \underline{\Lambda}_{\Gr_\cG^\mu}[\dim(\Gr^\mu_\cG)]$. Here we necessarily have
		\[
		\scF_1 = \pH^0(j_! j^* \scP_{X}(\mu,\Lambda)) 
		\]
		and the multiplicity of $\scJ_!(\mu, \Lambda)_{|X}$ in $\scF_2$ is the multiplicity of $\scJ_!(\mu, \Lambda)$ in $\scP_Z(\mu,\Lambda)$ (where $Z \subset \Gr_\cG$ is any closed subscheme as in~\S\ref{ss:representability} in which $X$ is open), i.e.~the dimension of
		\[
		\Hom(\scP_Z(\mu,\Lambda), \scJ_*(\mu, \Lambda)) \cong \sF_{\cG,\mu}(\scJ_*(\mu, \Lambda)) = \Lambda.
		\]
		Applying the exact functor $\sF_{\cG,\mu}^X$ to the exact sequence~\eqref{eqn:standard-filtration-P} we obtain an exact sequence
		\[
		\sF_{\cG,\mu}^X(\pH^0(j_! j^* \scP_{X}(\mu,\Lambda))) \hookrightarrow \sF_{\cG,\mu}^X(\scP_{X}(\mu,\Lambda)) \twoheadrightarrow \sF_{\cG,\mu}^X(\scF_2),
		\]
		where the rightmost term is $1$-dimensional. 
		Since the identity morphism of the perverse sheaf $\scP_{X}(\mu,\Lambda)$ does not factor through $\pH^0(j_! j^* \scP_{X}(\mu,\Lambda))$, we therefore have
		\[
		\sF_{\cG,\mu}^X(\scP_{X}(\mu,\Lambda)) = \sF_{\cG,\mu}^X(\pH^0(j_! j^* \scP_{X}(\mu,\Lambda))) \oplus \Lambda
		\]
		where the $1$-dimensional factor corresponds to $\id \in \End(\scP_{X}(\mu,\Lambda))$. This completes the proof in this case.
		
		Finally we consider the case $\Lambda=\bO$. Using the fact that $\bk \lotimes_\bO \scP_{X}(\mu,\bO) = \scP_{X}(\mu,\bk)$ is perverse (see Lemma~\ref{lem:change-scalars-P}) and the vanishing property~\eqref{eqn:vanishing-local-cohom} one checks that
		\[
		\bk \lotimes_{\bO} \sF_{\cG,\mu}^X(\scP_{X}(\mu,\bO)) = \sF_{\cG,\mu}^X(\scP_{X}(\mu,\bk)),
		\]
		so that in particular $\sF_{\cG,\mu}^X(\scP_{X}(\mu,\bO))$ is free over $\bO$. Similarly, we claim that
		\[
		\bk \lotimes_{\bO} \pH^0(j_! j^* \scP_{X}(\mu,\bO))
		\]
		is perverse, and hence identifies with $\pH^0(j_! j^* \scP_{X}(\mu,\bk))$. In fact, this follows from the fact that $\scP_{Z}(\mu,\bO)$ is a direct summand in $\scP_Z(\bO)$ (where $Z \subset \Gr_\cG$ is as above in the proof), which admits a filtration with subquotients of the form $\scJ_!(\lambda,\bO)$ (see~\S\ref{ss:structure-projective}), and Lemma~\ref{lem:properties-J!*}\eqref{it:properties-J!*-2}. From this claim we deduce as above that we have
		\[
		\bk \lotimes_{\bO} \sF_{\cG,\mu}^X(\pH^0(j_! j^* \scP_{X}(\mu,\bO))) = \sF_{\cG,\mu}^X(\pH^0(j_! j^* \scP_{X}(\mu,\bk))),
		\]
		hence in particular that $\sF_{\cG,\mu}^X(\pH^0(j_! j^* \scP_{X}(\mu,\bO)))$ is free over $\bO$. We now have a morphism of finite free $\bO$-modules
		\[
		\sF_{\cG,\mu}^X( \pH^0 (j_! j^* \scP_X(\mu,\bO))) \oplus \bO \to \sF^X_{\cG,\mu}(\scP_X(\mu,\bO))
		\]
		whose image under the functor $\bk \otimes_{\bO} (-)$ is an isomorphism by the case $\Lambda=\bk$ treated above; it follows that this morphism itself is an isomorphism.
	\end{proof}
	
	Let us consider as in the proof of Proposition~\ref{prop:commutativity}
	the quotient map $G\to G_{\rm ad}$ to the adjoint group, the induced morphism $\cG\to\cG_{\rm ad}$ on special parahoric group schemes, and finally the associated map $\Hk_\cG\to \Hk_{\cG_{\rm ad}}$ on Hecke stacks. We also have an induced
	map $\pi_1(G)_I\to \pi_1(G_{\rm ad})_I$, denoted $\tau\mapsto \tau_{\rm ad}$, and for any $\tau$ a map between the associated connected components
	\begin{equation}
		\label{eq:passingtoadjoint}
		\Hk_{\cG}^\tau \to \Hk_{\cG_{\rm ad}}^{\tau_{\rm ad}}.
	\end{equation}
	
	\begin{cor}
		\label{cor:passing-to-adjoint}
		For each $\tau\in \pi_1(G)_I$ the map \eqref{eq:passingtoadjoint} induces an equivalence of categories
		\[
		\Perv(\Hk_\cG^\tau,\Lambda)\overset{\sim}{\longrightarrow}\Perv(\Hk_{\cG_{\rm ad}}^{\tau_{\rm ad}},\Lambda).
		\]
	\end{cor}
	\begin{proof}
		The map $\Gr_\cG^\tau\to \Gr_{\cG_{\rm ad}}^{\tau_{\rm ad}}$ is a universal homeomorphism by \cite[Proposition 3.5]{HainesRicharz:CM}, so induces an equivalence on categories of sheaves.
		The corollary is now immediate from Proposition \ref{prop:Perv-equiv-const}. 
	\end{proof}

	\section{\'Etale sheaves on stacks}
	\label{app:sheaves}
		
	\subsection{Constructible derived categories of stacks}
	\label{ss:const}
	
	In this paper we consider constructible derived categories of \'etale sheaves on some stacks. In this subsection we make a few comments on the definition of such categories, and give appropriate references.
	
	\subsubsection{Finite coefficients}
	
	Let $S$ be a base scheme which is the spectrum of a field, either finite or separably closed. First, assume that $\Lambda$ is a finite field, of characteristic $\ell$ which is invertible on $S$. For any Artin stack $X$ of finite type over $S$, in~\cite{laszlo-olsson} the authors define the constructible derived category $\Dc(X,\Lambda)$ of \'etale sheaves of $\Lambda$-modules on $X$, together with their bounded versions $\Dc^+(X,\Lambda)$, $\Dc^-(X,\Lambda)$, $\Dbc(X,\Lambda)$. They also define bifunctors
	\begin{gather*}
		(-) \lotimes_\Lambda (-) \colon \Dc^-(X,\Lambda) \times \Dc^-(X,\Lambda) \to \Dc^-(X,\Lambda), \\
		R\sHom_\Lambda(-,-) \colon \Dc^-(X,\Lambda) \times \Dc^+(X,\Lambda) \to \Dc^+(X,\Lambda)
	\end{gather*}
	and, for an $S$-morphism $f \colon X \to Y$ between such stacks,
	push/pull functors
	\begin{gather*}
		f_* \colon \Dc^+(X,\Lambda) \to \Dc^+(Y,\Lambda), \quad f_! \colon \Dc^-(X,\Lambda) \to \Dc^-(Y,\Lambda), \\
		f^* \colon \Dc(Y,\Lambda) \to \Dc(X,\Lambda), \quad f^! \colon \Dc(Y,\Lambda) \to \Dc(X,\Lambda).
	\end{gather*}
	They prove that these functors satisfy the familiar properties usually gathered under the term ``six functors formalism,'' with the exception of the base change theorem, for which only weaker versions are obtained. 
	
	A more general formalism is constructed in~\cite{liu-zheng}, based on the theory of $\infty$-categories. (This does not require assumptions on $S$.) Passing to homotopy categories in their construction one recovers the categories of~\cite{laszlo-olsson} in the setting above, see~\cite[\S 6.5]{liu-zheng}. Their constructions also provide alternative constructions for the push/pull functors (assuming $f$ is quasi-separated for the functor $f_*$) and the other functors considered above, and they show that they do satisfy the base change theorem in its usual form.
	
	By construction the (bi)functors $f^*$, $f^!$, $R\sHom_\Lambda(-,-)$ and $(-) \lotimes_\Lambda (-)$ restrict to bounded constructible categories. For the functors $f_*$ and $f_!$ this is not always true (e.g.~for the natural morphism $\mathrm{pt} \to \mathrm{pt}/\Gm$), but it \emph{is} in case $f$ is obtained from an equivariant morphism of schemes by passing to quotient schemes (with respect to the action of an affine group scheme of finite type). This is the only case we consider in the body of the paper.
	
	In~\cite{laszlo-olsson-perv}, the authors explain the construction of the perverse t-structure (associated with the middle perversity) on the category $\Dbc(X,\Lambda)$. See also~\cite[\S 3]{liu-zheng-2} for an alternative construction.
	
	\subsubsection{Adic coefficients}
	\label{sss:coeff-O}
	
	We continue with the geometric setting above, and with our prime number $\ell$ invertible on $S$, but take now for $\Lambda$ the ring of integers in a finite extension of $\Q_\ell$. In~\cite{laszlo-olsson-2} the authors explain how to define for any Artin stack of finite type $X$ the constructible derived category $\Dc(X,\Lambda)$ of \'etale sheaves of $\Lambda$-modules on $X$, together with their bounded versions $\Dc^+(X,\Lambda)$, $\Dc^-(X,\Lambda)$, $\Dbc(X,\Lambda)$. They also define bifunctors $(-) \lotimes_\Lambda (-)$ and $R\sHom_\Lambda(-,-)$ and, for any $S$-morphism $f \colon X \to Y$ between such stacks, push/pull functors $f_*$, $f_!$, $f^*$ and $f^!$ as above. They prove that these functors satisfy all the expected properties, except for the base change theorem. 
	
	A more general formalism is constructed in~\cite{liu-zheng-2}, based on the theory of $\infty$-categories; see in particular~\cite[\S 2.1 and~\S 2.3]{liu-zheng-2}. Passing to homotopy categories in their construction one recovers the categories of~\cite{laszlo-olsson-2}, see~\cite[\S 2.5]{liu-zheng-2}. Their constructions also provide alternative constructions for the push/pull functors (assuming $f$ is quasi-separated for the functor $f_*$) and the bifunctors considered above, and they show that they do satisfy the base change theorem in its usual form.
	
	The same comments as above apply regarding restrictions to bounded derived categories. In the case $X$ is a scheme, the category $\Dbc(X,\Lambda)$ is equivalent to the version defined by Deligne, as explained in~\cite[\S 3.1]{laszlo-olsson-2}.
	
	In this setting also we have a perverse t-structure (for the middle perversity), as explained in~\cite{laszlo-olsson-perv} and in~\cite[\S 3]{liu-zheng-2}.
	
	\subsubsection{Characteristic-\texorpdfstring{$0$}{0} coefficients}
	\label{sss:coeff-K}
	
	We consider once again the geometric setting above and our prime number $\ell$, and take for $\Lambda$ a finite extension of $\Q_\ell$. For an Artin stack of finite type $X$, one can define the derived category $\Dc(X,\Lambda)$ following~\cite[Remark~3.1.7]{laszlo-olsson-2} (see also the discussion in~\cite[\S 6]{zheng}). Namely, denoting by $\Lambda_0$ the ring of integers in $\Lambda$, one defines $\Dc(X,\Lambda)$ as the Verdier quotient of $\Dc(X,\Lambda_0)$ by the triangulated subcategory consisting of complexes all of whose cohomology objects are annihilated by a power of a uniformizer. One can make similar definitions for the bounded versions.
	
	We will denote by
	\[
	\Lambda \otimes_{\Lambda_0} (-) \colon \Dbc(X,\Lambda_0) \to \Dbc(X,\Lambda)
	\]
	the quotient functor; it has the property that for any complexes $\scF, \scG \in \Dbc(X,\Lambda_0)$ we have a canonical identification
	\[
	\Lambda \otimes_{\Lambda_0} \Hom_{\Dbc(X,\Lambda_0)}(\scF,\scG) \simto \Hom_{\Dbc(X,\Lambda)}(\Lambda \otimes_{\Lambda_0} \scF, \Lambda \otimes_{\Lambda_0} \scG).
	\]
	
	The six operations for sheaves with coefficients in $\Lambda_0$ induce functors between the versions for $\Lambda$, which we will denote in a similar way, and these functors satisfy the same properties. Moreover, by construction the functor $\Lambda \otimes_{\Lambda_0} (-)$ commutes with all sheaf operations. Finally we have a perverse t-structure in this case too, such that the functor $\Lambda \otimes_{\Lambda_0} (-)$ is t-exact.
	
	\subsubsection{Change of scalars}
	\label{ss:const-scalars}
	
	An important role in our constructions is played by some ``change of scalars'' functors, which are defined as follows. First we consider an extension of finite fields $\Lambda \to \Lambda'$. For such data, and for any Artin stack $X$ as above, as explained in~\cite[\S 6.2]{liu-zheng}\footnote{It is not stated explicitly in~\cite{liu-zheng} that these functors send constructible complexes to constructible complexes. However, by definition it suffices to prove this property in the case of schemes, where it is clear. Similar comments apply for the variants considered below.} there exist natural adjoint functors
	\[
	\Lambda' \otimes_{\Lambda} (-) \colon \Dbc(X,\Lambda) \to \Dbc(X,\Lambda'), \quad \Dbc(X,\Lambda') \to \Dbc(X,\Lambda)
	\]
	which we will call extension and restriction of scalars, respectively. (The second functor will be given no notation.) By construction these functors commute with all sheaf-theoretic functors, and they are t-exact for the perverse t-structures.
	
	Next we consider an extension between finite extensions of $\Q_\ell$, and the induced morphism $\Lambda \to \Lambda'$ between the rings of integers. This morphism induces a morphism between the associated diagrams involved in the $\mathfrak{m}$-adic formalism of~\cite[\S 2]{liu-zheng-2}, and for any stack $X$ as above we have associated adjoint functors
	\[
	\Lambda' \otimes_{\Lambda} (-) \colon \Dbc(X,\Lambda) \to \Dbc(X,\Lambda'), \quad \Dbc(X,\Lambda') \to \Dbc(X,\Lambda).
	\]
	Once again, these functors commute with all sheaf-theoretic constructions in the obvious way, and they are t-exact for the perverse t-structures. (We also have similar functors for the fraction fields of $\lambda$ and $\Lambda'$, but they will not be considered in this paper.)
	
	Finally, we consider the case when $\Lambda$ is the ring of integers in a finite extension of $\Q_\ell$, and $\Lambda'$ is its residue field. The definition of $\Dbc(X,\Lambda)$ involves the ``ringed diagram'' $(\mathbb{N}, \Lambda_\bullet)$ in the notation of~\cite[\S 2.1]{liu-zheng-2}. There exists a natural morphism from this ringed diagram to the ringed diagram $(\{0\}, \Lambda')$, and associated with this diagram we have natural adjoint functors
	\[
	\Lambda' \lotimes_{\Lambda} (-) \colon \Dbc(X,\Lambda) \to \Dbc(X,\Lambda'), \quad \Dbc(X,\Lambda') \to \Dbc(X,\Lambda).
	\]
	Here again, these functors commute with all sheaf-theoretic constructions in the obvious way. The right-hand functor is t-exact for the perverse t-structures, but the functor $\Lambda' \lotimes_{\Lambda} (-)$ is only left t-exact.
	
	\subsubsection{Nearby cycles}
	\label{sss:const-nc}
	
	The last ingredient from (\'etale) sheaves theory we use in the paper is an appropriate version of the nearby cycles functor. (See Remark~\ref{rmk:nc} for the relation with more standard constructions.) First we assume that $\Lambda$ is a finite field. Let $R$ be an absolutely integrally closed valuation ring of rank $1$, set $S:=\Spec(R)$, and denote by $s$ and $\eta$ the special and generic points in $S$ respectively. (See~\S\ref{ss:definition-nc} for an example of this setting; in this example the scheme playing the role of $S$ is denoted $\overline{S}$, and the generic point is denoted $\overline{\eta}$.) Given an Artin stack of finite type $X$ over $S$, we set
	\[
	X_s := X \times_S s, \quad X_\eta := X \times_S \eta,
	\]
	and consider the obvious morphisms
	\[
	i \colon X_s \to X, \quad j \colon X_\eta \to X.
	\]
	We then set
	\[
	\Psi := i^* j_* \colon D(X_\eta, \Lambda) \to D(X_s, \Lambda),
	\]
	where the derived categories of sheaves we consider here are the categories denoted $\mathrm{D}_{\mathrm{cart}}(X_{\mathrm{lis-\acute{e}t}}, \Lambda)$ in~\cite{liu-zheng}. It follows from~\cite[Corollary~4.2]{hs} that this functor restricts to a functor
	\[
	\Dbc(X_\eta, \Lambda) \to \Dbc(X_s, \Lambda).
	\]
	By~\cite[Lemma~6.3]{hs} this functor is t-exact for the perverse t-structures.
	(The authors in~\cite{hs} work with schemes and not with stacks, but the properties we consider here can be tested on schemes.)
	
	The same constructions can be considered for the other choices of coefficients considered in~\S\S\ref{sss:coeff-O}--\ref{sss:coeff-K}.
	
	\subsection{A semismallness criterion}
	\label{ss:semismall}
	
	Let $k$ be a separably closed field, and let $X$, $Y$ be $k$-schemes of finite type. 
	We assume we are given some finite sets $\mathcal{A}$ and $\mathcal{B}$ any, for any $\alpha \in \mathcal{A}$, resp.~$\beta \in \mathcal{B}$, a smooth and irreducible locally closed subscheme $X_\alpha \subset X$, resp.~$Y_\beta \subset Y$, such that
	\[
	|X| = \bigsqcup_{\alpha \in \mathcal{A}} |X_\alpha|, \quad \text{resp.} \quad |Y| = \bigsqcup_{\beta \in \mathcal{B}} |Y_\beta|,
	\]
	and such that the closure of each stratum is a union of strata. We will denote by $j_\alpha \colon X_\alpha \to X$ and $j_\beta \colon Y_\beta \to Y$ the inclusions.  Recall that a morphism of schemes $f \colon X \to Y$ is said to be:
	\begin{itemize}
		\item
		\emph{stratified locally trivial} if for any $\alpha \in \mathcal{A}$ the subset $f(X_\alpha)$ is a union of strata in $Y$ and if moreover for any $\alpha \in \mathcal{A}$ and $\beta \in \mathcal{B}$ such that $Y_\beta \subset f(X_\alpha)$ the restriction of $f$ to a morphism $f^{-1}(Y_\beta) \cap X_\alpha \to Y_\beta$ is an \'etale locally trivial fibration;
		\item
		\emph{stratified semismall} if for any $\alpha \in \mathcal{A}$, $\beta \in \mathcal{B}$ and $y \in Y_\beta$ we have
		\begin{equation}
			\label{eqn:ineq-semismall}
			\dim(f^{-1}(y) \cap X_\alpha) \leq \textstyle\frac{1}{2} (\dim(X_\alpha)-\dim(Y_\beta)).
		\end{equation}
	\end{itemize}
	It is well known that if $f$ is proper, stratified locally trivial and stratified semi-small, then the push-forward functor $f_*$ (for any ring of coefficients as considered in the paper) sends perverse sheaves on $X$ that are constructible with respect to the given stratification of $X$ to perverse sheaves on $Y$ that are constructible with respect to the given stratification on $Y$.  (See, for instance,~\cite[Lemma~4.3]{mv} or~\cite[Proposition~1.6.1]{br}.)
	
	It turns out that this statement has a partial converse, which we will now explain. Let $\ell$ be a prime number that is invertible in $k$, and denote by $\bK$ a finite extension of $\Q_\ell$.
	Assume that the following conditions are satisfied:
	\begin{itemize}
		\item 
		for any $\alpha, \alpha' \in \mathcal{A}$
		and any $n \in \Z$ the sheaf
		\[
		\mathscr{H}^n((j_{\alpha'})^* (j_\alpha)_* \underline{\bK}_{X_\alpha})
		\]
		is constant. 
		\item
		for any $\alpha \in \mathcal{A}$ we have $\mathsf{H}^1(X_\alpha ; \bK)=0$.
	\end{itemize}
	Then as in~\cite[\S\S 2.2.9--2.2.18]{bbd} one can consider the full triangulated subcategory $\Db_{\mathcal{A}}(X,\bK)$ of $\Dbc(X,\bK)$
	whose objects are the complexes $\scF$ such that for any $n \in \Z$ and $\alpha \in \mathcal{A}$ the sheaf $\mathscr{H}^n((j_\alpha)^* \scF)$ is constant.
	The perverse t-structure on $\Dbc(X,\bK)$ restricts to a t-structure on $\Db_{\mathcal{A}}(X,\bK)$; in particular, for any $\alpha \in \mathcal{A}$ we have the intersection cohomology complex $\IC(X_\alpha, \underline{\bK}) \in \Db_{\mathcal{A}}(X,\bK)$ associated with the constant local system on $X_\alpha$, which is a simple perverse sheaf.
	Similarly, under the analgous assumptions on $Y$ and its stratification,
	one can consider the category $\Db_{\mathcal{B}}(Y,\bK)$ and its perverse t-structure.
	
	Let us assume that both of these conditions are satisfied, and consider a morphism $f \colon X \to Y$.
	
	\begin{lem}
		\label{lem:pushforward-semismall}
		Let $f \colon X \to Y$ be a morphism of $k$-schemes, and assume that the following conditions are satisfied:
		\begin{itemize}
			\item 
			for any $\alpha \in \mathcal{A}$, $\beta \in \mathcal{B}$ and $n \in \Z$ the sheaf
			\[
			\mathscr{H}^n((j_\beta)^* f_! \IC(X_\alpha, \underline{\bK}))
			\]
			is constant (in other words, for any $\alpha \in \mathcal{A}$ the complex $f_! \IC(X_\alpha, \underline{\bK})$ belongs to $\Db_{\mathcal{B}}(Y,\bK)$);
			\item
			for any $\alpha \in \mathcal{A}$ 
			the complex $f_! \IC(X_\alpha, \underline{\bK})$ is a perverse sheaf.
		\end{itemize}
		Then $f$ is stratified semismall.
	\end{lem}
	
	\begin{proof}
		We need to prove the inequality~\eqref{eqn:ineq-semismall} for any $\alpha \in \mathcal{A}$, $\beta \in \mathcal{B}$ and $y \in Y_\beta$.
		We proceed by induction on $\alpha$, with respect to the order given by inclusions of closures of strata.
		Let $\partial X_\alpha := \overline{X_\alpha} \smallsetminus X_\alpha$, so that we can assume the claim is known for any $\alpha'$ such that $X_{\alpha'} \subset \partial X_\alpha$. Fix also $\beta \in \mathcal{B}$ and $y \in Y_\beta$. The assumption that $f_! \IC(X_\alpha, \underline{\bK})$ is perverse implies that the $\bK$-vector space
		\[
		\sH^j \bigl( (f_! \IC(X_\alpha, \underline{\bK}))_y  \bigr)  = \sH^j_c( f^{-1}(y), \IC(X_\alpha, \underline{\bK}){}_{| f^{-1}(y)})
		\]
		vanishes unless $j \leq -\dim(Y_\beta)$. On the other hand, $\IC(X_\alpha, \underline{\bK})$ is supported on $\overline{X_\alpha}$, and its restriction to $X_\alpha$ is $\underline{\bK}_{X_\alpha}[\dim(X_\alpha)]$. We deduce a long exact sequence
		\begin{multline*}
			\cdots 
			\to \sH^{j-1}_c( f^{-1}(y) \cap \partial X_\alpha, \IC(X_\alpha, \underline{\bK})) 
			\to \sH^{j+\dim(X_\alpha)}_c( f^{-1}(y) \cap X_\alpha ; \bK) \\
			\to \sH^j_c( f^{-1}(y), \IC(X_\alpha, \underline{\bK}))
			\to \sH^{j}_c( f^{-1}(y) \cap \partial X_\alpha, \IC(X_\alpha, \underline{\bK})) \to \cdots
		\end{multline*}
		For any $\alpha' \in \mathcal{A}$ such that $X_{\alpha'} \subset \partial X_\alpha$,
		the complex $(j_{\alpha'})^* \IC(X_\alpha, \underline{\bK})$ is concentrated in degrees $\leq -\dim(X_{\alpha'})-1$, and by induction we have
		\[
		\dim(f^{-1}(y) \cap X_{\alpha'}) \leq \textstyle\frac{1}{2} (\dim(X_{\alpha'})-\dim(Y_\beta)).
		\]
		It follows that the space $\sH^{j}_c( f^{-1}(y) \cap \partial X_\alpha, \IC(X_\alpha, \underline{\bK}))$ vanishes unless
		\[
		j \leq \dim(X_{\alpha'})-\dim(Y_\beta) -\dim(X_{\alpha'})-1,
		\]
		i.e.~unless
		\[
		j \leq -\dim(Y_\beta) -1.
		\]
		(Here we use the fact that cohomology with compact supports of a scheme of finite type is concentrated in degrees at most twice the dimension of the scheme.)
		
		Now, assume for a contradiction that
		\[
		d:= \dim(f^{-1}(y) \cap X_\alpha) > \textstyle\frac{1}{2} (\dim(X_{\alpha})-\dim(Y_\beta)).
		\]
		Then we have an injection
		\[
		\sH^{2d}_c( f^{-1}(y) \cap X_\alpha ; \bK) \hookrightarrow
		\sH^{2d-\dim(X_\alpha)}_c( f^{-1}(y), \IC(X_\alpha, \underline{\bK}))
		\]
		since $\sH^{2d-\dim(X_\alpha)-1}_c( f^{-1}(y) \cap \partial X_\alpha, \IC(X_\alpha, \underline{\bK}))=0$. The left-hand side is nonzero, and hence so is $\sH^{2d-\dim(X_\alpha)}_c( f^{-1}(y), \IC(X_\alpha, \underline{\bK}))$, which implies that
		\[
		2d-\dim(X_\alpha) \leq -\dim(Y_\beta),
		\]
		a contradiction.
	\end{proof}
	
	\begin{rmk}
		\label{rmk:semismall}
		The principle expressed in Lemma~\ref{lem:pushforward-semismall} is mentioned in~\cite[Remark~4.5]{mv}, and was explained to the fourth named author by I.~Mirkovi{\'c} a long time ago. It justifies~\cite[Remark~1.6.5(2)]{br}, which explains that the stratified semismallness claim needed in the ``classical'' geometric Satake equivalence can be deduced from a claim about convolution of intersection cohomology complexes over $\Q_\ell$, which itself can be deduced from the results of~\cite{lusztig}.
	\end{rmk}
	

	\bibliography{biblio.bib}

\begin{thebibliography}{CvdHS23}

\bibitem[AGLR22]{aglr}
Johannes Ansch{\"u}tz, Ian Gleason, Jo{\~a}o Louren{\c c}o, and Timo Richarz.
\newblock On the $p$-adic theory of local models.
\newblock Preprint~\href{https://arxiv.org/abs/2201.01234}{arXiv:2201.01234},
  2022.

\bibitem[ALRR23]{alrr}
Pramod Achar, Jo{\~a}o Louren{\c c}o, Timo Richarz, and Simon Riche.
\newblock A modular ramified {S}atake equivalence, 2023.

\bibitem[AR]{ar-book}
Pramod Achar and Simon Riche.
\newblock Central sheaves on affine flag varieties.
\newblock Book in preparation, preliminary version available
  at~\url{https://lmbp.uca.fr/~riche/central.pdf}.

\bibitem[BBDG82]{bbd}
Alexander Be\u{\i}linson, Joseph Bernstein, Pierre Deligne, and Ofer Gabber.
\newblock Faisceaux pervers.
\newblock In {\em Analysis and topology on singular spaces, {I} ({L}uminy,
  1981)}, volume 100 of {\em Ast\'{e}risque}, pages 5--171. Soc. Math. France,
  Paris, 1982.

\bibitem[BD00]{beilinson-drinfeld}
Alexander Be{\u\i}linson and Vladimir Drinfeld.
\newblock Quantization of {H}itchin's integrable system and {H}ecke
  eigensheaves.
\newblock Preprint available
  at~\url{https://math.uchicago.edu/~drinfeld/langlands/QuantizationHitchin.pdf},
  2000.

\bibitem[BR18]{br}
Pierre Baumann and Simon Riche.
\newblock Notes on the geometric {S}atake equivalence.
\newblock In {\em Relative aspects in representation theory, {L}anglands
  functoriality and automorphic forms}, volume 2221 of {\em Lecture Notes in
  Math.}, pages 1--134. Springer, Cham, 2018.

\bibitem[BR23]{br-affgrass}
Patrick Bieker and Timo Richarz.
\newblock Normality of schubert varieties in affine grassmannians.
\newblock Preprint~\href{https://arxiv.org/abs/2312.11154}{arXiv:2312.11154},
  2023.

\bibitem[Bra03]{braden}
Tom Braden.
\newblock Hyperbolic localization of intersection cohomology.
\newblock {\em Transform. Groups}, 8:209--216, 2003.

\bibitem[Bru98]{brundan}
Jonathan Brundan.
\newblock Dense orbits and double cosets.
\newblock In {\em Algebraic groups and their representations ({C}ambridge,
  1997)}, volume 517 of {\em NATO Adv. Sci. Inst. Ser. C: Math. Phys. Sci.},
  pages 259--274. Kluwer Acad. Publ., Dordrecht, 1998.

\bibitem[BS68]{bs}
Armand Borel and Tony~A. Springer.
\newblock Rationality properties of linear algebraic groups. {II}.
\newblock {\em Tohoku Math. J. (2)}, 20:443--497, 1968.

\bibitem[BT65]{bot}
Armand Borel and Jacques Tits.
\newblock Groupes r\'{e}ductifs.
\newblock {\em Inst. Hautes \'{E}tudes Sci. Publ. Math.}, 27:55--150, 1965.

\bibitem[BT72]{bot2}
Armand Borel and Jacques Tits.
\newblock Compl\'{e}ments \`a l'article: ``{G}roupes r\'{e}ductifs''.
\newblock {\em Inst. Hautes \'{E}tudes Sci. Publ. Math.}, (41):253--276, 1972.

\bibitem[BT84]{BT84}
Fran{\c c}ois Bruhat and Jacques Tits.
\newblock Groupes r\'{e}ductifs sur un corps local. {II}. {S}ch\'{e}mas en
  groupes. {E}xistence d'une donn\'{e}e radicielle valu\'{e}e.
\newblock {\em Inst. Hautes \'{E}tudes Sci. Publ. Math.}, 60:197--376, 1984.

\bibitem[{\v{C}}es24]{cesnavicius}
K\k{e}stutis {\v{C}}esnavi\v{c}ius.
\newblock The affine {G}rassmannian as a presheaf quotient.
\newblock Preprint~\href{https://arxiv.org/abs/2401.04314}{arXiv:2401.04314},
  2024.

\bibitem[CGP15]{CGP10}
Brian Conrad, Ofer Gabber, and Gopal Prasad.
\newblock {\em Pseudo-reductive groups}, volume~26 of {\em New Mathematical
  Monographs}.
\newblock Cambridge University Press, Cambridge, second edition, 2015.

\bibitem[CvdHS23]{cvdhs}
Robert Cass, Thibaud van~den Hove, and Jakob Scholbach.
\newblock The geometric satake equivalence for integral motives.
\newblock Preprint~\href{https://arxiv.org/abs/2211.04832}{2211.04832}, 2023.

\bibitem[DG70]{dg}
Michel Demazure and Pierre Gabriel.
\newblock {\em Groupes alg\'{e}briques. {T}ome {I}: {G}\'{e}om\'{e}trie
  alg\'{e}brique, g\'{e}n\'{e}ralit\'{e}s, groupes commutatifs}.
\newblock Masson \& Cie, \'{E}diteurs, Paris; North-Holland Publishing Co.,
  Amsterdam, 1970.
\newblock Avec un appendice {\it Corps de classes local} par Michiel
  Hazewinkel.

\bibitem[DG11]{sga33}
Michel Demazure and Alexandre Grothendieck.
\newblock {\em Sch{\'e}mas en groupes (SGA 3). Tome III. Structure des
  sch{\'e}mas en groupes r{\'e}ductifs}, volume~8 of {\em Documents
  Math\'ematiques}.
\newblock S\'eminaire de G\'eom{\'e}trie Alg\'ebrique du Bois Marie 1962--64. A
  seminar directed by M. Demazure and A. Grothendieck with the collaboration of
  M. Artin, J.-E. Bertin, P. Gabriel, M. Raynaud and J-P. Serre. Revised and
  annotated edition of the 1970 French original. Edited by P.~Gille and
  P.~Polo. Soci\'et\'e Math\'ematique de France, Paris, 2011.

\bibitem[Dri13]{Dr13}
Vladimir Drinfeld.
\newblock On algebraic spaces with an action of $\mathbb{G}_m$.
\newblock Preprint~\href{https://arxiv.org/abs/1308.2604}{arXiv:1308.2604},
  2013.

\bibitem[Fal03]{faltings-loops}
Gerd Faltings.
\newblock Algebraic loop groups and moduli spaces of bundles.
\newblock {\em J. Eur. Math. Soc. (JEMS)}, 5(1):41--68, 2003.

\bibitem[FHLR22]{FHLR}
Najmuddin Fakhruddin, Thomas Haines, Jo{\~a}o Louren{\c c}o, and Timo Richarz.
\newblock Singularities of local models.
\newblock Preprint~\href{https://arxiv.org/abs/2208.12072}{arXiv:2208.12072},
  2022.

\bibitem[FS21]{fs}
Laurent Fargues and Peter Scholze.
\newblock Geometrization of the local {L}anglands correspondence.
\newblock Preprint~\href{https://arxiv.org/abs/2102.13459}{arXiv:2102.13459},
  2021.

\bibitem[Gai01]{gaitsgory}
Dennis Gaitsgory.
\newblock Construction of central elements in the affine {H}ecke algebra via
  nearby cycles.
\newblock {\em Invent. Math.}, 144(2):253--280, 2001.

\bibitem[Gin00]{ginzburg}
Victor Ginzburg.
\newblock Perverse sheaves on a loop group and {L}anglands' duality.
\newblock
  Preprint~\href{https://arxiv.org/abs/alg-geom/9511007}{arXiv:alg-geom/9511007},
  2000.

\bibitem[G{\"o}r10]{goertz}
Ulrich G{\"o}rtz.
\newblock Affine springer fibers and affine deligne--lusztig varieties.
\newblock In Alexander Schmitt, editor, {\em Affine flag manifolds and
  principal bundles}, Trends in Mathematics, pages 1--50.
  Birkh\"{a}user/Springer Basel AG, Basel, 2010.

\bibitem[GW20]{goertz-wedhorn}
Ulrich G\"{o}rtz and Torsten Wedhorn.
\newblock {\em Algebraic geometry {I}. {S}chemes---with examples and
  exercises}.
\newblock Springer Studium Mathematik---Master. Springer Spektrum, Wiesbaden,
  2020.
\newblock Second edition.

\bibitem[Hai18]{Haines:Dualities}
Thomas~J. Haines.
\newblock Dualities for root systems with automorphisms and applications to
  non-split groups.
\newblock {\em Represent. Theory}, 22:1--26, 2018.

\bibitem[HLR18]{HLR}
Thomas~J. Haines, Jo{\~a}o Louren{\c c}o, and Timo Richarz.
\newblock On the normality of {S}chubert varieties: remaining cases in positive
  characteristic.
\newblock Accepted in \textit{Ann. Sci. \'{E}cole Norm. Sup\'{e}r. (4)}.
  Preprint~\href{https://arxiv.org/abs/1806.11001}{arXiv:1806.11001}, 2018.

\bibitem[HR20]{HainesRicharz:Smoothness}
Thomas~J. Haines and Timo Richarz.
\newblock Smoothness of {S}chubert varieties in twisted affine {G}rassmannians.
\newblock {\em Duke Math. J.}, 169(17):3223--3260, 2020.

\bibitem[HR21]{HainesRicharz_TestFunctions}
Thomas~J. Haines and Timo Richarz.
\newblock The test function conjecture for parahoric local models.
\newblock {\em J. Amer. Math. Soc.}, 34(1):135--218, 2021.

\bibitem[HR23]{HainesRicharz:CM}
Thomas~J. Haines and Timo Richarz.
\newblock Normality and {C}ohen-{M}acaulayness of parahoric local models.
\newblock {\em J. Eur. Math. Soc. (JEMS)}, 25(2):703--729, 2023.

\bibitem[HS23]{hs}
David Hansen and Peter Scholze.
\newblock Relative perversity.
\newblock {\em Comm. Amer. Math. Soc.}, 3:631--668, 2023.

\bibitem[JMW14]{jmw}
Daniel Juteau, Carl Mautner, and Geordie Williamson.
\newblock Parity sheaves.
\newblock {\em J. Amer. Math. Soc.}, 27(4):1169--1212, 2014.

\bibitem[Kol16]{kollar}
J\'{a}nos Koll\'{a}r.
\newblock Variants of normality for {N}oetherian schemes.
\newblock {\em Pure Appl. Math. Q.}, 12(1):1--31, 2016.

\bibitem[LO08a]{laszlo-olsson}
Yves Laszlo and Martin Olsson.
\newblock The six operations for sheaves on {A}rtin stacks. {I}. {F}inite
  coefficients.
\newblock {\em Publ. Math. Inst. Hautes \'{E}tudes Sci.}, (107):109--168, 2008.

\bibitem[LO08b]{laszlo-olsson-2}
Yves Laszlo and Martin Olsson.
\newblock The six operations for sheaves on {A}rtin stacks. {II}. {A}dic
  coefficients.
\newblock {\em Publ. Math. Inst. Hautes \'{E}tudes Sci.}, (107):169--210, 2008.

\bibitem[LO09]{laszlo-olsson-perv}
Yves Laszlo and Martin Olsson.
\newblock Perverse {$t$}-structure on {A}rtin stacks.
\newblock {\em Math. Z.}, 261(4):737--748, 2009.

\bibitem[Lou22]{lourenco-BT}
Jo{\~a}o Louren{\c c}o.
\newblock Th{\'e}orie de {B}ruhat-{T}its pour les groupes quasi-r{\'e}ductifs.
\newblock {\em J. Inst. Math. Jussieu}, 21(4):1331--1362, 2022.

\bibitem[Lou23]{lourenco-normality}
João Lourenço.
\newblock Distributions and normality theorems.
\newblock Preprint~\href{https://arxiv.org/abs/2312.17121}{arXiv:2312.17121},
  2023.

\bibitem[Lus83]{lusztig}
George Lusztig.
\newblock Singularities, character formulas, and a {$q$}-analog of weight
  multiplicities.
\newblock In {\em Analysis and topology on singular spaces, {II}, {III}
  ({L}uminy, 1981)}, volume 101 of {\em Ast\'{e}risque}, pages 208--229. Soc.
  Math. France, Paris, 1983.

\bibitem[LZ17a]{liu-zheng-2}
Yifeng Liu and Weizhe Zheng.
\newblock Enhanced adic formalism and perverse t-structures for higher {A}rtin
  stacks.
\newblock Preprint~\href{https://arxiv.org/abs/1404.1128}{1404.1128}, 2017.

\bibitem[LZ17b]{liu-zheng}
Yifeng Liu and Weizhe Zheng.
\newblock Enhanced six operations and base change theorem for higher {A}rtin
  stacks.
\newblock Preprint~\href{https://arxiv.org/abs/1211.5948}{1211.5948}, 2017.

\bibitem[MV87]{mirollo-vilonen}
Renato Mirollo and Kari Vilonen.
\newblock Bernstein-{G}elfand-{G}elfand reciprocity on perverse sheaves.
\newblock {\em Ann. Sci. \'{E}cole Norm. Sup. (4)}, 20(3):311--323, 1987.

\bibitem[MV07]{mv}
Ivan Mirkovi\'{c} and Kari Vilonen.
\newblock Geometric {L}anglands duality and representations of algebraic groups
  over commutative rings.
\newblock {\em Ann. of Math. (2)}, 166(1):95--143, 2007.

\bibitem[PR08]{pr}
George Pappas and Michael Rapoport.
\newblock Twisted loop groups and their affine flag varieties.
\newblock {\em Adv. Math.}, 219(1):118--198, 2008.
\newblock With an appendix by Thomas J. Haines and Michael Rapoport.

\bibitem[PY06]{py}
Gopal Prasad and Jiu-Kang Yu.
\newblock On quasi-reductive group schemes.
\newblock {\em J. Algebraic Geom.}, 15(3):507--549, 2006.
\newblock With an appendix by Brian Conrad.

\bibitem[Ric]{riche-hab}
Simon Riche.
\newblock Geometric representation theory in positive characteristic.
\newblock \url{https://theses.hal.science/tel-01431526}.

\bibitem[Ric13]{richarz-schubert}
Timo Richarz.
\newblock Schubert varieties in twisted affine flag varieties and local models.
\newblock {\em J. Algebra}, 375:121--147, 2013.

\bibitem[Ric16]{richarz}
Timo Richarz.
\newblock Affine {G}rassmannians and geometric {S}atake equivalences.
\newblock {\em Int. Math. Res. Not. IMRN}, 12:3717--3767, 2016.

\bibitem[Ric19]{richarz-Gm}
Timo Richarz.
\newblock Spaces with {$\mathbb G_m$}-action, hyperbolic localization and
  nearby cycles.
\newblock {\em J. Algebraic Geom.}, 28(2):251--289, 2019.

\bibitem[Ric20]{richarz-basics}
Timo Richarz.
\newblock Basics on affine {G}rassmannians.
\newblock {\em
  \url{https://timo-richarz.com/wp-content/uploads/2020/02/BoAG_02.pdf}}, 2020.

\bibitem[Ric21]{richarz-erratum}
Timo Richarz.
\newblock Erratum to ``{A}ffine {G}rassmannians and geometric {S}atake
  equivalences''.
\newblock {\em Int. Math. Res. Not. IMRN}, 17:13602--13608, 2021.

\bibitem[RS20]{rs-intersection}
Timo Richarz and Jakob Scholbach.
\newblock The intersection motive of the moduli stack of shtukas.
\newblock {\em Forum Math. Sigma}, 8:Paper No. e8, 99, 2020.

\bibitem[RSW14]{rsw}
Simon Riche, Wolfgang Soergel, and Geordie Williamson.
\newblock Modular {K}oszul duality.
\newblock {\em Compos. Math.}, 150(2):273--332, 2014.

\bibitem[Sch92]{schauenburg}
Peter Schauenburg.
\newblock {\em Tannaka duality for arbitrary {H}opf algebras}, volume~66 of
  {\em Algebra Berichte [Algebra Reports]}.
\newblock Verlag Reinhard Fischer, Munich, 1992.

\bibitem[{Sta}22]{stacks-project}
The {Stacks Project Authors}.
\newblock Stacks project.
\newblock \url{http://stacks.math.columbia.edu}, 2022.

\bibitem[Ste65]{steinberg-reg}
Robert Steinberg.
\newblock Regular elements of semisimple algebraic groups.
\newblock {\em Inst. Hautes \'{E}tudes Sci. Publ. Math.}, 25:49--80, 1965.

\bibitem[vdK01]{vdk}
Wilberd van~der Kallen.
\newblock Steinberg modules and {D}onkin pairs.
\newblock {\em Transform. Groups}, 6(1):87--98, 2001.

\bibitem[Vil94]{vilonen}
Kari Vilonen.
\newblock Perverse sheaves and finite-dimensional algebras.
\newblock {\em Trans. Amer. Math. Soc.}, 341(2):665--676, 1994.

\bibitem[Yu22]{yu-integral}
Jize Yu.
\newblock The integral geometric {S}atake equivalence in mixed characteristic.
\newblock {\em Represent. Theory}, 26:874--905, 2022.

\bibitem[YZ11]{yun-zhu}
Zhiwei Yun and Xinwen Zhu.
\newblock Integral homology of loop groups via {L}anglands dual groups.
\newblock {\em Represent. Theory}, 15:347--369, 2011.

\bibitem[Zhe15]{zheng}
Weizhe Zheng.
\newblock Six operations and {L}efschetz-{V}erdier formula for
  {D}eligne-{M}umford stacks.
\newblock {\em Sci. China Math.}, 58(3):565--632, 2015.

\bibitem[Zhu15]{zhu}
Xinwen Zhu.
\newblock The geometric {S}atake correspondence for ramified groups.
\newblock {\em Ann. Sci. \'{E}c. Norm. Sup\'{e}r. (4)}, 48(2):409--451, 2015.

\bibitem[Zhu17]{zhu-notes}
Xinwen Zhu.
\newblock An introduction to affine {G}rassmannians and the geometric {S}atake
  equivalence.
\newblock In {\em Geometry of moduli spaces and representation theory},
  volume~24 of {\em IAS/Park City Math. Ser.}, pages 59--154. Amer. Math. Soc.,
  Providence, RI, 2017.

\end{thebibliography}
	\bibliographystyle{alpha}

\end{document}